\documentclass[10.5pt]{article}

\usepackage{fullpage, amsmath}
\usepackage{amsthm,amsfonts,url,amssymb}
\usepackage{mathrsfs}
\usepackage{amssymb,amscd}
\usepackage{color}
\usepackage{chemarrow}
\usepackage{pictexwd,dcpic}
\usepackage[textwidth=18cm,centering]{geometry}
\usepackage{extarrows}
\usepackage[colorlinks,linkcolor=red,anchorcolor=green,citecolor=blue]{hyperref}
\usepackage{enumitem}
\usepackage{lipsum}
\usepackage[all,cmtip]{xy}
\usepackage{leftidx}
\usepackage[version=3]{mhchem}
\usepackage{makeidx}
\usepackage{bm} 
\usepackage{stmaryrd}
\usepackage{dashrule}

\makeindex
\newtheorem{thm}{Theorem}[section]
\newtheorem{cor}[thm]{Corollary}
\newtheorem{lem}[thm]{Lemma}
\newtheorem{pro}[thm]{Proposition}
\newtheorem{dfn}[thm]{Definition}

\newtheorem{rmk}[thm]{Remark}

\newtheorem{exmp}[thm]{Example}

\newtheorem{ques}[thm]{Question}

\newcommand{\hooklongrightarrow}{\lhook\joinrel\longrightarrow}
\newcommand{\twoheadlongrightarrow}{\relbar\joinrel\twoheadrightarrow}

\newcommand{\alg}{{\rm alg}}

\newcommand{\ana}{{\rm an}}
\newcommand{\adm}{{\rm ad}}

\newcommand{\alge}{{\rm alg}}

\newcommand{\ccyc}{{\epsilon}}

\newcommand{\cEo}{{\rm \overline{\cE}}}

\newcommand{\dr}{{\rm dR}}
\newcommand{\Dpik}{{\mathbf{D}}}

\newcommand{\EndO}{{\rm End}}
\newcommand{\ext}{{\rm Ext}}

\newcommand{\fil}{{\rm Fil}}

\newcommand{\gal}{{\rm Gal}}
\newcommand{\GLN}{{\rm GL}}

\newcommand{\gr}{{\rm gr}}

\newcommand{\homo}{{\rm Hom}}
\newcommand{\hH}{{\mathrm{H}}}
\newcommand{\hpi}{{{\mathbf{h}}}}

\newcommand{\ind}{{\rm Ind}}
\newcommand{\im}{{\rm Im}}
\newcommand{\lM}{{\overline{M}}}
\newcommand{\lL}{{\overline{L}}}

\newcommand{\leg}{{\rm lg}}

\newcommand{\lalg}{{\rm lalg}}
\newcommand{\lra}{\longrightarrow}

\newcommand{\op}{{\overline{{\mathbf{P}}}}}

\newcommand{\ob}{{\overline{{\mathbf{B}}}}}
\newcommand{\on}{{\overline{{\mathbf{N}}}}}

\newcommand{\PGLN}{{\rm PGL}}

\newcommand{\pcr}{{\rm pcr}}

\newcommand{\pr}{{\rm pr}}
\newcommand{\pst}{{\rm pst}}

\newcommand{\Rep}{\mathrm{Rep}}

\newcommand{\rig}{{\rm rig}}
\newcommand{\ra}{{\longrightarrow}}
\newcommand{\rec}{{\rm rec}}

\newcommand{\res}{{\rm Res}}

\newcommand{\ssE}{{\mathcal{S}}}
\newcommand{\sEo}{{\rm \overline{\ssE}}}

\newcommand{\st}{{\rm St}}
\newcommand{\soc}{{\rm soc}}
\newcommand{\sm}{{\rm sm}}

\newcommand{\ur}{{\rm ur}}
\newcommand{\unr}{{\rm unr}}

\newcommand{\univ}{{\rm univ}}
\newcommand{\ul}{\underline}

\newcommand{\valua}{\mathrm{val}}

\newcommand{\BA}{{\mathbb{A}}}

\newcommand{\BC}{{\mathbb{C}}}

\newcommand{\BE}{{\mathbb{E}}}

\newcommand{\BG}{{\mathbb{G}}}
\newcommand{\BH}{{\mathbb{H}}}
\newcommand{\BI}{{\mathbb{I}}}

\newcommand{\BQ}{{\mathbb{Q}}}
\newcommand{\BR}{{\mathbb{R}}}

\newcommand{\BZ}{{\mathbb{Z}}}

\newcommand{\bB}{{\mathbf{B}}}
\newcommand{\bC}{{\mathbf{C}}}
\newcommand{\bD}{{\mathbf{D}}}
\newcommand{\bE}{{\mathbf{E}}}

\newcommand{\bG}{{\mathbf{G}}}

\newcommand{\bI}{{\mathbf{I}}}

\newcommand{\bL}{{\mathbf{L}}}

\newcommand{\bN}{{\mathbf{N}}}

\newcommand{\bP}{{\mathbf{P}}}
\newcommand{\bQ}{{\mathbf{Q}}}

\newcommand{\bT}{{\mathbf{T}}}

\newcommand{\bZ}{{\mathbf{Z}}}

\newcommand{\bh}{{\mathbf{h}}}

\newcommand{\cL}{\mathcal L}
\newcommand{\co}{\mathcal O}
\newcommand{\cA}{\mathcal A}
\newcommand{\cR}{\mathcal R}

\newcommand{\cC}{\mathcal C}

\newcommand{\cD}{\mathcal D}
\newcommand{\cI}{\mathcal I}
\newcommand{\cW}{\mathcal W}

\newcommand{\cM}{\mathcal M}
\newcommand{\cF}{\mathcal F}
\newcommand{\cV}{\mathcal V}
\newcommand{\cE}{\mathcal E}

\newcommand{\cO}{\mathcal O}

\newcommand{\fb}{{\mathfrak{b}}}

\newcommand{\fd}{{\mathfrak{d}}}

\newcommand{\fg}{{\mathfrak{g}}}
\newcommand{\fh}{{\mathfrak{h}}}

\newcommand{\fl}{{\mathfrak{l}}}
\newcommand{\fm}{{\mathfrak{m}}}
\newcommand{\fn}{{\mathfrak{n}}}

\newcommand{\fp}{{\mathfrak{p}}}

\newcommand{\ft}{{\mathfrak{t}}}

\newcommand{\fz}{{\mathfrak{z}}}

\newcommand{\sE}{\mathscr E}

\newcommand{\sW}{\mathscr W}
\newcommand{\sL}{\mathscr L}

\begin{document}	
	
	\title{\textbf{\textsc{Extensions of locally analytic generalized Steinberg representations for $\mathrm{GL}_{n}$}}}
	
	\author{Yiqin He
		\thanks{Morningside Center of Mathematics, Chinese Academy of Sciences,\;No.\;55, Zhongguancun East Road, Haidian District, Beijing 100190, P. R. China,\;E-mail address:\texttt{\;heyiqin@amss.ac.cn}
		}}
	
	\date{}
	\maketitle

\begin{abstract}
Let $L$ be a finite extension of $\bQ_p$.\;We study locally $\bQ_p$-analytic generalized Steinberg representations of $\GLN_n(L)$ associated with a general Zelevinsky segment and compute certain $\ext^1$-groups involving these representations.\;These extension groups give rise to parabolic Breuil simple $\sL$-invariants on the automorphic side of a conjectural $p$-adic local Langlands correspondence for potentially semistable $p$-adic Galois representations in the non-trianguline case.\;
\end{abstract}

\maketitle
	
{\hypersetup{linkcolor=black}\tableofcontents}

\setlength{\baselineskip}{13pt}


	
	\numberwithin{equation}{section}
	\section{Introduction}
	Let $L$ (resp.\;$E$) be a finite extension of $\bQ_p$.\;Suppose that $E$ is sufficiently large to contain the images of all embeddings of $L$ into $\overline{\bQ}_p$.\;Put $\Sigma_L:=\{\sigma:L\hookrightarrow\overline{\bQ}_p\}$,\;$d_L:=|\Sigma_L|$ and $G:=\GLN_n(L)$.\;Let $|\cdot|_L=q_L^{-\mathrm{val}_L(x)}$,\;where $q_L$ is the cardinality of the residue field $k_L$ of $L$ and $\mathrm{val}_L(\cdot)$ is a valuation on $L$ normalized by sending uniformizers of $L$ to $1$.\;

	Zelevinsky segments play an important role in the Bernstein--Zelevinsky classification of irreducible smooth representations of $G$.\;It is therefore natural to seek a corresponding theory in the locally analytic setting.\;In this paper, we study locally $\bQ_p$-analytic generalized Steinberg representations of $G$ attached to a general Zelevinsky segment and compute certain $\ext^1$-groups between them.\;These groups arise on the automorphic side of the \textit{$p$-adic local Langlands program} in the non-trianguline case and give rise to the parabolic Breuil simple $\sL$-invariants of \cite{2022parabolivinv}.\;Previously, such computations were known only for representations induced from a Borel subgroup.\;In the Borel case, Yiwen Ding \cite[Section 2]{2019DINGSimple} computes the relevant $\ext^1$-groups, which are also related to the automorphic $\sL$-invariants defined by Lennart Gehrmann \cite{gehrmann2019automorphic}.\;Our results provide a parabolic version of Ding's work; a corresponding parabolic theory of automorphic $\sL$-invariants is also expected (see some expectation in Section \ref{paraGehr}).\;We now sketch the main results.\;
	
	Let $\bT$ (resp.,\;$\bB$) be the torus of diagonal matrices (resp., the Borel subgroup of upper triangular matrices) in $\GLN_n$.\;Let $\Delta:=\{1,\cdots,n-1\}$ be the set of simple roots of $\GLN_n$.\;For $I\subseteq \Delta$, let $\bL_I\supseteq \bT$ (resp., $\bP_I\supseteq \bL_I\bB$) be the associated standard Levi subgroup (resp., parabolic subgroup), and let $\op_I$ be the parabolic subgroup opposite to $\bP_I$.\;Fix positive integers $k$ and $r$ such that $n=kr$, and consider the partition $\beta=(r,\ldots,r)$ of $n$ with $k$ parts.\;Put $\Delta_{\beta}:=\{r,2r,\ldots,(k-1)r\}\subseteq \Delta$ and $\Delta^{\beta}:=\Delta\backslash \Delta_{\beta}$.\;Define
	$\bL_{\beta,I}:=\bL_{\Delta^{\beta}\cup I}$ and $\op_{\beta,I}:=\op_{\Delta^{\beta}\cup I}$.\;Thus $\bL_{\beta}:=\bL_{\beta,\emptyset}=\GLN_r\times\cdots\times\GLN_r\subseteq \op_{\beta}:=\op_{\beta,\emptyset}$.\;
	
	Let $\pi$ be a fixed irreducible  cuspidal smooth representation of $\GLN_r(L)$ over $E$.\;Following \cite{av1980induced2},\;consider the following Zelevinsky segment  of length $k$:
	\begin{equation}\label{introBZsegment}
		\bm\beta:=\{\pi\nu^{k-1},\;\cdots,\;\pi\nu,\pi\},
	\end{equation}
	where $\nu$ is the character of $\GLN_r(L)$:\;$g\mapsto |\det(g)|_L$.\;Associated with $\bm\beta$ is the following irreducible cuspidal smooth representation of $\bL_{\beta}(L)$ over $E$ (fix a choice of $q_L^{1/2}$ in $\overline{\bQ}_p$ and assume that $q_L^{1/2}\in E$):
	\begin{equation}
		\pi_{\beta}:=\big(\boxtimes_{i=1}^{k}\pi\nu^{k-i}\big)\otimes_E \delta_{\op_{\beta}}^{1/2}=\boxtimes_{i=1}^{k}\pi \nu^{-\frac{r}{2}(k-2i+1)+k-i},\;
	\end{equation}
where $\delta_{\op_{\beta}}$ is the modulus character of $\op_{\beta}$.\;For $I\subseteq \Delta_{\beta}$, put (the following parabolic induction has a unique irreducible subrepresentation by Bernstein-Zelevinsky theory \cite[Proposition 2.10]{av1980induced2})
\[\pi_{\beta,I}=\soc_{\bL_{\beta,I}(L)}\Big(\mathrm{Ind}_{\op_{\beta}(L)\cap\bL_{\beta,I}(L)}^{\bL_{\beta,I}(L)}\pi_{\beta}\Big)^\infty.\]
Let $\ul{\lambda}$ be a dominant weight of $\mathrm{Res}_{L/\bQ_p}\GLN_n$ with respect to $\mathrm{Res}_{L/\bQ_p}\bB$.\;Let $L(\ul{\lambda})$ (resp.,\;$L(\ul{\lambda})_{\Delta^{\beta}\cup I}$) be the unique irreducible $\bQ_p$-algebraic representation of $G$ (resp.,\;$\bL_{\beta,I}(L)$) with highest weight $\ul{\lambda}$.\;For $I\subseteq \Delta_{\beta}$, put
	\begin{equation*}
		\begin{aligned}
			&i_{I}(\bm\beta,\ul{\lambda})=\Big(\mathrm{Ind}_{\op_{\beta,I}(L)}^G\pi_{\beta,I}\Big)^\infty\otimes L(\ul{\lambda}),\quad v^{\lalg}_{I}(\bm\beta,\ul{\lambda})=i_{I}(\bm\beta,\ul{\lambda})\big/\sum_{I\subsetneq J\subseteq \Delta_{\beta}}i_{J}(\bm\beta,\ul{\lambda}).\;
		\end{aligned}
	\end{equation*}
	Put $v^{\infty}_{I}(\bm\beta):=v^{\lalg}_{I}(\bm\beta,\ul{0})$.\;Then $\{v^{\lalg}_{I}(\bm\beta,\ul{\lambda})\}_{I\subseteq \Delta_{\beta}}$ (resp.,\;$\{v^{\infty}_{I}(\bm\beta)\}_{I\subseteq \Delta_{\beta}}$) is the family of locally $\bQ_p$-algebraic (resp.\;smooth) generalized Steinberg representations associated with the Zelevinsky segment $\bm\beta$ and weight $\ul{\lambda}$.\;We show that they are precisely the Jordan--H\"older factors of $i_{\emptyset}(\bm\beta,\ul{\lambda})$ (resp.,\;$i_{\emptyset}(\bm\beta,\ul{0})$).\;In particular, $\st_n^{\lalg}(\bm\beta,\ul{\lambda}):=v^{\lalg}_{\emptyset}(\bm\beta,\ul{\lambda})$ (resp., $\st_n^{\infty}(\bm\beta):=\st_n^{\infty}(\bm\beta)$) is the locally $\bQ_p$-algebraic (resp., smooth) Steinberg representation.\;The Jordan--H\"older factors of $i_{\emptyset}(\bm\beta,\ul{\lambda})$ are classified by orientations of a graph in \cite[Section 2.2, Theorem]{av1980induced2}; we also compare our definition with this earlier description.\;
	
	The locally $\bQ_p$-analytic analogue is defined as follows.\;For $I\subseteq \Delta_{\beta}$,\;put
	\begin{equation*}
		\begin{aligned}
			\bI_I(\bm\beta,\ul{\lambda})=\Big(\mathrm{Ind}_{\op_{\beta,I}(L)}^G\pi_{\beta,I}\otimes L(\ul{\lambda})_{\Delta^{\beta}\cup I}\Big)^{\bQ_p-\ana},\quad v^{\ana}_{I}(\bm\beta,\ul{\lambda})=\bI_{I}(\bm\beta,\ul{\lambda})\big/\sum_{I\subsetneq J\subseteq \Delta_{\beta}}\bI_{J}(\bm\beta,\ul{\lambda}),
		\end{aligned}
	\end{equation*}
	We call $\{v^{\ana}_{I}(\bm\beta,\ul{\lambda})\}_{I\subseteq \Delta_{\beta}}$ the locally $\bQ_p$-analytic generalized Steinberg representations associated with $\bm\beta$ and $\ul{\lambda}$.\;Denote by $\st_n^{\ana}(\bm\beta,\ul{\lambda}):=v^{\ana}_{\emptyset}(\bm\beta,\ul{\lambda})$ the locally $\bQ_p$-analytic Steinberg representation.\;The locally $\bQ_p$-algebraic vectors in
	$v^{\ana}_{I}(\bm\beta,\ul{\lambda})$ form $v^{\lalg}_{I}(\bm\beta,\ul{\lambda})$.\;The explicit structure of $\bI_I(\bm\beta,\ul{\lambda})$ is well understood through the work of Orlik and Strauch \cite[Theorem]{orlik2015jordan}.\;
	
	By generalizing the results of \cite{orlik2014jordan},\;we establish locally $\bQ_p$-analytic (resp., locally $\bQ_p$-algebraic) Tits complexes that realize $v^{\ana}_{I}(\bm\beta,\ul{\lambda})$ (resp.,\;$v^{\lalg}_{I}(\bm\beta,\ul{\lambda})$) by complexes whose terms are locally $\bQ_p$-analytic (resp., locally $\bQ_p$-algebraic) principal series $\{\bI_I(\bm\beta,\ul{\lambda})\}_{I\subseteq \Delta_{\beta}}$ (resp., $\{i_I(\bm\beta,\ul{\lambda})\}_{I\subseteq \Delta_{\beta}}$);\;see Propositions \ref{smoothexact} and \ref{analyticexact}.\;Their first application is to compute the multiplicities of Jordan H\"{o}lder factors of $v^{\ana}_{I}(\bm\beta,\ul{\lambda})$;\;see Proposition \ref{JHanastein1}.\;Their second application yields the main results of this paper:\;the computation of extension groups between locally $\bQ_p$-analytic generalized Steinberg representations.\;We prove in Theorem \ref{analyticExt3.3} (see Section \ref{smanaextgps} for the precise definition of these $\ext^1$-groups):
	\begin{thm}\label{intromain}Let $J\subseteq I\subseteq\Delta_{\beta}$ and $|I|=|J|+1$.\;There is an isomorphism of $E$-vector spaces 
		\begin{equation}\label{introextiso}
			\homo(L^\times,E)\xrightarrow{\sim } \ext^1_G\big(v^{\lalg}_I(\bm\beta,\ul{\lambda}),v^{\ana}_J(\bm\beta,\ul{\lambda})\big).
		\end{equation}
		where $\homo(L^\times,E)$ denotes the $(d_L+1)$-dimensional $E$-vector space of $E$-valued additive characters on $L^\times$.
	\end{thm}
	In Remark \ref{rmkforhigherext},\;we explain why we only concern $\ext^1$-groups and describe the limitations of (higher) extension groups along $\{v^{\ana}_{I}(\bm\beta,\ul{\lambda})\}_{I\subseteq \Delta_{\beta}}$.\;
	
	Moreover, for $\psi\in\homo(L^{\times},E)$, we give an explicit construction of the locally $\bQ_p$-analytic extension of $v^{\lalg}_{I}(\bm\beta,\ul{\lambda})$ by $v^{\ana}_J(\bm\beta,\ul{\lambda})$ attached to $\psi$ via (\ref{introextiso}).\;In particular, for $1\leq i\leq k-1$ (so that $ir\in \Delta_{\beta}$), setting $(I,J)=(\{ir\},\emptyset)$ gives an isomorphism of $E$-vector spaces
	\begin{equation}\label{introsteinbergext}
		\homo(L^\times,E)\xrightarrow{\sim }\ext^1_{G}\big(v_{ir}^{\lalg}(\bm\beta,\ul{\lambda}), \st_n^{\ana}(\bm\beta,\ul{\lambda})\big).
	\end{equation}
		We find a subrepresentation $\Sigma(\bm\beta,\ul{\lambda})$ of $\st_n^{\ana}(\bm\beta,\ul{\lambda})$ (if $\soc_G\st_n^{\ana}(\bm\beta,\ul{\lambda})=\st_n^{\lalg}(\bm\beta,\ul{\lambda})$,\;then $\Sigma(\bm\beta,\ul{\lambda})$ is the first two layers of the socle filtration of $\st_n^{\ana}(\bm\beta,\ul{\lambda})$); their explicit structure is described by Orlik--Strauch \cite[Theorem]{orlik2015jordan}.\;There are $k-1$ branches $\{\Sigma_{ir}(\bm\beta,\ul{\lambda})\}_{1\leq i\leq k-1}$ of $\Sigma(\bm\beta,\ul{\lambda})$ such that $\Sigma(\bm\beta,\ul{\lambda})$ is equal to the amalgamation $\bigoplus^{1\leq i\leq k-1}_{\st_n^{\lalg}(\bm\beta,\ul{\lambda})} \Sigma_{ir}(\bm\beta,\ul{\lambda})$.\;We have natural isomorphisms (see Proposition \ref{lastextensions}) of $E$-vector spaces
	\begin{equation}\label{Intro:equ: BrLi}
		\homo(L^\times,E)\xrightarrow{\sim}  \ext^1_G\big(v_{ir}^{\lalg}(\bm\beta,\ul{\lambda}), \Sigma_{ir}(\bm\beta,\ul{\lambda})\big)\xrightarrow{\sim}  \ext^1_G\big(v_{ir}^{\lalg}(\bm\beta,\ul{\lambda}), \Sigma(\bm\beta,\ul{\lambda})\big).
	\end{equation}


	\begin{rmk} 
Our locally $\bQ_p$-analytic representations arise naturally on the automorphic side of the $p$-adic Langlands program.\;In \cite{2022parabolivinv},\;Theorem \ref{intromain} is called \textit{parabolic Breuil simple $\sL$-invariants}.\;Let  $\gal_L$ be the absolute Galois group of $L$.\;When $r>1$ (resp.,\;$r=1$),\;for a potentially semistable (resp.,\;semistable) $p$-adic Galois representation $\rho_L:\gal_L\rightarrow\GLN_n(E)$  admitting a \textit{\underline{non-critical} special $\bP_{\beta}$-parabolic filtration} (resp,\;non-critical special triangulation in \cite{2019DINGSimple}),\;the author's work \cite{2022parabolivinv} (resp.,\;Ding's work \cite{2019DINGSimple}) defines parabolic Fontaine--Mazur simple $\sL$-invariants $\sL(\rho_L)=\prod_{i=1}^{k-1}\sL(\rho_L)_{ir}\subseteq \prod_{i=1}^{k-1}\homo(L^\times,E)$,\;which encode certain parameters in the admissible Hodge filtration corresponding to the simple roots in $\Delta_{\beta}$.\;The parabolic simple Breuil $\sL$-invariants contain counterparts of $\sL(\rho_L)$ using the following rule (we usually use $\bullet-\bullet$ to denote an extension of two  $G$-representations,\;where the first (resp., second) object is the subobject (resp., quotient)):
\[\rho_L \rightsquigarrow \sL(\rho_L)\rightsquigarrow \Sigma(\bm\beta,\ul{\lambda},\sL(\rho_L)):=
	\bigoplus^{1\leq i\leq k-1}_{\st_n^{\lalg}(\bm\beta,\ul{\lambda})}\Big[\Sigma_{ir}(\bm\beta,\ul{\lambda})-v_{ir}^{\lalg}(\bm\beta,\ul{\lambda})\otimes_E\big(\sL(\rho_L)_{ir}\big)^\vee\Big].\]
Moreover, using the Breuil-Ding's 
 Bernstein eigenvarieties \cite{Ding2021},\;we establish $p$-adic local--global compatibility results in \cite{2022parabolivinv}.\;When $\rho_L$ arises from a patched automorphic representation of $\bG(\BA_{F^+})$, where $\bG$ is a definite unitary group over a totally real field $F^+$ that is compact at the infinite places and isomorphic to $\GLN_n$ at the $p$-adic places, we prove under mild hypotheses that $\Sigma(\bm\beta,\ul{\lambda},\sL(\rho_L))$ is a subrepresentation of the associated Hecke-isotypic subspace of the Banach space of patched $p$-adic automorphic forms on $\bG(\BA_{F^+})$.\;
			\end{rmk}

	\begin{rmk}\label{rmkforhigherext}(Section \ref{HigherDrTower}) We assume $L=\bQ_p$ for simplicity.\;The recent work of Zicheng Qian \cite{wholeLINV} provides a systematic method (quite technical computations) for computing the higher extension groups $\ext^{i}_{G}\big(v^{\ana}_I(\bm\beta,\ul{\lambda}),v^{\ana}_J(\bm\beta,\ul{\lambda})\big)$ for arbitrary $J\subseteq I\subseteq\Delta_{\beta}$.\;When $r>1$, we prove  in Theorem \ref{decompforextgps} that
		\[\BE_{I_1,I_2}(\bm\beta):=\ext_{G,Z}
		^{|I_1\backslash I_2|}\big(v^{\ana}_{I_1}(\bm\beta,\ul{\lambda}),v^{\ana}_{I_2}(\bm\beta,\ul{\lambda})\big)\cong \homo(\bQ_p^{\times},E)^{\otimes |I_1\backslash I_2|}.\]
		Thus the case $r>1$ differs substantially from the Borel case $r=1$.\;
		
		When $r=1$ and $\pi=|\cdot|_L^{\frac{1-n}{2}}$ (so that $\Delta_{\beta}=\Delta$,\;$\bm\beta:=\{|\cdot|_L^{\frac{n-1}{2}},\;\cdots,\;|\cdot|_L^{\frac{1-n}{2}}\}$ and $\pi_{\beta}$ is trivial representation), we recover the locally $\bQ_p$-analytic generalized (Borel) Steinberg representations $\{v_{I}^{\ana}(\ul{\lambda})\}_{I\subseteq \Delta}$ in \cite{2019DINGSimple} or \cite{wholeLINV}.\;Put $\bE_{n-1}:=\bE_{\Delta,\emptyset}(\bm\beta)$ and $a_{n-1}:=\dim_E\bE_{n-1}$.\;Qian's computation shows that $\{a_k\}_{k\geq 0}$ satisfies $a_k=2a_{k-1}+\sum_{m=0}^{k-2}a_m$ ($k\geq 2$), with $a_0=1$ and $a_1=2$.\;This is also the dimension of the space of universal extensions between certain universal special filtered $(\varphi,N)$-modules.\;This phenomenon indicates that higher extensions between $\{v_{I}^{\ana}(\ul{\lambda})\}_{I\subseteq \Delta}$ play a crucial role in identifying automorphic counterparts of (higher) Fontaine--Mazur $\sL$-invariants in the semistale case (i.e.,\;parameters in the admissible Hodge filtration corresponding to all the positive roots of $\GLN_n$).\;Roughly speaking,\;we may encode all the Fontaine--Mazur $\sL$-invariants by using $\{v^{\ana}_{I}(\ul{\lambda})\}_{I\subseteq \Delta_{\beta}}$; this is the content of the Breuil--Schraen $\sL$-invariants in \cite{wholeLINV}.\;Moreover,\;the corresponding universal successive extension of special filtered $(\varphi,N)$-modules can be realized as $\homo_{\cM(G)}((\st_n^{\ana})^{\vee}[n-1],R\Gamma_{\dr}(\BH^{n-1}))$ (here $\st_n^{\ana}:=v_{\emptyset}^{\ana}(\ul{0})$),\;where $\BH^{n-1}$ is the Drinfeld upper half-space and $R\Gamma_{\dr}(\BH^{n-1})$ denotes the derived global sections of its de Rham complex.\;

	For $r>1$, however, $\{\BE_{I_1,I_2}(\bm\beta)\}_{I_2\subseteq I_1\subseteq \Delta_{\beta}}$ only encodes parabolic Fontaine--Mazur simple $\sL$-invariants. The dimension of $\BE_{\Delta_{\beta},\emptyset}(\bm\beta)$ (i.e., $2^{k-1}$) is much lower than the corresponding dimension of universal de Rham extensions (i.e., the dimension $b_{k-1}$ satisfies the recursion formula $b_{k-1}=(1+r^2)b_{k-2}+r^2\sum_{m=0}^{k-3}b_m$), indicating that the collection $\{\BE_{I_1,I_2}(\bm\beta)\}_{I_2\subseteq I_1\subseteq \Delta_{\beta}}$ is insufficient.\;Motivated by parabolic analogues of Gehrmann's automorphic $\sL$-invariants and the related discussion in \cite[Section 1.2.4]{wholeLINV}, we seek possible locally analytic representations enlarging $\{v^{\ana}_{I}(\bm\beta,\ul{\lambda})\}_{I\subseteq \Delta_{\beta}}$.\;We expect their higher extension groups to provide additional parameters encoding Galois-theoretic information in the potentially semistable case (although the structure of representations are not clear and we do not know how to compute higher extension groups in general).\;We also formulate a possible geometric realization using the Drinfeld tower:\;when a suitable locally analytic representation is available, a parabolic version of the universal successive extension of filtered modules (resp.,\;filtered $(\varphi,N)$-modules) should be realized as the space of homomorphisms from the dual of that representation to the derived global sections of the de Rham complex (resp.,\;the Hyodo--Kato complex, if possible) of the Drinfeld tower.\;
	\end{rmk}


	\subsection*{Sketch of the proof of Theorem \ref{intromain}}
	The proof follows the approach of \cite[Section 2.2]{2019DINGSimple}.\;We begin by computing smooth extension groups between
	$\{v^{\infty}_{I}(\bm\beta)\}_{I\subseteq \Delta_{\beta}}$; see Appendix \ref{presmoothext}.\;Using locally $\bQ_p$-analytic Tits complexes, spectral sequences, the Bruhat filtration (Section \ref{Brufil}), and d\'{e}vissage, we reduce the computations in (\ref{analyticExt1}) and (\ref{analyticExtC1}) to those of
	$\ext^i_{\bL_{\beta,I}(L)}(\pi_{\beta,I},\pi_{\beta,I})$ and $\ext^i_{\bL_{\beta,I}(L),Z}(\pi_{\beta,I},\pi_{\beta,I})$,\;where $\ext^i_{\bL_{\beta,I}(L),Z}(\pi_{\beta,I},\pi_{\beta,I})$ consists of extensions on which the center $Z$ of $\bL_{\beta,I}(L)$ acts via the central character of $\pi_{\beta,I}$.\;The theory of Bushnell--Kutzko types, more precisely semisimple types for $\GLN_n$, provides an explicit computation of the smooth extension groups $\ext^{i,\infty}_{\bL_{\beta,I}(L)}(\pi_{\beta,I},\pi_{\beta,I})$ and $\ext^{i,\infty}_{\bL_{\beta,I}(L),Z}(\pi_{\beta,I},\pi_{\beta,I})$; see Lemma \ref{SmoothEXTlemmaPre} in the Appendix.\;By Lemma \ref{smoothtoanalytic},\;a spectral sequence then relates these smooth extension groups to the locally $\bQ_p$-analytic extension groups $\ext^i_{\bL_{\beta,I}(L)}(\pi_{\beta,I},\pi_{\beta,I})$ and $\ext^i_{\bL_{\beta,I}(L),Z}(\pi_{\beta,I},\pi_{\beta,I})$.\;For $i\geq 0$,\;Section \ref{comphinftyhana} proves that tensoring extensions with $\pi_{\beta,I}$ induces isomorphisms in every degree between the locally analytic cohomology of the Levi subgroup and these locally analytic self-extension groups: \[\ext^i_{\bL_{\beta,I}(L)}(1,1)\xrightarrow{\sim} \ext^i_{\bL_{\beta,I}(L)}(\pi_{\beta,I},\pi_{\beta,I}),\quad \ext^i_{\bL_{\beta,I}(L),Z}(1,1)\xrightarrow{\sim} \ext^i_{\bL_{\beta,I}(L),Z}(\pi_{\beta,I},\pi_{\beta,I}).\]
    Note that $\ext^1_{\bL_{\beta,I}(L)}(1,1)\cong \homo(\bZ_{\beta,I}(L),E)$,\;where $\bZ_{\beta,I}$ is the center of $\bL_{\beta,I}$.\;Based on these results,\;for any $J\subset I$ with $|I|=|J|+1$,\;we define the following map and prove that it is an isomorphism (the main task of Section \ref{anaext1})
	\begin{equation*}
		\varrho^0_{I,J}: \homo\big(\bZ_{\beta,J}(L),E\big)/\homo\big(\bZ_{\beta,I}(L),E\big)\cong \homo(L^\times,E)\xrightarrow{\sim} \ext^{1}_{G}\big(i_{I}(\bm\beta,\ul{\lambda}),v_{J}^{\ana}(\bm\beta,\ul{\lambda})\big).\;
	\end{equation*}
	Then we show that $\varrho^0_{I,J}$ factors through the following isomorphism (induced by the natural surjection $i_{I}(\bm\beta,\ul{\lambda})\twoheadrightarrow v^{\lalg}_{I}(\bm\beta,\ul{\lambda})$)
	\[\ext^{1}_{G}\big(v^{\lalg}_{I}(\bm\beta,\ul{\lambda}),v_{J}^{\ana}(\bm\beta,\ul{\lambda})\big)\xrightarrow{\sim }\ext^{1}_{G}\big(i_{I}(\bm\beta,\ul{\lambda}),v_{J}^{\ana}(\bm\beta,\ul{\lambda})\big).\]


	
	
\subsubsection*{Acknowledgements}

This paper forms part of the author's PhD thesis at the School of Mathematical Sciences of Peking University.\;I would like to express my sincere gratitude to my advisor Yiwen Ding for introducing me to this subject, for many helpful discussions and suggestions, and for his comments on an earlier draft of this paper.\;This work is funded by the Natural Science Foundation of China [Grant numbers 8200905010, 8200800065].

	\section{Preliminaries}
	
	\subsection{General notation}\label{notationger}
	Let $L$ (resp.\;$E$) be a finite extension of $\bQ_p$ with $\co_L$ (resp.\;$\co_E$) as its ring of integers and $\varpi_L$ (resp.\;$\varpi_E$) a uniformizer.\;Suppose that $E$ is sufficiently large to contain the images of all embeddings of $L$ into $\overline{\bQ}_p$.\;Put $\Sigma_L:=\{\sigma: L\hookrightarrow \overline{\bQ}_p\} =\{\sigma: L\hookrightarrow E\}$,\;the set of all the embeddings of $L$ in $\overline{\bQ}_p$ (equivalently, in $E$).\;Let $k_E$ be the residue field of $E$.\;Let $\valua_L(\cdot)$ (resp.\;$\valua_p$) be the $p$-adic valuation on $\overline{\bQ}_p$ normalized by sending uniformizers of $\co_L$ (resp.\;of $\BZ_p$) to $1$.\;Let $d_L:=[L:\bQ_p]=|\Sigma_L|$ and $q_L:=|\co_L/\varpi_L|=|k_L|$.

    For a Lie algebra $\fh$ over $L$ and $\sigma\in \Sigma_L$, let $\fh_{\sigma}:=\fh\otimes_{L,\sigma} E$.\;For $J\subseteq \Sigma_L$, put $\fh_J:=\prod_{\sigma\in J} \fh_{\sigma}$.\;In particular, $\fh_{L}:=\fh_{\Sigma_L}\cong \fh\otimes_{\bQ_p} E$.\;

	Let $\GLN_n$ denote the general linear group over $L$.\;Let $\Delta:=\Delta_n$ be its set of simple roots with respect to the Borel subgroup $\bB$ of upper triangular matrices.\;We identify $\Delta$ with $\{1,\ldots,n-1\}$ so that $i$ corresponds to the simple root $\alpha_i:(x_1,\ldots,x_n)\mapsto x_i-x_{i+1}$ on the Lie algebra $\ft$ of the diagonal torus $\bT$.\;Each subset $I\subseteq \Delta$ determines a standard Levi subgroup $\bL_I\subseteq \GLN_n$ containing $\bT$, whose set of simple roots is $I$.\;Let $\mathbf{P}_{I}:=\mathbf{L}_{I}\bB$ be the standard parabolic subgroup of
	$\GLN_n$ containing $\bB$.\;In particular,\;we have $\mathbf{P}_{\Delta}=\GLN_n$, $\mathbf{P}_{\emptyset}=\mathbf{B}$.\;Let $\overline{\mathbf{P}}_{I}$ (resp.,\;$\overline{\bB}$) be the parabolic subgroup opposite to $\mathbf{P}_{I}$ (resp.,\;$\bB$).\;Let $\bN_{I}$ (resp.\;$\overline{\bN}_{I}$) be the unipotent radical of $\bP_{I}$ (resp.\;$\overline{\bP}_{I}$).\;We have Levi decompositions $\bP_I=\bL_I\bN_I$,\;$\overline{\bP}_I=\bL_I\overline{\bN}_I$ and $\Delta \backslash I$ is precisely the set of simple roots of the unipotent radical $\bN_{I}$.\;Let $\bZ_I$ (resp.,\;$\bZ$) be the center of $\bL_I$ (resp.,\;$\GLN_n$).\;Let $\fg$, $\fb$, $\overline{\fb}$, $\fp_I$, $\fn_I$, $\fl_I$, $\overline{\fl}_I$, $\overline{\fp}_I$, $\overline{\fn}_I$, $\fz_I$, $\overline{\fz}_I$ and $\ft$ be the $L$-Lie algebras of $\GLN_n$, $\bB$, $\overline{\bB}$,  $\mathbf{P}_I$, $\mathbf{N}_I$, $\mathbf{L}_I$, $\overline{\bL}_I:={\bL}_I/\bZ$, $\overline{\mathbf{P}}_I$, $\overline{\mathbf{N}}_I$, $\bZ_{I}$, $\bZ_{I}/\bZ$ and $\bT$ respectively.\;Throughout this paper, we put $G:=\GLN_n(L)$.\;

Let $\ul{\lambda}=(\ul{\lambda}_{\sigma})_{\sigma\in \Sigma_L}=(\lambda_{1,\sigma}, \ldots, \lambda_{n,\sigma})_{\sigma\in \Sigma_L}$ be a weight of $\ft_L$.\;For $I\subseteq \Delta$, we say that $\ul{\lambda}$ is $I$-dominant with respect to $\res_{L/\bQ_p}\bB$ if $\lambda_{i,\sigma}\geq \lambda_{i+1,\sigma}$ for all $i\in I$ and $\sigma\in \Sigma_L$.\;Let $X_I^+$ denote the set of integral weights that are $I$-dominant with respect to $\res_{L/\bQ_p}\bB$.\;For $\ul{\lambda}\in X_I^+$, let $L(\ul{\lambda})_I$ be the unique irreducible algebraic representation of $\res_{L/\bQ_p}\bL_I$ with highest weight $\ul{\lambda}$ relative to $\res_{L/\bQ_p}(\bL_I\cap\bB)$.\;Put $\overline{L}(-\ul{\lambda})_I:=L(\ul{\lambda})_I^\vee$; this is the irreducible algebraic representation of $\res_{L/\bQ_p}\bL_I$ with highest weight $-\ul{\lambda}$ relative to $\res_{L/\bQ_p}(\bL_I\cap\ob)$.\;Write $\chi_{\ul{\lambda}}:=L(\ul{\lambda})_{\emptyset}$, and, if $\ul{\lambda}\in X_{\Delta}^+$, set $L(\ul{\lambda}):=L(\ul{\lambda})_{\Delta}$.\;A $\bQ_p$-algebraic representation of $G$ over $E$ is obtained by restricting an algebraic representation of $\res_{L/\bQ_p}\GLN_n$ to $G$.\;

Let $\ul{\lambda}$ be an integral weight.\;We denote by $M(\ul{\lambda}):=\text{U}(\fg_{L})\otimes_{\text{U}(\fb_{L})} \ul{\lambda}$ (resp., $\overline{M}(\ul{\lambda}):=\text{U}(\fg_{L})\otimes_{\text{U}(\overline{\fb}_{L})} \ul{\lambda}$) the corresponding Verma module with respect to $\fb_{L}$ (resp., $\overline{\fb}_{L}$).\;Let $L(\ul{\lambda})$ (resp., $\overline{L}(\ul{\lambda})$) be the unique simple quotient of $M(\ul{\lambda})$ (resp., $\overline{M}(\ul{\lambda})$).\;When $\ul{\lambda}\in X_{\Delta}^+$, $L(\ul{\lambda})$ is finite-dimensional and isomorphic to the algebraic representation introduced above under the same notation.\;We have $\overline{L}(-\ul{\lambda})\cong L(\ul{\lambda})^{\vee}$.\;For $I\subseteq \Delta$ and $\ul{\lambda}\in X_{I}^+$, we define the parabolic Verma module
\begin{equation}
	\begin{aligned}
		M_I(\ul{\lambda}):=\text{U}(\fg_{L})\otimes_{\text{U}(\fp_{I,L})} L(\ul{\lambda})_I,\;\quad\text{resp.,\;}\overline{M}_I(-\ul{\lambda}):=\text{U}(\fg_{L})\otimes_{\text{U}(\overline{\fp}_{I,L})} \overline{L}(-\ul{\lambda})_I
	\end{aligned}
\end{equation}
with respect to $\fp_{I,L}$ (resp.\;$\overline{\fp}_{I,L}$),\;see \cite[Chapter 9]{humphreysBGG} for more precise statements.\;

	Let $\sW_n\cong S_n$ be the Weyl group of $\GLN_n$, and let $s_i$ be the simple reflection corresponding to $i\in \Delta$.\;For $I\subseteq\Delta$, let $\sW_I$ be the subgroup generated by the reflections $s_i$ with $i\in I$.\;We identify each $w\in\sW_n$ with its permutation matrix.\;For subsets $I,J\subseteq\Delta$, let $\sW_n^{I,J}$ denote the canonical set of minimal-length representatives for $\sW_I\backslash\sW_n/\sW_J$.\;The relative Bruhat decompositions are
	$\GLN_n=\coprod_{w\in \sW_n^{I,J}}\bP_{I}w\bP_{J}=\coprod_{w\in \sW_n^{I,J}}\op_{I}w\op_{J}$.\;Put $\sW_{n,L}:=\prod_{\sigma\in \Sigma_L}\sW_{n,\sigma}\cong S_n^{d_L}$, where $\sW_{n,\sigma}\cong\sW_n$ is the $\sigma$-factor.\;For $i\in\Delta$ and $\sigma\in\Sigma_L$, let $s_{i,\sigma}\in\sW_{n,\sigma}$ be the corresponding simple reflection.\;Let $\rho$ be half the sum of the positive roots of $\GLN_n$.\;The dot action of $w=(w_\sigma)\in\sW_{n,L}$ on $\ul{\lambda}$ is given by $w\cdot\ul{\lambda}=(w_{\sigma}(\ul{\lambda}_{\sigma}+\rho)-\rho)_{\sigma\in \Sigma_L}$.\;For $I\subseteq\Delta$, let $\sW_{I,\sigma}$ be the copy of $\sW_I$ in $\sW_{n,\sigma}$, generated by the $s_{i,\sigma}$ with $i\in I$.\;For $S\subseteq\Sigma_L$ and $I\subseteq\Delta$, put $\sW_{n,L}^{I,\emptyset}:=\prod_{\sigma\in \Sigma_L}\sW_{n,\sigma}^{I,\emptyset}$.\;
	

	For a topological commutative group $M$, we use $\homo(M,E)$ (resp.\;$\homo_{\sm}(M,E)$) to denote the $E$-space of continuous (resp.\;locally constant) additive $E$-valued characters on $M$.\;If $M$ is a locally $L$-analytic group, denote by $\homo_\sigma(M,E)$ the $E$-vector space of locally $\sigma$-analytic characters on $M$ (i.e., the continuous characters which are locally $\sigma$-analytic $E$-valued functions on $M$).\;For $\alpha\in E^\times$, denote by $\unr(\alpha)$ the unramified character of $L^\times$ sending uniformizers to $\alpha$.\;We normalize the local class theory by sending   the uniformizer $\varpi_L$ to a lift of the geometric Frobenius,\;we obtain  a bijection $\homo(L^\times,E)\cong \homo(\gal_L,E)$.\;


	Let $H$ be a closed subgroup of $G$.\;If $V$ is a continuous representation of $H$ over $E$, we denote by $V^{\ana}$ its locally $\bQ_p$-analytic vectors.\;If $V$ is
	a locally $\bQ_p$-analytic representation of $H$, we denote by $V^{\mathrm{lalg}}$ its locally $\bQ_p$-algebraic subrepresentation of $V$ consisting of its locally $\bQ_p$-algebraic vectors (see \cite{Emerton2007summary}, \cite{emerton2017locally} and \cite{schneider2002banach} for details).\;
	
	Let $(\rho,V)$ be an abstract representation of $H$ on an $E$-vector space $V$.\;Put $H^w:=w^{-1}Hw$ (resp., ${}^wH:=wHw^{-1}$) and $\rho^w(h):=\rho(whw^{-1})$.\;Then $(\rho^w,V)$ is an abstract representation of $H^w$ on the same vector space.\;Let $\pi_H$ be a smooth representation of $H$ over $E$.\;Define
	\begin{equation*}
		\begin{aligned}
			\big(\mathrm{ind}^G_H\pi_H\big)^{\infty}:=\big\{f:G\rightarrow \pi_H \text{ smooth}: \text{the support of $f$ is compact modulo $H$}, f(hg)=hf(g),\;\forall\;h\in H\big\}
		\end{aligned}
	\end{equation*}
	to be the smooth compact induction.\;The group $G$ acts smoothly by right translation on functions.\;

	Let $I\subseteq \Delta$, and set $M:=\bL_I(L)\subset P:=\op_{I}(L)$ and $N:=\on_I(L)$ for short.\;We review continuous,\;locally $\bQ_p$-analytic,\;and smooth parabolic inductions.\;Let $\pi'$ be a continuous representation of $P$ over $E$ (resp., a locally $\bQ_p$-analytic representation of $P$ on a \textit{locally convex} $E$-vector space of compact type, resp., a smooth representation of $P$ over $E$).\;We denote by
	\begin{equation}\label{smoothadj2}
		\begin{aligned}
			&\big(\mathrm{Ind}_{P}^{G}\pi'\big)^{\cC^0} \big(\text{resp.,\;}\big(\mathrm{Ind}_{P}^{G}\pi'\big)^{\bQ_p-\ana},\;\text{resp.,\;}\big(\mathrm{Ind}_{P}^{G}\pi'\big)^{\infty}\big)\\
			:=&\;\Big\{f:G\rightarrow \pi' \text{\;continuous} \;(\text{resp.,\;locally $\bQ_p$-analytic,\;resp.,\;smooth}),\;f(pg)=pf(g),\forall\;p\in P\Big\}
		\end{aligned}
	\end{equation}
	the unnormalized continuous (resp., locally $\bQ_p$-analytic, resp., smooth) parabolic induction to $G$.\;With the action $(gf)(g')=f(g'g)$ by right translation, this is a continuous (resp., locally $\bQ_p$-analytic, resp., smooth) representation of $G$ in the corresponding category.\;
	
Let $(\pi,V)$ be a smooth representation of $G$ over $E$.\;Its unnormalized Jacquet module with respect to $P$ is
	$r^G_P(V):=J_{N}(V)=V/\langle \pi(n)v-v:n\in N, v\in V\rangle$.\;This is a smooth representation of $M$ over $E$.\;If $V$ is a smooth $G$-representation and $W$ is a smooth $M$-representation, then the adjunction formula gives
	\begin{equation}\label{smoothadj1}
		\begin{aligned}
			\homo_G\big(V,(\mathrm{Ind}_{P}^{G}W)^{\infty}\big)\xrightarrow{\sim}\homo_{M}\big(r^G_P(V),W\big).\;
		\end{aligned}
	\end{equation}

	\subsection{Preliminaries on locally \texorpdfstring{$\bQ_p$}{Lg}-analytic representations}\label{Basicsettingoflocanalyticrep}
	
	In this section, let $H$ be a strictly paracompact locally $L$-analytic group (where strictly paracompact means that any open covering of $H$ can be refined into a covering by pairwise disjoint open subsets).\;
	
	Let $V$ be a Hausdorff \textit{locally convex} $E$-vector space.\;Let $\cC^{\bQ_p-\ana}(H,V)$ be the locally convex $E$-vector space of locally $\bQ_p$-analytic $V$-valued functions on $H$ (\cite[Section 2, Page 447]{schneider2002locally}; see also \cite[Definition 2.1.25]{emerton2017locally}).\;Let $\cD^{\bQ_p-\ana}(H,E)$ be the strong dual of $\cC^{\bQ_p-\ana}(H,E)$,\;called the locally convex $E$-algebra of $E$-valued distributions on $H$ (see \cite[Proposition 2.3]{schneider2002locally}).\;

	Let $Z'$ be a  locally $\bQ_p$-analytic closed subgroup of the center $Z_H$ of $H$, and fix a smooth character $\chi^{\infty}$ (resp., a locally $\bQ_p$-analytic character $\chi$) of $Z'$.\;Let ${\textbf{Rep}^{\infty}_E(H)}$ denote the category of smooth representations of $H$ over $E$, and let ${\textbf{Rep}_{E,Z'=\chi^{\infty}}^{\infty}(H)}$ be its full subcategory on which $Z'$ acts through $\chi^{\infty}$.\;Morphisms in both categories are $H$-equivariant linear maps.\;We refer to \cite[Section 3.6]{emerton2017locally} and \cite[Section 3]{schneider2002locally} for the definitions of locally $\bQ_p$-analytic representations of $H$ on locally convex $E$-vector spaces.\;Similarly, let
	${\textbf{Rep}_E(H)}$ denote the category of locally $\bQ_p$-analytic representations of $H$ on barrelled locally convex Hausdorff $E$-vector spaces, and let ${\textbf{Rep}_{E,Z'=\chi}(H})$ be the full subcategory on which $Z'$ acts through $\chi$.\;Morphisms are continuous $H$-equivariant linear maps.\;

   Let $\cM(H)$ be the abelian category of abstract $\cD^{\bQ_p-\ana}(H,E)$-modules, and let $\cM_{Z',\chi}(H)$ be the full subcategory consisting of modules on which $Z'$ acts through $\chi^{-1}$.\;If $V$ is a locally $\bQ_p$-analytic representation of $H$, its continuous dual $V^\vee$ carries a canonical $\cD^{\bQ_p-\ana}(H,E)$-module structure defined by $\delta_g\cdot w(v)=w(g^{-1}v)$.\;

	\subsection{Smooth and locally $\bQ_p$-analytic extension groups}\label{smanaextgps}
	
	Recall that ${\textbf{Rep}^{\infty}_E(H)}$ has enough injectives and projectives;\;see \cite[Section 3]{2012Orlsmoothextensions}.\;Let $V,W\in\Rep^{\infty}_E(H)$ be admissible smooth representations.\;For $i\in\BZ_{\geq0}$, put
	\begin{equation*}
		\begin{aligned}
			&\ext^{i,\infty}_{H}(V, W):=\ext^i_{\Rep^{\infty}_E(H)}(V, W),\;\hH^{i}_{\sm}(H,V):=\ext^i_{\Rep^{\infty}_E(H)}(1,V).\;
		\end{aligned}
	\end{equation*}
	for the $i$-th smooth extension group of $V$ and $W$.\;If further $V,W\in \Rep^{\infty}_{E,Z'=\chi^{\infty}}(H)$, we put
	\[\ext^{i,\infty}_{H,Z'}(V, W):=\ext^i_{\Rep^{\infty}_{E,Z'=\chi^{\infty}}(H)}(V, W)\]
	for the $i$-th smooth extension group of $V$ and $W$ on which $Z'$ acts via $\chi^{\infty}$.\;
	
	In this paper, we define extension groups of locally $\bQ_p$-analytic representations in the category $\cM(H)$ of all abstract $D^{\bQ_p-\ana}(H,E)$-modules.\;Let $V$ and $W$ be locally $\bQ_p$-analytic representations of $H$ over $E$.\;As in \cite[D\'{e}finition 3.1]{schraen2011GL3}, we put
	\begin{equation}
		\begin{aligned}
			&\ext^r_{H}(V,W):=\ext^r_{\cM(H)}(W^{\vee},V^{\vee}),\quad\hH^r_{\ana}(H,V)=\ext^r_H(1,V), 
		\end{aligned}
	\end{equation}
	where $\ext^r_{\cM(H)}(-,-)$ denotes the $r$-th extension group in the abelian category $\cM(H)$.\;If $Z'$ acts on both $V$ and $W$ via the character $\chi$, we put
	\begin{equation*}
		\ext^r_{H,Z'}(V,W):=\ext^r_{\cM_{Z',\chi}(H)}(W^{\vee},V^{\vee}),
	\end{equation*}
	where the latter denotes the $r$-th extension group in $\cM_{Z',\chi}(H)$.\;If moreover $V$ and $W$ are admissible locally $\bQ_p$-analytic representations of $H$ over $E$ (resp.,  admissible locally $\bQ_p$-analytic representations of $H$ on which $Z'$ acts via $\chi$), then \cite[Lemma 2.2]{2019DINGSimple} implies that
	$\ext^1_{H}(V,W)$  (resp., $\ext^1_{H,Z'}(V,W)$) consists of admissible locally $\bQ_p$-analytic representations which are extensions of $V$ by $W$ (resp.,\;on which $Z'$ acts via $\chi$).\;
	
	\begin{rmk}\label{introsmoothfixcenter}
If the smooth character $\chi^{\infty}$ can be lifted to a smooth character of $H$,\;we have an equivalence of categories $\Rep_{E,Z'}^{\infty}(H)\xrightarrow{\sim}\Rep^{\infty}_E(H/Z')$, $V\mapsto V\otimes_E (\chi^{\infty})^{-1}$.\;In the locally $\bQ_p$-analytic case,\;if the locally $\bQ_p$-analytic character  $\chi$ can be lifted to  a locally $\bQ_p$-analytic character of $H$,\;we have a natural equivalence of categories between $\cM_{Z',\chi}(H)\cong \cM(H/Z')$.\;Therefore,\;we have $\ext^r_{H,Z'}(V,W)\cong\ext^r_{H/Z'}(V\otimes_E\chi^{-1},W\otimes_E\chi^{-1})$.\;

	\end{rmk}
	
By \cite[Section 6]{vigneras1997extensions}, we have
\begin{lem}\label{centervan}Let $V,W\in \Rep^{\infty}_E(H)$.\;Suppose there is an element $z\in Z_H$ such that $z$ acts on $V$ and $W$ by multiplication by two different scalars $z_V\neq z_W$.\;Then we have	$\ext^{\bullet,\infty}_{H}(V, W)=0$.
\end{lem}
	
	We end this  section by recalling some spectral sequences which give relations between smooth extension groups and locally $\bQ_p$-analytic extension groups.\;By \cite[Page 306-307 (+)]{schneider2005duality} and \cite[(46)]{kohlhaase2011cohomology}, we have
	\begin{lem}\label{smoothtoanalytic}For any $V,W\in \mathbf{{Rep}}_E^\infty(H)$,\;there exists a spectral sequence
		\begin{equation}\label{Specsequencesmanalytic}
			\ext^{r,\infty}_{H}(\hH_s(\fh_{L},1)\otimes_EV,W)\Rightarrow \ext^{r+s}_H(V,W).\;
		\end{equation}
		Furthermore, if $V,W\in \mathbf{{Rep}}_{E,Z'=\chi^{\infty}}^\infty(H)$,\;there exists a spectral sequence
		\begin{equation}\label{Specsequencesmanalyticcenter}
			\ext^{r,\infty}_{H,Z'}(\hH_s(\overline{\fh}_{L},1)\otimes_EV,W)\Rightarrow \ext^{r+s}_{H,Z'}(V,W),
		\end{equation}
		where $\fh$ (resp.,\;$\overline{\fh}$ is the $L$-Lie algebra of  $H$ (resp.,\;$H/Z'$).\;Note that $\hH_s(\fh_{L},-)$ (resp.,\;$\hH_s(\overline{\fh}_{L},-)$) is $s$-th Lie algebra homology for $\fh_{L}$ (resp.,\;$\overline{\fh}_L$).\;
	\end{lem}
	\begin{rmk}
		We obtain the well-known  $5$-terms  exact sequence
		\begin{equation}\label{smoothanalyticedge}
			\begin{aligned}
				0\rightarrow \ext^{1,\infty}_{H}(V, W)\rightarrow \ext^{1}_{H}(V, W) \rightarrow \homo_{H}(V,&W\otimes_E\hH_1(\fh_{L},E))\\
				&    \rightarrow \ext^{2,\infty}_{H}(V, W)\rightarrow \ext^{2}_{H}(V, W)\\
			\end{aligned}
		\end{equation}
		\begin{equation}\label{smoothanalyticedgecenter}
			\begin{aligned}
				&\big(\mathrm{resp.\;}0\rightarrow \ext^{1,\infty}_{H,Z'}(V, W)\rightarrow \ext^{1}_{H,Z'}(V, W) \rightarrow \homo_{H,Z'}(V,W\otimes_E\hH_1(\overline{{\fh}}_{L},E))\\
				&\rightarrow \ext^{2,\infty}_{H,Z'}(V, W)\rightarrow \ext^{2}_{H,Z'}(V, W)\big).
			\end{aligned}
		\end{equation}
		
	\end{rmk}
	
The next lemma lists some well-known exact sequences arising from spectral sequences.\;
\begin{lem}\label{degspectral}Let $E$ be a biregular spectral sequence.\;
	\begin{description}
		\item[(a)] Suppose $E_2^{p,q}=0$ for any $p\neq p_1,p_2$ and $p_2-p_1=1$, then we have a short exact sequence
		\[0\rightarrow E_{2}^{p_2,m-p_2}\rightarrow H^m\rightarrow  E_{2}^{p_1,m-p_1}\rightarrow 0.\]
		\item[(b)] Suppose $E_2^{p,q}=0$ for any $q\neq q_1,q_2$ ($q_1<q_2$) , then we have a long exact sequence
		\[\cdots\rightarrow E_{2}^{m-q_1,q_1}\rightarrow H^m\rightarrow  E_{2}^{m-q_2,q_2}\rightarrow E_{2}^{m+1-q_1,q_1} \rightarrow H^{m+1}\rightarrow\cdots.\]
	\end{description}
\end{lem}

\section{Generalized Steinberg representations}

Before studying the locally $\bQ_p$-analytic generalized Steinberg representations, we need some preliminary results on smooth generalized Steinberg representations of $G$.\;We study the generalized  Steinberg representation attached to the Zelevinsky segments (we recall the basic concepts of Zelevinsky segments in Appendix Section \ref{introBZ}).\;To study parabolically induced smooth principal series (Section \ref{smoothgerparastrep}), and define  Bruhat filtration (see (\ref{filext})),\;we review in Section \ref{Doubleset} some properties of the minimal-length representatives of certain double cosets in $\sW_n$.\;In Section \ref{smoothgerparastrep}, we discuss Smooth and locally $\bQ_p$-algebraic generalized Steinberg representations and describe 
the Jordan--H\"older factors of them.\;A proof is included in Appendix \ref{JHfactorPARA}.\;

Throughout this paper,\;fix two integers $k,r$ such that $n=kr$.\;Let $\beta=(r,\cdots,r)$ (the number of $r$ is $k$) be the corresponding partition  of $n$.\;Put
\[\Delta_{\beta}:=\{r,2r,\cdots,(k-1)r\}\subseteq \Delta,\quad \Delta^{\beta}:=\Delta\backslash\Delta_{\beta}.\]
For a subset $I\subset \Delta_{\beta}$, we use the shortened notation $S^{\beta}(I):=\Delta^{\beta}\cup I$ and 
\begin{equation}\label{lrropbL}
	\begin{aligned}
		&\bL_{\beta,I}:=\bL_{S^{\beta}(I)}, \bP_{\beta,I}:=\bP_{S^{\beta}(I)}, \op_{\beta,I}:=\op_{S^{\beta}(I)},\bN_{\beta,I}:=\bN_{S^{\beta}(I)}, \on_{\beta,I}:=\on_{S^{\beta}(I)}, \bZ_{\beta,I}:=\bZ_{S^{\beta}(I)}.
	\end{aligned}
\end{equation}
Let $\fl_{\beta,I}$, $\overline{\fl}_{\beta,I}$, $\fp_{\beta,I}$, $\fn_{\beta,I}$, $\overline{\fp}_{\beta,I}$, $\overline{\fn}_{\beta,I}$, $\fz_{\beta,I}$, $\overline{\fz}_{\beta,I}$ be the $L$-Lie algebras of $\bL_{\beta,I}$, $\bL_{\beta,I}/\bZ$, $\bP_{\beta,I}$, $\bN_{\beta,I}$, $\op_{\beta,I}$, $\on_{\beta,I}$ and $\bZ_{\beta,I}$, $\bZ_{\beta,I}/\bZ$ respectively.\;When $I=\emptyset$, we omit the subscripts $I$ in (\ref{lrropbL}).\;In particular,\;
\[\bL_{\beta}:=\left(\begin{array}{cccc}
	\GLN_r & 0 & \cdots & 0 \\
	0 & \GLN_r & \cdots & 0 \\
	\vdots & \vdots & \ddots & 0 \\
	0 & 0 & 0 & \GLN_r \\
\end{array}\right),\quad \op_{\beta}:=\left(\begin{array}{cccc}
	\GLN_r & 0 & \cdots & 0 \\
	\ast & \GLN_r & \cdots & 0 \\
	\vdots & \vdots & \ddots & 0 \\
	\ast & \ast & \cdots & \GLN_r \\
\end{array}\right).\]
The parabolic subgroups of $\GLN_n$ containing the parabolic subgroup $\op_{\beta}$ are  $\{\op_{\beta,I}\}_{I\subseteq \Delta_{\beta}}$.\;

Let $I\subseteq\Delta_{\beta}$.\;Put $l_I:=k-|I|$.\;There exists an ordered partition $\underline{k}_{I}:=(k_1,\cdots,k_{l_I})$ of $k$ such that $\Delta_{\beta}\backslash I=\{k_1r,(k_1+k_2)r,\cdots,(k_1+\cdots+k_{l_I-1})r\}$ and  (and thus $\bL_{\beta,I}=\GLN_{k_1r}\times \cdots\times\GLN_{k_{l_I}r}$).\;For $0\leq i\leq l_I$, we put $s_0:=0$ and  $s_i=\sum_{j=1}^i k_j$.\;Then the modulus character $\delta_{\op_{\beta,{I}}}$ of $\op_{\beta,{I}}$ is given  by (see \cite[Definition 14.3.6]{goldfeld2011automorphic})
$$\delta_{\op_{\beta,{I}}}\left(\left(
\begin{array}{cccc}
	A_1 & 0 & \cdots & 0 \\
	\ast & A_2 & \cdots & 0 \\
	\ast & \ast & \ddots & 0 \\
	\ast & \ast & \ast & A_{l_I} \\
\end{array}
\right)\right)=\prod_{i=1}^{l_I}|\det(A_i)|^{r(2s_{i-1}+k_i-k)}_L.
$$
In particular, if $I=\emptyset$, then $r(2s_{i-1}+k_i-k)=r(2i-1-k)$.\;


\subsection{Double coset representatives}\label{Doubleset}

This section aims to review some properties of the minimal-length representatives of certain double cosets in $\sW_n$, which are related to the smooth parabolic induction associated with a Zelevinsky segment.\;

Let $I,J\subseteq \Delta_{\beta}$ and $w\in \sW_n^{S^{\beta}(I),S^{\beta}(J)}$.\;Following \cite[Theorem 2.7.4]{RWcarter}, consider
\begin{equation}
	\begin{aligned}
		S^{\beta}(J)\cap w^{-1}(S^{\beta}(I))=\{j\in S^{\beta}(J): w(j)\in S^{\beta}(I)\},\;S^{\beta}(I)
		\cap w(S^{\beta}(J))=w(S^{\beta}(J)\cap w^{-1}(S^{\beta}(I))).\;
	\end{aligned}
\end{equation}
Moreover,\;by \cite[Corollary 2.8.9]{RWcarter},\;$\bL_{\beta,J}\cap \op_{\beta,I}^{w}$ (resp.,\;$\bL_{\beta,I}\cap {}^w\op_{\beta,J}$) is a parabolic subgroup of $\bL_{\beta,J}$ (resp.,\;$\bL_{\beta,I}$),\;so that the Levi decomposition of $\bL_{\beta,J}\cap \op_{\beta,I}^{w}$ (resp., $\bL_{\beta,I}\cap {}^w\op_{\beta,J}$)  is given by 
\begin{equation}\label{LEIVEIJ}
	\begin{aligned}
		\bL_{\beta,J}\cap \op_{\beta,I}^{w}=&\;\bL_{\beta,J}\cap \bL_{\beta,I}^{w}\cdot(\bL_{\beta,J}\cap \on_{\beta,I}^{w})=\bL_{S^{\beta}(J)\cap w^{-1}(S^{\beta}(I))}\cdot(\bL_{\beta,J}\cap \on_{\beta,I}^{w}),\\
		\text{resp., }\bL_{\beta,I}\cap {}^w\op_{\beta,J}=&\;\bL_{\beta,I}\cap {}^w\bL_{\beta,J}\cdot(\bL_{\beta,I}\cap {}^w\on_{\beta,J})=\bL_{S^{\beta}(I)
			\cap w(S^{\beta}(J))}\cdot(\bL_{\beta,I}\cap {}^w\on_{\beta,J}).
	\end{aligned}
\end{equation}
We have a natural map
\begin{equation}\label{fbetasWbeta}
	\begin{aligned}
		f_{\beta}:\sW_k\rightarrow \sW^{\Delta_\beta,\Delta_\beta}_n,\quad u\mapsto f_{\beta}(u),\;
	\end{aligned}
\end{equation}
where $f_{\beta}(u)(ir+l)=(u(i+1)-1)r+l$ for $0\leq i\leq k-1$ and $1\leq l\leq r$.\;For $I=\{i_1r,\cdots,i_{|I|}r\},J=\{j_1r,\cdots,j_{|J|}r\}\subseteq \Delta_{\beta}$,\;let 
\[\sW_n^{\beta,I,J}:=f_{\beta}\big(\sW_k^{\{i_1,\cdots,i_{|I|}\},\{j_1,\cdots,j_{|J|}\}}\big).\]
If $w\in\sW_n^{\beta}:=\sW_n^{\beta,\emptyset,\emptyset}\cong \sW_k$, we further have Levi decompositions:
\begin{equation}\label{LEIVEIJrk}
	\begin{aligned}
		\bL_{\beta,J}\cap \op_{\beta,I}^{w}=\bL_{\beta,J\cap w^{-1}I}\cdot(\bL_{\beta,J}\cap \on_{\beta,I}^{w}),\quad \bL_{\beta,I}\cap {}^w\op_{\beta,J}=\bL_{\beta, I\cap wJ}\cdot(\bL_{\beta,I}\cap {}^w\on_{\beta,J}).
	\end{aligned}
\end{equation}

\subsection{Definition of the smooth generalized Steinberg representations}\label{smoothgerparastrep}

In the sequel,\;for $J\subseteq I\subseteq \Delta_{\beta}$,\;we use the shortened notation:
\[i^{\beta}_{J,I}(\cdot):=\left(\mathrm{Ind}_{\op_{\beta,{J}}(L)\cap \bL_{\beta,I}(L)}^{\bL_{\beta,I}(L)}\cdot\right)^{\infty},\quad J^{\beta}_{I,J}(\cdot):=r^{\bL_{\beta,I}(L)}_{\op_{\beta,J}\cap \bL_{\beta,I}(L)}(\cdot)=J_{\on_{\beta,J}\cap \bL_{\beta,I}(L)}(\cdot).\;\]
Fix  an irreducible cuspidal smooth representation $\pi$ of $\GLN_{r}(L)$ over $E$.\;Let $\bm\beta:=\bm\beta_{\pi,k}:=\{\pi\nu^{k-1},\;\cdots,\;\pi\nu,\pi\},$ be the Zelevinsky-segment of length $k$ associated to $\pi$.\;There is an irreducible cuspidal smooth representation of $\bL_{\beta}(L)$ over $E$ associated to $\bm\beta$:\;
\begin{equation}\label{pilrr}
	\pi_{\beta}:=\big(\boxtimes_{i=1}^{k}\pi\nu^{k-i}\big)\otimes_E \delta_{\op_{\beta}}^{1/2}=\boxtimes_{i=1}^{k}\pi\nu^{-\frac{r}{2}(k-2i+1)+k-i}.\;
\end{equation}
For $I\subseteq \Delta_{\beta}$,\;put
\[\pi_{\beta,I}:=\soc_{\bL_{\beta,I}(L)}i^{\beta}_{\emptyset,I}(\pi_{\beta}).\;\]
It is irreducible by \cite[2.10.\;Proposition]{av1980induced2}.\;In particular, $\pi_{\beta,\emptyset}=\pi_{\beta}$, and $\pi_{\beta,\Delta_{\beta}}=\langle\bm\beta\rangle$ (see the discussion after Definition \ref{dfnparastein}) in Appendix.\;

The following proposition plays a key role in proving the exactness of the locally $\bQ_p$-algebraic and locally $\bQ_p$-analytic Tits complexes; see Section \ref{Titcomp}.\;Its proof is elementary but lengthy, whereas the case $r=1$ is immediate.\;We defer the proof to Appendix \ref{proofsaximos}; it may safely be omitted on a first reading.\;

\begin{pro}\label{axioms}The family $\{\pi_{\beta,I}\}_{I\subset \Delta_{\beta}}$ satisfies the following conditions $\mathbf{[A1],[A2]}$.\;
\begin{description}
	\item[\textbf{[A1]}] For $J\subseteq I$, $	\pi_{\beta,J}=J^{\beta}_{I,J}(\pi_{\beta,I})$ and  there is an injection of $\bL_{\beta,I}(L)$-representations	$p^{\infty}_{J,I}:\pi_{\beta,I}\hookrightarrow  i^{\beta}_{J,I}(\pi_{\beta,J})$ (and thus $i^{\beta}_{I,\Delta_{\beta}}(\pi_{\beta,I})\hookrightarrow i^{\beta}_{J,\Delta_{\beta}}(\pi_{\beta,J})\hookrightarrow i^{\beta}_{\emptyset,\Delta_\beta}(\pi_{\beta})$).\;
	\item[\textbf{[A2]}] For any $I,J,I_1,\cdots,I_m\subseteq\Delta_{\beta}$ (all the intersections are taken in $i^{\beta}_{\emptyset,\Delta_\beta}(\pi_{\beta})$), 
	\begin{equation}
			\begin{aligned}
			&i^{\beta}_{I,\Delta_{\beta}}(\pi_{\beta,I})\cap i^{\beta}_{J,\Delta_{\beta}}(\pi_{\beta,J})=i^{\beta}_{I\cup J,\Delta_{\beta}}(\pi_{\beta,I\cup J}), \\
			&i^{\beta}_{I,\Delta_{\beta}}(\pi_{\beta,I})\cap \big(\sum_{j=1}^m i^{\beta}_{I_j,\Delta_{\beta}}(\pi_{I_j})\big)
			=\sum_{j=1}^m i^{\beta}_{I,\Delta_{\beta}}\big(\pi_{\beta,I})\cap i^{\beta}_{I_j,\Delta_{\beta}}(\pi_{I_j}).\;
		\end{aligned}
	\end{equation}
\end{description}
\end{pro}
\begin{dfn}\label{dfnlagpstreplalg}\textbf{(Smooth and locally $\bQ_p$-algebraic  generalized Steinberg representations)}
Let $\underline{\lambda}\in X_{\Delta}^+$, and let $J\subset I\subseteq \Delta_{\beta}$.\;Put 
\begin{equation*}
	\begin{aligned}
		&i_{J,I}(\bm\beta)=i^{\beta}_{J,I}(\pi_{\beta,J}),\quad v^{\infty}_{J,I}(\bm\beta)=i_{J,I}(\bm\beta)\big/u_{J,I}(\bm\beta),\quad u_{J,I}(\bm\beta)=\sum_{I \supseteq K\supsetneq J}i_{K,I}(\bm\beta).\;
	\end{aligned}
\end{equation*}
where the sum in $u_{J,I}(\bm\beta)$ is taken inside $i_{J,I}(\bm\beta)$.\;We add $\underline{\lambda}$ into $(\bm\beta)$ (i.e.,\;replace $(\bm\beta)$ with $(\bm\beta,\underline{\lambda})$) and replace $"\infty"$ with $"\lalg"$ when such smoooth representations are tensored  with $L(\underline{\lambda})$.\;We call $\{v^{\lalg}_{J,I}(\bm\beta,\underline{\lambda})\}_{J\subseteq I}$ (\text{resp., }$\{v_{J,I}^{\infty}(\bm\beta)\}_{J\subseteq I})$ the locally $\bQ_p$-algebraic (resp., smooth) generalized Steinberg representations of $\bL_{\beta,I}(L)$, associated to the Zelevinsky segment $\bm\beta$ and weight $\ul{\lambda}$.\;In particular,\;we call
\[\st_I^{\lalg}(\bm\beta,\underline{\lambda}):=v_{\emptyset,I}^{\lalg}(\bm\beta,\underline{\lambda}),\quad \st_I^{\infty}(\bm\beta):=v_{\emptyset,I}^{\infty}(\bm\beta),\]
the locally $\bQ_p$-algebraic (resp., smooth) Steinberg representations of $\bL_{\beta,I}(L)$.\;When $I=\Delta_{\beta}$,\;we omit the subscripts $I$ and put $\st_n^{\lalg}(\bm\beta,\underline{\lambda}):=v_{\emptyset,\Delta_{\beta}}^{\lalg}(\bm\beta,\underline{\lambda})$ (resp.,\;$\st_n^{\infty}(\bm\beta):=v_{\emptyset,\Delta_{\beta}}^{\infty}(\bm\beta)$).\;

\end{dfn}
The following standard proposition is proved in Appendix \ref{JHfactorPARA}, Lemma \ref{JHsmooth}.\;
\begin{pro}
$\{v^{\infty}_{I}(\bm\beta)\}_{I\subseteq \Delta_{\beta}}$ are precisely the Jordan-H\"{o}lder factors of  $i_{\emptyset}(\bm\beta)$ and $ \st_n^{\infty}(\bm\beta)\cong \langle\bm{\beta}\rangle^t$ (the smooth  Steinberg representation of $G$ introduced in Definition \ref{dfnparastein}).\;Therefore, the Jordan-H\"{o}lder factors of  $i_{\emptyset}(\bm\beta,\underline{\lambda})$ are $\{v^{\lalg}_{I}(\bm\beta,\ul{\lambda})\}_{I\subseteq \Delta_{\beta}}$.\;
\end{pro}

\subsection{Bernstein-Zelevinsky geometric lemma and Bruhat filtrations}\label{Brufil}
In this section, we discuss the Bruhat filtrations of parabolically induced representations and describe the associated graded representations.\;

Let $I,J\subseteq \Delta_{\beta}$.\;Using the relative Bruhat decomposition
\begin{equation}
	G=\coprod_{w\in \sW_n^{S^{\beta}(I),S^{\beta}(J)}}\op_{\beta,I}(L)w\op_{\beta,J}(L),
\end{equation}
we define an increasing filtration $\cF^{\bullet}$ on $i_I(\bm\beta)$ by $\op_{\beta,{J}}(L)$-invariant subspaces:
\begin{equation}\label{filext}
	\begin{aligned}
		\cF^hi_I(\bm\beta)=\bigg\{f\in i_I(\bm\beta)\;|\;\mathrm{supp}(f)\subseteq \bigcup_{\substack{w\in \sW_n^{S^{\beta}(I),S^{\beta}(J)}\\\mathrm{lg}(w)\leq h}}\op_{\beta,{I}}(L)\backslash \op_{\beta,{I}}(L)w\op_{\beta,{J}}(L)\bigg\},
	\end{aligned}
\end{equation}
for all $h\in\BZ_{\geq 0}$.\;By \cite[Proposition 6.3.1]{casselman1975introduction} (or a corollary of it), there are canonical isomorphisms
\[\gr^h_{\cF^{\bullet}}(i^{\beta}_{I,\Delta_{\beta}}(\bm\beta))\cong \bigoplus_{\substack{w\in \sW_n^{S^{\beta}(I),S^{\beta}(J)}\\\mathrm{lg}(w)=h}}\Big(\mathrm{ind}_{\op_{\beta,{J}}(L)\cap \op_{\beta,{I}}(L)^w}^{\op_{\beta,{J}}(L)}\pi_{\beta,I}^w\Big)^{\infty}\]
of smooth $\op_{\beta,{J}}(L)$-representations for all $h\in\BZ_{\geq 0}$, where  $\Big(\mathrm{ind}_{\op_{\beta,{J}}(L)\cap \op_{\beta,{I}}(L)^w}^{\op_{\beta,{J}}(L)}\pi_{\beta,I}^w\Big)^{\infty}$ is the smooth induction with compact support.\;By \cite[Proposition 6.3.3]{casselman1975introduction} and Lemma \ref{nonzeroIJw} below, we have isomorphisms
\begin{equation}\label{rgradfilext}
	\begin{aligned}
		J^{\beta}_{\Delta_{\beta},J}(\gr^h_{\cF^{\bullet}}(i^{\beta}_{I,\Delta_{\beta}}(\bm\beta)))\cong&\;
		\bigoplus_{\substack{w\in \sW_n^{S^{\beta}(I),S^{\beta}(J)}\\\mathrm{lg}(w)=h}}\Big(\ind_{\bL_{\beta,J}(L)\cap \op_{\beta,{I}}(L)^w}^{\bL_{\beta,J}(L)}\big(r_{^w\op_{\beta,{J}}(L)\cap \bL_{\beta,I}(L)}^{\bL_{\beta,I}(L)}\pi_{\beta,I}\big)^w\otimes_E\delta^{\beta}_{I,J,w}\Big)^{\infty}\\
		\cong&\;
		\bigoplus_{\substack{w\in \sW_n^{{\beta},I,J}\\\mathrm{lg}(w)=h}}i^{\beta}_{w^{-1}(I)\cap J,J}\big(J^{\beta,w}_{I,I\cap w(J)}(\pi_{\beta,I})\big).\;
	\end{aligned}
\end{equation}
of smooth $\bL_{\beta,{J}}(L)$-representations for all $h\in\BZ_{\geq 0}$,\;where  $\delta^{\beta}_{I,J,w}$ is the modulus character of $\op_{\beta,{J}}\cap \op_{\beta,{I}}^w$ acting on $\overline{\bN}_{\beta,J}/\overline{\bN}_{\beta,J}\cap \op_{\beta,{I}}^w$,\;and for $w\in \sW_n^{\beta,I,J}$,\;we use the shortened notation:
\begin{equation}
	\begin{aligned}
			J^{\beta,w}_{I,I\cap w(J)}(\cdot):=&\;\big(r_{^w\op_{\beta,{J}}(L)\cap \bL_{\beta,I}(L)}^{\bL_{\beta,I}(L)}(\cdot)\big)^w\otimes_E\delta^{\beta}_{I,J,w}=\big(J_{{}^w\on_{\beta,{J}}(L)\cap \bL_{\beta,I}(L)}(\cdot)\big)^w\otimes_E\delta^{\beta}_{I,J,w},\\
		i^{\beta}_{w^{-1}(I)\cap J,J}(\cdot):=&\;\big(\mathrm{Ind}_{\bL_{\beta,J}\cap \op^w_{\beta,{I}}(L)}^{\bL_{\beta,J}(L)}\cdot\big)^{\infty}.\;
	\end{aligned}
\end{equation}

The following lemmas  are contained  in \cite[Corollary 6.3.4,\;Theorem 6.3.5]{casselman1975introduction}.\;
\begin{lem}\label{nonzeroIJw}For $w\in \sW_n^{S^{\beta}(I),S^{\beta}(J)}$,  $r_{^w\op_{\beta,{J}}(L)\cap \bL_{\beta,I}(L)}^{\bL_{\beta,I}(L)}\pi_{\beta,I}=0$  unless $w\in  \sW_n^{\beta,I,J}$.\;
\end{lem}
\begin{proof} There is nothing to prove if $r=1$.\;Suppose that $r>1$.\;For $w\in \sW_n^{S^{\beta}(I),S^{\beta}(J)}\subseteq \sW_n^{S^{\beta}(I),\emptyset}$, we have an injection of smooth representations:\;
	\begin{equation}
		\begin{aligned}
			&r_{^w\op_{\beta,{J}}(L)\cap \bL_{\beta,I}(L)}^{\bL_{\beta,I}(L)}\pi_{\beta,I}\hookrightarrow r_{^w\op_{\beta,{J}}(L)\cap \bL_{\beta,I}(L)}^{\bL_{\beta,I}(L)}
			i^{\beta}_{\emptyset,I}(\pi_{\beta}).\;
		\end{aligned}
	\end{equation}
	Note that $r_{^w\op_{\beta,{J}}(L)\cap \bL_{\beta,I}(L)}^{\bL_{\beta,I}(L)}
	i^{\beta}_{\emptyset,I}(\pi_{\beta})$ admits a canonical increasing filtrartion with graded pieces
	\begin{equation*}
		\begin{aligned}
			\Big\{\Big(\ind_{\bL_{\beta,S^{\beta}(I)\cap w(S^{\beta}(J))}(L)\cap \bP_{\beta}(L)^{u} }^{\bL_{\beta,S^{\beta}(I)\cap w(S^{\beta}(J))}(L)}
			&\gamma_{I,J,u}\otimes_E\big(r_{^u(^w\op_{\beta,{J}}(L)\cap \bL_{\beta,I}(L))\cap \bL_{\beta}(L)}^{\bL_{\beta}(L)}\pi_{\beta}\big)^{u}\Big)^{\infty} \Big\}_{u\in \sW^{S^{\beta}(I)\cap w(S^{\beta}(J)),\Delta^\beta}_{S^{\beta}(I)}},
		\end{aligned}
	\end{equation*}
	where $\gamma_{I,J,u}$ is the corresponding modulus character.\;By \cite[Corollary 2.8.8]{RWcarter} (see also \cite[Corollary 6.3.4]{casselman1975introduction}), the smooth representation $r_{^u(^w\op_{\beta,{J}}(L)\cap \bL_{\beta,I}(L))\cap \bL_{\beta}(L)}^{\bL_{\beta}(L)}\pi_{\beta}$ can be nonzero only if $\bL_{\beta}(L)\subseteq {}^u(^w\op_{\beta,{J}}(L)\cap \bL_{\beta,I}(L))$, or equivalently if $\Delta^{\beta}\subset uw(S^{\beta}(J))\cap w(S^{\beta}(I))$.\;Lemma \ref{blockeva} then implies that $w^{-1}\in \sW_n^{\beta}$, and hence $w\in \sW_n^{\beta}$.\;Together with $w\in \sW_n^{S^{\beta}(I),S^{\beta}(J)}$, this gives $w\in\sW_n^{\beta,I,J}$.\;This completes the proof.\;
\end{proof}

\begin{lem}\label{blockeva}
\begin{itemize}
	\item[(i)]	If $w(\Delta^{\beta})\subseteq S^{\beta}(J)$ and $r>1$.\;Then $w(\Delta^{\beta})= \Delta^{\beta}$,\;i.e.,\;$w(\Delta^{\beta})\cap J=\emptyset$.\;
\item[(ii)] If $w(\Delta^{\beta})= \Delta^{\beta}$,\;then $w\in \sW_n^{\beta}$.\;
\end{itemize}
\end{lem}
\begin{proof}Let $\Delta^{\beta}_{i}:=\{(i-1)r+1,\cdots,ir-1\}$ be the $i$-th component of $\Delta^{\beta}$.\;Then $w(\Delta^{\beta}_{i})=\{w((i-1)r+1),w((i-1)r+1)+1,\cdots,w((i-1)r+1)+(r-2)\}$; equivalently, $w((i-1)r+j)=w((i-1)r+1)+(j-1)$ for $1\leq j\leq r$.\;Hence $w(\Delta^{\beta})=\Delta^{\beta}$ implies $w\in\sW_n^{\beta}$, proving $(ii)$.\;It remains to prove $(i)$.\;Suppose otherwise, and write $w(\Delta^{\beta})\cap J=\{j_1r<j_2r<\cdots<j_ar\}$.\;The condition $w(\Delta^{\beta})\subseteq S^{\beta}(J)$ implies that, for every $1\leq i\leq a$, there is an integers $1\leq l_i\leq r-1$ and  $1\leq b_i\leq k-1$ such that $J_i:=\{(j_i-1)r+l_i,(j_i-1)r+l_i+1,\cdots,(j_i-1)r+(r-1),j_ir,j_ir+1,\cdots,j_ir+l_i-2\}=w(\Delta^{\beta}_{b_i})$.\;Therefore, $(k-1-a)\times(r-1)=|w(\cup_{i\in \{1,2,\cdots,k-1\}\backslash\{l_1,\cdots,l_a\}}\Delta^{\beta}_{i})|\leq (\sum_{i=1}^{a+1}\max\{0,j_{i}-j_{i-1}-2\})\times(r-1)=(\sum_{i=1}^{a+1}(j_{i}-j_{i-1}-2)+s)=r(k-1-2(a+1)+s)\times(r-1)$, where $s$ is the number of indices $i$ such that $j_i=j_{i-1}+1$.\;Thus $a+2\leq s$, a contradiction.\;
\end{proof}


For $w\in \sW_n^{\beta}$, we put	
\begin{equation}\label{pilrrw}
	\begin{aligned}
		&w(\pi_{\beta}):=(\boxtimes_{i=1}^{k}\pi\nu^{k-w(i)})\otimes_E \delta_{\op_{\beta}}^{1/2}.
	\end{aligned}
\end{equation}

\begin{cor}\label{rigeolemma}There is an isomorphism of smooth $\bL_{\beta}(L)$-representations:
	\[J^{\beta}_{\Delta_\beta,\emptyset}i^{\beta}_{\emptyset,\Delta_\beta}(\pi_{\beta})\cong \bigoplus_{w\in \sW_n^{\beta}}w(\pi_{\beta}).\]
\end{cor}
\begin{proof} This result is  given in \cite[Proposition 6.4.1]{casselman1975introduction}.\;The previous argument gives $	(J^{\beta}_{\Delta_\beta,\emptyset}i^{\beta}_{\emptyset,\Delta_\beta}(\pi_{\beta}))^{\mathrm{ss}}
	\cong\bigoplus_{w\in \sW_n^{\beta}}w(\pi_{\beta})$.\;Note that all the central characters of $\{w(\pi_{\beta})\}_{w\in \sW_n^{\beta}}$ are pairwise different (since the central characters of $\{\nu^{k-i}\}_{1\leq i\leq k}$ are pairwise different).\;Therefore,\;we deduce from Lemma \ref{centervan} that
	$\ext^{1,\infty}_{\bL_{\beta}(L)}\big(w(\pi_{\beta}),w'(\pi_{\beta})\big)=0$ for any pair of $w\neq w'\in \sW_n^{\beta}$.\;Then $J^{\beta}_{\Delta_\beta,\emptyset}i^{\beta}_{\emptyset,\Delta_\beta}(\pi_{\beta})$ is semi-simple and hence is isomorphic to $\bigoplus_{w\in \sW_n^{\beta}}w(\pi_{\beta})$.\;
\end{proof}

\section{Locally \texorpdfstring{$\bQ_p$}{Lg}-analytic generalized Steinberg representations}

The main goal of this section is to study locally $\bQ_p$-analytic generalized Steinberg representations.\;In Section \ref{osfunctor}, we recall the theory of Orlik-Strauch.\;In Section \ref{dfnanaparastrep}, we define locally $\bQ_p$-analytic generalized Steinberg representations.\;In order to study the Jordan-H\"{o}lder factors of locally $\bQ_p$-analytic generalized Steinberg representation, we first establish the locally $\bQ_p$-algebraic and locally $\bQ_p$-analytic Tits complexes in Section \ref{Titcomp}.\;In Section \ref{JHofanaparastrep}, we compute the multiplicities of Jordan-H\"{o}lder factors of locally $\bQ_p$-analytic  Steinberg representations by following the approach of \cite{orlik2014jordan}.\;

\subsection{Orlik-Strauch theory}\label{osfunctor}

Let $\cO_{\alge}^{\overline{\fp}_{I,L}}$ be the Bernstein-Gelfand-Gelfand (BGG) category (see \cite{orlik2015jordan}; see also \cite[Section 2]{breuil2016socle}).\;Recall that if $\underline{{\lambda}}\in X^+_{I}$, then the generalized parabolic Verma module $\overline{M}_I(-\underline{{\lambda}})$ belongs to the BGG category $\cO^{\overline{{\fp}}_{I,L}}_{\alge}$.\;Furthermore, by \cite[Theorem 3.37]{bergdall2018adjunction}, all irreducible constituents of $\overline{M}_I(-\underline{{\lambda}})\in \cO^{\overline{{\fp}}_{I,L}}_{\alge}$ in $\cO^{\overline{{\fp}}_{I,L}}_{\alge}$ have the form $\overline{L}(-\underline{\lambda}')$ such that  $\underline{\lambda}'$ is strongly linked to $\underline{\lambda}$ and $\underline{\lambda}'\in X^+_I$ (see \cite[Definition 3.31]{bergdall2018adjunction} for the notion of strongly linked).\;Let $I'$ be a subset of $\Delta$ containing $I$, then $\cO_{\alge}^{\overline{\fp}_{I',L}}$ is a full subcategory of $\cO_{\alge}^{\overline{\fp}_{I,L}}$.\;Therefore, for any object $M\in \cO_{\alge}^{\overline{\fp}_{I,L}}$, there is a maximal subset $I'\subseteq \Delta$ such that $M\in \cO_{\alge}^{\overline{\fp}_{I',L}}$.\;We  say that $I'$ is maximal for $M$.\;

For $\sigma\in\Sigma_L$ and $w_\sigma\in \sW_{n,\sigma}$,\;let $D_L(w_{\sigma})=\{{i}\in \Delta |\; l(s_{i,\sigma}w_\sigma)<l(w_\sigma)\}$ and $I(w_\sigma)=\Delta\backslash D_L(w_\sigma)$, where $l(\cdot)$ is the length function.\;Let $\mathrm{supp}(w_\sigma)$ be the set of simple reflections appearing in one (and so in any) reduced expression of $w_\sigma$.\;Let $w=(w_\sigma)_{\sigma\in \Sigma_L}\in\sW_{n,L}$.\;We put
\[I(w)=\cap_{\sigma\in \Sigma_L}I(w_\sigma),\quad \mathrm{supp}(w)=\cup_{\sigma\in \Sigma_L}\mathrm{supp}(w_\sigma).\]
Let $w\in \sW_{n,L}$ and $\underline{{\lambda}}\in X^+_{\Delta}$.\;Then $\lM(-w\cdot  \underline{\lambda})\in \cO_{\alge}^{{\overline{\fp}}_{I,L}}$
if and only if $I\subseteq I(w_\sigma)$ for all $\sigma\in \Sigma_{L}$.\;Therefore,\;$I(w)$ is  maximal for $\lM(-w\cdot  \underline{\lambda})$.\;

Let $I$ be a subset of $\Delta$, and let $\underline{\lambda}=(\underline{{\lambda}}_{\sigma})_{\sigma\in\Sigma_L}\in X^+_{I}$.\;We put
\[\overline{M}_{I,\sigma}(-\underline{{\lambda}}_\sigma):=U(\mathfrak{g}_{\sigma})\otimes_{U(\overline{\mathfrak{p}}_{I,\sigma})}\overline{L}(-\underline{\lambda}_\sigma)_{I}, \quad
\overline{M}_{\sigma}(-\underline{{\lambda}}_\sigma):=\overline{M}_{\emptyset,\sigma}(-\underline{{\lambda}}_\sigma).\]
Then we have
\[\overline{M}_I(-\underline{{\lambda}})\cong \bigotimes_{\sigma\in \Sigma_L}\overline{M}_{I,\sigma}(-\underline{{\lambda}}_\sigma),\quad \overline{M}(-\underline{{\lambda}})\cong \bigotimes_{\sigma\in \Sigma_L}\overline{M}_{\sigma}(-\underline{{\lambda}}_\sigma).\]

For each $w_\sigma\in \sW_{n,\sigma}$ and $\sigma\in \Sigma_L$, fix once
and for all an embedding $i_{w_\sigma}:\overline{M}_{\sigma}(-w_\sigma\cdot \underline{\lambda}_\sigma)\hookrightarrow \overline{M}_{\sigma}(-\underline{\lambda}_\sigma)$.\;If $w_\sigma'\geq w_\sigma$ (for the Bruhat order on $\sW_{n,\sigma}$), we have a unique map $i_{w'_\sigma,w_\sigma}:\overline{M}_{\sigma}(-w'_\sigma\cdot \underline{\lambda}_\sigma)\rightarrow \overline{M}_{\sigma}(-w_\sigma\cdot \underline{\lambda})$ such that $i_{w'_\sigma}= i_{w_\sigma}\circ i_{w_{\sigma'},w_\sigma}$ (\cite[Page 653]{orlik2014jordan}).\;By the argument in  \cite[Proposition 2.1]{GeneralizedVerma} or \cite[Section 2, Page 653]{orlik2014jordan}, we have a surjective map $q_{I}:\overline{M}(-\underline{{\lambda}})\longrightarrow \overline{M}_{I}(-\underline{{\lambda}})$ and an exact sequence, 
\begin{equation}\label{qIvermamodule}
	\bigoplus_{\substack{w\in \sW_{I,L}\\l(w)=1}}\overline{M}(-w\cdot\underline{{\lambda}})\xrightarrow{	\bigoplus_w\otimes_{\sigma\in \Sigma_L}i_{w_{\sigma}}} \overline{M}(-\underline{{\lambda}})\longrightarrow \overline{M}_{I}(-\underline{{\lambda}})\longrightarrow 0.
\end{equation}
If $I'\subset I\subseteq \Delta$, then we obtain a transition map $q_{I',I}:\overline{M}_{I'}(-\underline{{\lambda}})\longrightarrow \overline{M}_{I}(-\underline{{\lambda}})$ such that $q_{I}=q_{I',I}\circ q_{I'}$.\;

The Orlik-Strauch functor (see \cite[Theorem]{orlik2015jordan}), see also \cite[Section 2]{breuil2016socle} has the form:
\begin{equation*}
	\begin{aligned}
		\cF^{G}_{\op_{{I}}}(-,-):\cO_{\alge}^{\overline{\fp}_{{{I}}},L}\times \Rep^{\infty,\adm}_{E}(\bL_{{I}}(L))&\longrightarrow \mathbf{{Rep}}_E(G),\\
		(M,\pi_I)&\mapsto\cF^{G}_{\op_{{I}}}(M,\pi_{I}),
	\end{aligned}
\end{equation*}
where $\Rep^{\infty,\adm}_{E}(\bL_{{I}}(L))$ is the subcategory of $\Rep^{\infty}_{E}(\bL_{{I}}(L))$ of admissible smooth representations of $\bL_{{I}}(L)$ over $E$.\;By \cite[Theorem 2.3]{breuil2016socle}, or \cite[Theorem]{orlik2015jordan}, this functor satisfies the following properties.\;$\cF^{G}_{\op_{{I}}}(-,-)$ is an exact bifunctor, which is contravariant (resp.\;covariant) in the first (resp.\;second) variable (see \cite[Propositions 4.7 and 4.9]{orlik2015jordan}).\;If $M\in \cO_{\alge}^{{\overline{\fp}}_{I,L}}$ is simple, and $\op_{I}$ is the maximal parabolic subgroup associated to $M$, then the representation $\cF^{G}_{\op_{{I}}}(M,\pi_{I})$ is topologically irreducible for any  irreducible $\pi_{I}\in \Rep^{\infty,\adm}_{E}(\bL_{{I}}(L))$ .\;Let $I\subseteq J$ and let $M\in \cO_{\alge}^{{\overline{\fp}}_{J,L}}$.\;Then for any $\pi_{I}\in \Rep^{\infty,\adm}_{E}(\bL_{{I}}(L))$,\;we have an isomorphism of $G$-representations $\cF^{G}_{\op_{I}}(M,\pi_{I})\cong \cF^{G}_{\op_{J}}\big(M, \big(\mathrm{Ind}_{\op_{J}(L)\cap \bL_{I}(L)}^{\bL_{J}(L)}\pi_I\big)^{\infty}\big)$.\;

\subsection{Locally \texorpdfstring{$\bQ_p$}{Lg}-analytic generalized Steinberg representations}\label{dfnanaparastrep}
Let $\underline{\lambda}\in X_{\Delta}^+$, and let $I,J\subseteq \Delta_{\beta}$.\;We put  $\pi_{\beta,I}(\underline{\lambda}):=\pi_{\beta,I}\otimes_EL(\underline{\lambda})_{S^{\beta}(I)}$ and put
\begin{equation*}
\begin{aligned}
	\bI_I(\bm\beta,\underline{\lambda})&=\Big(\mathrm{Ind}_{\op_{\beta,{I}}(L)}^{G}\pi_{\beta,I}(\underline{\lambda})\Big)^{\bQ_p-\mathrm{an}}.\\
\end{aligned}
\end{equation*}
If $J\supsetneq I$, then by \cite[Theorem]{orlik2015jordan},\;Proposition \ref{axioms}  induces an injection
\begin{equation}\label{injINANAN}
\begin{aligned}		p^{\ana}_{J,I}:\bI_J(\bm\beta,\underline{\lambda})=\mathcal{F}^{G}_{\op_{\beta,{J}}}\big(\overline{M}_{S^{\beta}(J)}(-\underline{{\lambda}}),\pi_{\beta,J}\big)\xrightarrow{\mathcal{F}_{\op_{\beta,{J}}}^{G}(q_{I,J}, p^{\infty}_{J,I})}&\;\mathcal{F}^{G}_{\op_{\beta,J}}\big(\overline{M}_{S^{\beta}(I)}(-\underline{{\lambda}}),i^{\beta}_{I,J}(\pi_{\beta,I})\big)\\
	&\cong\mathcal{F}^{G}_{\op_{\beta,I}}\big(\overline{M}_{S^{\beta}(I)}(-\underline{{\lambda}}),\pi_{\beta,I}\big)=\bI_I(\bm\beta,\underline{\lambda}).\;
\end{aligned}
\end{equation}

\begin{dfn}\label{dfnlagpstrep}\textbf{(Locally $\bQ_p$-analytic generalized Steinberg representations)}\\
For $\underline{\lambda}\in X_{\Delta}^+$,\;put
\begin{equation*}
	\begin{aligned}
		&v^{\ana}_I(\bm\beta,\underline{\lambda})=\bI_I(\bm\beta,\underline{\lambda})\big/\sum_{I\subsetneq J\subseteq \Delta_{\beta}}\bI_J(\bm\beta,\underline{\lambda}).
	\end{aligned}
\end{equation*}
we call $\{v^{\ana}_I(\bm\beta,\underline{\lambda})\}_{I\subseteq \Delta_{\beta}}$ the locally $\bQ_p$-analytic generalized Steinberg representations of $G$ associated to the Zelevinsky-segment $\bm\beta$ and weight $\ul{\lambda}$.\;Denoted by $\st_n^{\ana}(\bm\beta,\ul{\lambda}):=v_{\emptyset}^{\ana}(\bm\beta,{\lambda})$ the locally $\bQ_p$-analytic Steinberg representations.\;Furthermore, for $J\subset I\subseteq \Delta_{\beta}$ and $\underline{\lambda}\in X_{S^{\beta}(I)}^+$.\;Put
\begin{equation*}
	\begin{aligned}
		\bI_{I,J}(\bm\beta,\underline{\lambda})
		:=\left(\mathrm{Ind}_{\op_{\beta,{J}}(L)\cap \bL_{\beta,I}(L)}^{\bL_{\beta,I}(L)}\pi_{\beta,J}(\underline{\lambda})\right)^{\bQ_p-\ana}.\;
	\end{aligned}
\end{equation*}
The locally $\bQ_p$-analytic generalized Steinberg representations of $\bL_{\beta,I}(L)$ are given by 
\begin{equation*}
	\begin{aligned}
		v_{J,I}^{\ana}(\bm\beta,\underline{\lambda})=\bI_{I,J}(\bm\beta,\underline{\lambda})\big/\sum_{J \subsetneq K\subseteq I}\bI_{I,K}(\bm\beta,\underline{\lambda}).\;
	\end{aligned}
\end{equation*}
We call $\st_I^{\ana}(\bm\beta,\ul{\lambda}):=v_{\emptyset,I}^{\ana}(\bm\beta,\underline{\lambda})$ the locally $\bQ_p$-analytic Steinberg representation of $\bL_{\beta,I}(L)$.\;
\end{dfn}

Note that the locally $\bQ_p$-algebraic vectors in
$v^{\ana}_{I}(\bm\beta,\ul{\lambda})$ form $v^{\lalg}_{I}(\bm\beta,\ul{\lambda})$.\;

\begin{exmp}\label{exampleBorel} (Borel case) Let $r=1$ and $\pi=|\cdot|_L^{\frac{1-n}{2}}$.\;Then $
		\bm\beta=\{|\cdot|_L^{\frac{n-1}{2}},\;\cdots,\;|\cdot|_L^{\frac{1-n}{2}+1},|\cdot|_L^{\frac{1-n}{2}}\}$ and $\pi_{\beta}$ is trivial representation.\;We recover the locally $\bQ_p$-analytic generalized (Borel) Steinberg representations defined in \cite[Section 2.1.2]{2019DINGSimple}.\;More precisely,\;we have (keep the notation in \cite[Section 2.1.2]{2019DINGSimple}) $v^{\ana}_I(\bm\beta,\ul{\lambda})=v_{I}^{\ana}(\ul{\lambda})$,\;$v_I^{\lalg}(\bm\beta,\ul{\lambda})=v^{\infty}_{I}(\ul{\lambda})$,\;$\st_{n}^{\ana}(\bm\beta,\ul{\lambda})=\st_{n}^{\ana}(\ul{\lambda})$ and 
$\st_{n}^{\lalg}(\bm\beta,\ul{\lambda})=\st_{n}^{\infty}(\ul{\lambda})$,\;etc.
\end{exmp}

\subsection{The locally \texorpdfstring{$\bQ_p$}{Lg}-algebraic and locally \texorpdfstring{$\bQ_p$}{Lg}-analytic Tits complexes}\label{Titcomp}
In this section,\;we prove the exactness of the locally $\bQ_p$-algebraic and locally $\bQ_p$-analytic Tits complexes.\;These complexes are used to compute the multiplicities of irreducible constituents in the locally $\bQ_p$-analytic generalized Steinberg representation $v_{S}^{\ana}(\bm\beta,\underline{\lambda})$ (see Section \ref{JHofanaparastrep}).\;In Sections \ref{JHofanaparastrep} and \ref{sectionextana}, they will also be used to compute extension groups between certain locally $\bQ_p$-analytic representations.\;

Let $\underline{{\lambda}}\in X^+_{\Delta}$ and $I\subseteq\Delta_{\beta}$.\;By \cite[Section 9.16,\;Theorem]{humphreysBGG},\;there exists a $S^{\beta}(I)$-parabolic BGG resolution of $\overline{L}(-\underline{{\lambda}})$:
\begin{equation}\label{BGGRES2}
	\begin{aligned}
		0\rightarrow \overline{M}_{S^\beta(I)}(-w_{S^{\beta}(I),\emptyset}\cdot\underline{{\lambda}})&\rightarrow\bigoplus_{\substack{w\in
				\sW^{S^{\beta}(I),\emptyset}_{n,L}\\l(w)=l(w_{S^{\beta}(I),\emptyset})-1}}\overline{M}_{S^\beta(I)}(-w\cdot\underline{{\lambda}})\rightarrow\cdots\\
		&\cdots\rightarrow\bigoplus_{\substack{w\in \sW^{S^{\beta}(I),\emptyset}_{n,L}\\l(w)=1}}\overline{M}_{S^\beta(I)}(-w\cdot\underline{{\lambda}})\rightarrow \overline{M}_{S^\beta(I)}(-\underline{{\lambda}})\rightarrow \overline{L}(-\underline{\lambda})\rightarrow 0,
	\end{aligned}
\end{equation}
where $w_{S^{\beta}(I),\emptyset}$ is the longest element of $\sW^{S^{\beta}(I),\emptyset}_{n,L}$, and the differentials in (\ref{BGGRES2}) are described at the end of \cite[Section 2]{orlik2014jordan} as alternating sums of the maps in $\{i_{w_{\sigma}',w_\sigma}\}_{w_{\sigma}',w_\sigma\in \sW_{n,\sigma},\sigma\in \Sigma_L}$.\;Applying the exact functor $\mathcal{F}^{G}_{\op_{\beta,I}}(-,\pi_{\beta,I})$ yields the exact sequence
\begin{equation}\label{BGGRESfunctor}
\begin{aligned}
	   0\leftarrow\bI_I(\bm\beta,w_{S^{\beta}(I),\emptyset}\cdot\underline{{\lambda}})&\leftarrow\bigoplus_{\substack{w\in
			\sW^{S^{\beta}(I),\emptyset}_{n,L}\\l(w)=l(w_{S^{\beta}(I),\emptyset})-1}}\bI_I(\bm\beta,w\cdot\underline{{\lambda}})\leftarrow\cdots\\
	&\cdots\leftarrow\bigoplus_{\substack{w\in \sW^{S^{\beta}(I),\emptyset}_{n,L}\\l(w)=1}}\bI_I(\bm\beta,w\cdot\underline{{\lambda}})\leftarrow \bI_I(\bm\beta,\underline{{\lambda}})\leftarrow i_I(\bm\beta,\underline{{\lambda}})\leftarrow 0.
\end{aligned}
\end{equation}

The following locally $\bQ_p$-analytic (resp., locally $\bQ_p$-algebraic) Tits complexes show that $v^{\ana}_{I}(\bm\beta,\ul{\lambda})$ is represented by a complex whose terms are the locally $\bQ_p$-analytic (resp., locally $\bQ_p$-algebraic) principal series $\{\bI_I(\bm\beta,\ul{\lambda})\}_{I\subseteq \Delta_{\beta}}$ (resp., $\{i_I(\bm\beta,\ul{\lambda})\}_{I\subseteq \Delta_{\beta}}$).\;Define
\begin{equation}
	\begin{aligned}
		\bC_{I,\bullet}^{\lalg}(\bm\beta,\ul{\lambda}):=&\;\big[C^{\lalg}_{I,l_I-1} \rightarrow C^{\lalg}_{I,l_I-2} \rightarrow \cdots \rightarrow C^{\lalg}_{I,1} \rightarrow i_I(\bm\beta,\underline{\lambda})\big],\quad C^{\lalg}_{I,j}:=\bigoplus_{\substack{I\subseteq K\subseteq \Delta_{\beta}\\|K\backslash I|=j}}i_K(\bm\beta,\underline{\lambda})\\
		\bC_{I,\bullet}^{\ana}(\bm\beta,\ul{\lambda}):=&\;\big[C^{\ana}_{I,l_I-1} \rightarrow C^{\ana}_{I,l_I-2} \rightarrow\cdots \rightarrow C^{\ana}_{I,1} \rightarrow \bI_I(\bm\beta,\underline{\lambda}) \big],\quad C^{\ana}_{I,j}:=\bigoplus_{\substack{I\subseteq K\subseteq \Delta_{\beta}\\|K\backslash I|=j}}\bI_K(\bm\beta,\underline{\lambda}).
	\end{aligned}
\end{equation}
where $C^{\lalg}_{I,j}$ (resp.,\;$C^{\ana}_{I,j}$) is placed in degree $-j$ for $0\leq j\leq l_I-1$ (recall $l_I=k-|I|$).\;The differentials
\[d^{\infty}_{K',K}:i_{K'}(\bm\beta,\underline{\lambda})\rightarrow i_K(\bm\beta,\underline{\lambda}),\quad\text{resp.,\;} d^{\ana}_{K',K}:\bI_{K'}(\bm\beta,\underline{\lambda})\rightarrow \bI_K(\bm\beta,\underline{\lambda})\]
are defined as follows.\;If $K\nsubseteq K'$, $d^{\infty}_{K',K}=0$ (resp.,\;$d^{\ana}_{K',K}=0$).\;If $K'=\{k_1r<\cdots<k_lr\}$ and $K'=K\cup\{k_ir\}$, we put $d^{\infty}_{K',K}=(-1)^ip^{\infty}_{K',K}$ (resp.,\;$d^{\ana}_{K',K}=(-1)^ip^{\ana}_{K',K}$), where $p^{\infty}_{K',K}:i_{K'}(\bm\beta,\underline{\lambda})\hookrightarrow i_K(\bm\beta,\underline{\lambda})$ is the injection defined in Proposition \ref{axioms} (resp.,\;$p^{\ana}_{K',K}:\bI_{K'}(\bm\beta,\underline{\lambda})\hookrightarrow \bI_K(\bm\beta,\underline{\lambda})$) is the injection defined in (\ref{injINANAN})).\;

\begin{rmk}When $r=1$ and $\pi=|\cdot|_L^{\frac{1-n}{2}}$, we recover the locally $\bQ_p$-algebraic and locally $\bQ_p$-analytic Tits complexes given in \cite[Section 2.1.2]{2019DINGSimple}.\;
\end{rmk}

\begin{pro}\label{smoothexact} Let $I\subseteq \Delta_{\beta}$ and $\underline{\lambda}\in X_{\Delta}^+$.\;The  surjection $i_I(\bm\beta,\underline{\lambda}) \twoheadrightarrow v^{\lalg}_I(\bm\beta,\underline{\lambda})$ extends to a quasi-isomorphism $\bC_{I,\bullet}^{\lalg}(\bm\beta,\ul{\lambda})\xrightarrow{\sim}v^{\lalg}_I(\bm\beta,\underline{\lambda})$.\;
	\end{pro}
\begin{proof}
 The acyclicity of the locally $\bQ_p$-algebraic Tits complexes follows from Proposition \ref{axioms} and the combinatorial argument used in \cite[Proposition 11]{2012Orlsmoothextensions}; see also \cite[Section 2, Proposition 6]{schneider1991cohomology}.\;
\end{proof}

There is also a locally analytic analogue of the preceding result, generalizing \cite[Theorem 4.2]{orlik2014jordan} and \cite[Theorem 2.6]{2019DINGSimple}.\;We adapt the proof of \cite[Theorem 4.2]{orlik2014jordan} to the present setting.\;
\begin{pro}\label{analyticexact}
	Let $I\subseteq \Delta_{\beta}$ and $\underline{\lambda}\in X_{\Delta}^+$.\;The surjection $\bI_I(\bm\beta,\underline{\lambda}) \twoheadrightarrow v^{\ana}_I(\bm\beta,\underline{\lambda})$ extends to a quasi-isomorphism $\bC_{I,\bullet}^{\ana}(\bm\beta,\ul{\lambda})\xrightarrow{\sim}v^{\ana}_I(\bm\beta,\underline{\lambda})$.\;
\end{pro}

\begin{proof}We first prove the proposition for $\underline{\lambda}=\underline{0}$ and write $\bI_{I}(\bm\beta):=\bI_{I}(\bm\beta,\underline{0})$ for short.\;By the same combinatorial technique as in \cite[Proposition 11]{2012Orlsmoothextensions} (see also \cite[Section 2, Proposition 6]{schneider1991cohomology}), it suffices to prove that
\begin{equation}\label{analytic[A3]}
	\begin{aligned}
		&\bI_{I}(\bm\beta)\cap \bI_{J}(\bm\beta)=\bI_{I\cup J}(\bm\beta), \quad\bI_{I}(\bm\beta)\cap\big(\sum_{j=1}^m\bI_{I_j}(\bm\beta)\big)=\sum_{j=1}^m\bI_{I}(\bm\beta)\cap \bI_{I_j}(\bm\beta).
	\end{aligned}
\end{equation}
\textbf{Proof of the first identity of (\ref{analytic[A3]}).\;}\\
By the exactness of the induction functor $\big(\mathrm{Ind}_{\op_{\beta,{I\cup J}}(L)}^{G}\cdot\big)^{\bQ_p-\ana}$, it suffices to establish the identity 
\[\Big(\mathrm{Ind}^{\bL_{\beta,I\cup J}(L)}_{\op_{\beta,{I}}(L)\cap\bL_{\beta,I\cup J}(L)}\pi_{\beta,I}\Big)^{\bQ_p-\ana}\cap \Big(\mathrm{Ind}^{\bL_{\beta,I\cup J}(L)}_{\op_{\beta,{J}}(L)\cap\bL_{\beta,I\cup J}(L)}\pi_{\beta,J}\Big)^{\bQ_p-\ana}=\pi_{\beta,I\cup J}.\]
The right-hand side is clearly contained in the left-hand side.\;Let $\mathcal{X}$ be an irreducible constituent of the cokernel.\;By \cite[Theorem]{orlik2015jordan}, $\mathcal{X}$ is of the form $\mathcal{F}^{\bL_{\beta,I\cup J}}_{\bL_{\beta,I\cup J}\cap\op_{\beta,{K}}}(\overline{L}\big(-s\cdot\underline{0})_{S^\beta(I\cup J )},M\big)$,\;where $s\in \sW_{n,L}$ satisfies $s\cdot\underline{0}\in X^+_{S^{\beta}(K)}$ for some maximal subset $K$ with $I\subseteq K\subseteq I\cup J$, and $M$ is an irreducible constituent of $i^{\beta}_{I,K}(\pi_{\beta,I})$.\;On the other hand, $\mathcal{X}$ is also of the form
$\mathcal{F}^{\bL_{\beta,I\cup J}}_{\bL_{\beta,I\cup J}\cap\op_{\beta,{K'}}}\big(\overline{L}(-s'\cdot\underline{0})_{S^\beta(I\cup J )},M'\big)$, where $s'\in \sW_{n,L}$ satisfies $s'\cdot\underline{0}\in X^+_{S^{\beta}(K')}$ for some maximal subset $K'$ with $J\subseteq K'\subseteq I\cup J$, and $M'$ is an irreducible constituent of $i^{\beta}_{J,K'}(\pi_{\beta,J})$.\;By \cite[Corollary 2.7]{breuil2016socle}, we have $K=K'$, $s=s'$ and $M'\cong M$.\;Therefore, we deduce that $s=s'\in [\sW_{ S^{\beta}(K),L}\backslash\sW_{S^{\beta}(I\cup J),L}]\cap [\sW_{ S^{\beta}(K'),L}\backslash\sW_{S^{\beta}(I\cup J),L}]=1$ since $K'\cup K=I\cup J$ (since we have $I\subseteq K\subseteq I\cup J$ and $J\subseteq K'\subseteq I\cup J$),\;here $ [\sW_{ S^{\beta}(\ast),L}\backslash\sW_{S^{\beta}(I\cup J),L}]$  minimal-length representatives for $\sW_{ S^{\beta}(\ast),L}\backslash\sW_{S^{\beta}(I\cup J),L}$  for $\ast\in\{K,K'\}$.\;In this case, we see that $M=M'$ is an irreducible constituent of $i^{\beta}_{I,K}(\pi_{\beta,I})$ and $i^{\beta}_{I,J}(\pi_{\beta,J})$.\;Then, by Proposition  \ref{axioms} \textbf{[A2]} and \cite[Proposition 1.5 (b)]{bernstein1977induced1}, we deduce $M\cong M'=\pi_{\beta,I\cup J}$.\;This proves the first identity.\;\\
\textbf{Proof of the second identity of (\ref{analytic[A3]}).\;}\\
We prove the second identity by induction on $m$.\;The first identity shows that the right-hand side of the second identity is $\bI_{I\cup I_1}(\bm\beta)+\cdots+\bI_{I\cup I_m}(\bm\beta)$.\;Set $\cV=\bI_{I_1}(\bm\beta)+\cdots+\bI_{I_{m-1}}(\bm\beta)$ and $\cW=\bI_{I\cup I_1}(\bm\beta)+\cdots+\bI_{I\cup I_{m-1}}(\bm\beta)$.\;The induction hypothesis gives the following commutative diagram:
\begin{equation}
	\xymatrix@R-0.2pc@C-1pc{0 \ar[r] &  \sum\limits_{j=1}^{m-1}\bI_{I\cup I_j\cup I_m}(\bm\beta) \ar[d]^{\simeq}  \ar[r] & \bI_{I\cup I_m}(\bm\beta)\bigoplus\cW  \ar[d]^{\simeq}  \ar[r] & \bI_{I\cup I_m}(\bm\beta)+\cW \ar@{^(->}[d] \ar[r] & 0  \\
		0 \ar[r] & \bI_{I}(\bm\beta)\cap \sum\limits_{j=1}^{m-1}\bI_{I_j\cup I_m}(\bm\beta)  \ar[r] & \bI_{I\cup I_m}(\bm\beta)\bigoplus \big(\bI_{I}(\bm\beta)\cap \cV\big) \ar[r] & \bI_{I}(\bm\beta)\cap\sum\limits_{j=1}^{m}\bI_{I_j}(\bm\beta)\ar[r] & 0.}
\end{equation}
The result follows from the snake lemma.\;\\
\textbf{For general $\underline{\lambda}$}, we follow the proof of \cite[Theorem 4.2]{orlik2014jordan}, arguing by induction on $|\Delta_{\beta}|$.\;The case $|\Delta_{\beta}|=1$ is immediate, so assume that $|\Delta_{\beta}|>1$.\;For brevity, write
\begin{equation}
	\begin{aligned}
\quad \cI_{I,d,l}:=\bigoplus\limits_{\substack{I\subseteq K\subseteq \Delta_{\beta}\\|K\backslash I|=d}}\bigoplus_{\substack{w\in\sW^{S^{\beta}(K),\emptyset}_{n,L}\\l(w)=l}}\bI_K(\bm\beta,w\cdot\underline{{\lambda}}),\\
	\end{aligned}
\end{equation}
for short (recall that $\cI_{I,d,0}=C^{\ana}_{I,d}$).\;The complexes in (\ref{BGGRESfunctor}), for various $I\subseteq K\subseteq \Delta_{\beta}$, induce the following double complex (note that the vertical sequences are given by (\ref{BGGRESfunctor}), for various $I\subseteq K\subseteq \Delta_{\beta}$):
\[\xymatrix{0 \ar[r]  &  i_{\Delta_{\beta}}(\bm\beta,\underline{\lambda})\ar[d]^{=} \ar[r]  &   C^{\lalg}_{I,i_0-1}   \ar[d] \ar[r]  &  \cdots \ar[r] &   C^{\lalg}_{I,1}    \ar[d] \ar[r] &   C^{\lalg}_{I,0} \ar[d] \\
	0  \ar[r]&  \bI_{\Delta_{\beta}}(\bm\beta,\underline{\lambda})\ar[d]   \ar[r]&   C^{\ana}_{I,i_0-1}   \ar[d] \ar[r]&  \cdots  \ar[r]&   C^{\ana}_{I,1}    \ar[d] \ar[r]&   C^{\ana}_{I,0} \ar[d] \\
	0 \ar[r]&    0  \ar[d]  \ar[r] &   \cI_{I,i_0-1,1}   \ar[d] \ar[r]&  \cdots  \ar[r]&   \cI_{I,1,1}    \ar[d] \ar[r]&   \cI_{I,0,1} \ar[d]\\
	0 \ar[r]&    0  \ar[d]   \ar[r]&   \cI_{I,i_0-1,2}   \ar[d]\ar[r]&  \cdots  \ar[r]&   \cI_{I,1,2}    \ar[d] \ar[r]&   \cI_{I,0,2} \ar[d]\\
	0 \ar[r]&    \vdots   \ar[d]  \ar[r]&   \vdots   \ar[d] \ar[r]&  \cdots \ar[r] & \vdots  \ar[d] \ar[r]&   \vdots   \ar[d]\\
	0 \ar[r]&    0  \ar[d]   \ar[r]&    0   \ar[d] \ar[r]&  \cdots  \ar[r]&   \cI_{I,1,t}    \ar[d] \ar[r]&   \cI_{I,0,t} \ar[d]\\
	0 \ar[r]&    0    \ar[r]&    0    \ar[r]&  \cdots  \ar[r]&   0  \ar[r]&  \cI_{I,0,t+1} ,}\]
where $t=l(w_{S^{\beta}(I),\emptyset})-1$ (so that $t+1=l(w_{S^{\beta}(I),\emptyset})$) and $i_0=|\Delta_{\beta}\backslash I|$.\;To prove the commutativity of this diagram, the key is to establish the following commutative diagram
\begin{equation}\label{doublecomplexverma}
	\xymatrix{ \overline{M}_{S^{\beta}(K')}(-w'\cdot\underline{{\lambda}}) \ar[d]_{i_{w',w}}\ar[rr]^{q_{S^{\beta}(K'),S^{\beta}(K)}} &  & \overline{M}_{S^{\beta}(K)}(-w'\cdot\underline{{\lambda}}) \ar[d]^{i_{w',w}} \\  \overline{M}_{S^{\beta}(K')}(-w\cdot\underline{{\lambda}})\ar[rr]^{q_{S^{\beta}(K'),S^{\beta}(K)}} &  & \overline{M}_{S^{\beta}(K)}(-w\cdot\underline{{\lambda}}),}
\end{equation}
for any $K'\subset K\subseteq \Delta_{\beta}$ with $|K'|=|K|-1$.\;Since $w'\geq w$ and $l(w')=l(w)+1$, there exists $\sigma\in \Sigma_L$ such that $w'_{\sigma'}=w_{\sigma'}$ for ${\sigma'}\neq \sigma$, while $w'_{\sigma}\geq w_{\sigma}$ and $l(w'_{\sigma})=l(w_{\sigma})+1$.\;The commutativity of (\ref{doublecomplexverma}) follows by applying the argument for the analogous diagram in \cite[Theorem 4.2, p. 662]{orlik2014jordan} to this $\sigma$-component.\;Applying the functor $\mathcal{F}^G_{\op_{\beta,{K'}}}(-,-)$ to (\ref{doublecomplexverma}), we obtain from \cite[Theorem (iii)]{orlik2015jordan} the commutative diagram
\[\xymatrix@R+2pc@C+4pc{ \bI_K(\bm\beta,w\cdot\underline{\lambda}) \ar[r]^{p^{\ana}_{K,K'}} \ar[d]_{\mathcal{F}_{\op_{\beta,{K}}}^{G}(i_{w',w},-)} &  \ar[d]^{\mathcal{F}_{\op_{\beta,{K'}}}^{G}(i_{w',w},-)} \bI_{K'}(\bm\beta,w\cdot\underline{\lambda})\\  \bI_K(\bm\beta,w'\cdot\underline{\lambda}) \ar[r]^{p^{\ana}_{K,K'}} &  \bI_{K'}(\bm\beta,w'\cdot\underline{\lambda})}\]
These commutative diagrams, for various $K'\subset K\subseteq \Delta_{\beta}$, establish the commutativity of the above double complex.\;

As in the proof of \cite[Theorem 4.2, Page 663]{orlik2014jordan}, it suffices to prove that each row of this double complex is exact except the second row and the rightmost column.\;The first row is exact by Proposition \ref{smoothexact}.\;It remains to show the exactness of the following sequence
\begin{equation}\label{analyticBTseq1}
	0\rightarrow \bI_{I(w)}^{G}(\bm\beta,w\cdot\underline{\lambda}) \rightarrow \cdots \rightarrow \bigoplus_{\substack{I\subseteq K\subseteq {I(w)}\\|K\backslash I|=1}}\bI_K(\bm\beta,w\cdot\underline{\lambda}) \rightarrow \bI_I(\bm\beta,w\cdot\underline{\lambda}) \rightarrow v^{\ana}_I(\bm\beta,w\cdot\underline{\lambda}) \rightarrow 0
\end{equation}
for each $w\in \sW^{S^{\beta}(I),\emptyset}_{n,L}$.\;Note that
\[\bI_K(\bm\beta,w\cdot\underline{\lambda})\cong \Big(\mathrm{Ind}_{\op_{\beta,{I(w)}}(L)}^{G}\bI_{K,I(w)}(\bm\beta,w\cdot\underline{\lambda})\Big)^{\bQ_p-\mathrm{an}}.\]
Therefore, the exactness of the induction functor $\big(\mathrm{Ind}_{\op_{\beta,{I(w)}}(L)}^{G}\cdot\big)^{\bQ_p-\ana}$ implies that it suffices to establish the exactness of the sequence
\begin{equation*}\label{analyticBTseq2}
	\begin{aligned}
		0\rightarrow \bI_{I(w),I(w)}(\bm\beta,&w\cdot\underline{\lambda}) \rightarrow \cdots \rightarrow \bigoplus_{\substack{I\subseteq K\subseteq {I(w)}\\|K\backslash I|=1}}\bI_{K,I(w)}(\bm\beta,w\cdot\underline{\lambda})\rightarrow \bI_{I,I(w)}(\bm\beta,w\cdot\underline{\lambda}) \rightarrow v_{I,I(w)}^{\ana}(\bm\beta,w\cdot\underline{\lambda}) \rightarrow 0.
	\end{aligned}
\end{equation*}
This follows by induction because $I(w)\backslash \Delta^{\beta}\subsetneq \Delta_{\beta}$ and $|I(w)\backslash\Delta^{\beta}|<|\Delta_{\beta}|$.\;This completes the proof.\;
\end{proof}

\subsection{The Jordan-H\"{o}lder series of locally \texorpdfstring{$\bQ_p$}{Lg}-analytic generalized Steinberg representations}\label{JHofanaparastrep}

Let $m(w',w):=[\overline{M}(-w'\cdot \underline{0}):\overline{L}(-w\cdot\underline{0})]$ denote the multiplicity of $\overline{L}(-w\cdot\underline{0})$ in  $\overline{M}(-w'\cdot \underline{0})$.\;By \cite[Section 5.2, Corollary]{humphreysBGG}, $m(w',w)\neq 0$ if and only if $w\geq w'$.\;For a dominant weight $\underline{{\lambda}}\in X^+_{\Delta}$, one also has $m(w',w)=[\overline{M}(-w'\cdot \underline{\lambda}):\overline{L}(-w\cdot\underline{\lambda})]$.\;Indeed, the translation functor $T^{w_0\cdot \underline{\lambda}}_{w_0\cdot \underline{0}}$ is exact and sends $\overline{M}(-ww_0\cdot \underline{0})$ and $\overline{L}(-ww_0\cdot \underline{0})$ to $\overline{M}(-ww_0\cdot \underline{\lambda})$ and $\overline{L}(-ww_0\cdot \underline{\lambda})$, respectively; see \cite[Section 7]{humphreysBGG}.\;In particular, $[\overline{M}(-w'\cdot \underline{\lambda}):\overline{L}(-w\cdot\underline{\lambda})]\neq 0$ if and only if $w\geq w'$.\;These multiplicities can be computed using Kazhdan--\textit{Lusztig} polynomials.\;

For $w\in \sW_{n,L}$ such that $\Delta^{\beta}\subseteq I(w)$ and $J\subseteq I(w)\backslash\Delta^{\beta}$,\;put 
\[C_{w,J}(\bm\beta):=\mathcal{F}^G_{\op_{\beta,I(w)}}\big(\overline{L}(-w\cdot\underline{\lambda}),v_{J,I(w)\backslash\Delta^{\beta}}^{\infty}(\bm\beta)\big).\]
The following proposition slightly generalizes \cite[Theorem 4.6]{orlik2014jordan}.\;

\begin{pro}\label{JHanastein1} Let $S\subseteq I\subseteq \Delta_{\beta}$.\;For $w\in \sW_{n,L}$ such that  $I(w)=S^{\beta}(I)$ and $J\subseteq I$,\;
	\begin{equation}\label{alternatesum}
		[v_{S}^{\ana}(\bm\beta,\underline{\lambda}):C_{w,J}(\bm\beta)]=\sum_{\substack{w'\in \sW_{n,L}\\ \mathrm{ supp}(w')\backslash \Delta^{\beta}=J\backslash S}}(-1)^{l(w')+|J\backslash S|}m(w',w).
	\end{equation}
\end{pro}
\begin{proof}Proposition \ref{analyticexact} gives the multiplicity formula
	\[[v_{S}^{\ana}(\bm\beta,\underline{\lambda}):C_{w,J}(\bm\beta)]=\sum_{S\subseteq K\subseteq \Delta_{\beta}}(-1)^{|K\backslash S|}\big[\bI_K(\bm\beta,\ul{\lambda}):C_{w,J}(\bm\beta)\big].\]
	By definition, $\bI_K(\bm\beta,\underline{\lambda})\cong
	\mathcal{F}^G_{\op_{\beta,K}}\big(\overline{M}_{S^\beta(K)}(-\underline{\lambda}),\pi_{\beta,K}\big)$.\;The multiplicity $[\bI_K(\bm\beta,\ul{\lambda}):C_{w,J}(\bm\beta)]$ can be nonzero only if $K\subset I$.\;Moreover, $v_{J,I}^{\infty}(\bm\beta)$ is an irreducible constituent of $i_{K,I}(\bm\beta,\underline{\lambda})$ if and only if $K\subset J$.\;The properties of the Orlik--Strauch functor recalled in Section \ref{osfunctor}, together with \cite[Corollary 2.7]{breuil2016socle}, therefore imply that $C_{w,J}(\bm\beta)$ occurs with multiplicity one in a Jordan--H\"{o}lder series of $\mathcal{F}^G_{\op_{\beta,{I}}}\big(\overline{L}(-w\cdot\underline{\lambda}),i_{I,K}(\bm\beta,\underline{\lambda})\big)$.\;Hence $\big[\bI_K(\bm\beta,\ul{\lambda}):C_{w,J}(\bm\beta)\big]=\big[\overline{M}_{S^\beta(K)}(-\underline{\lambda}):\lL(-w\cdot\underline{\lambda})\big]$.\;We now apply the character formula in \cite[Section 9.6, Proposition]{humphreysBGG} to $\overline{M}_{S^{\beta}(K)}(-\underline{\lambda})$:
	\[\mathrm{ch}\;\lM_{S^{\beta}(K)}(-\underline{\lambda})=\sum_{w'\in \sW_{S^{\beta}(K),L}}(-1)^{l(w')}\mathrm{ch}\; \lM(-w'\cdot\underline{\lambda}).\]
	Thus the multiplicity $\big[v_{S}^{\ana}(\bm\beta,\underline{\lambda}):C_{w,J}(\bm\beta)\big]$ equals
	\begin{equation}
		\begin{aligned}		 &\;\sum_{S \subseteq K\subseteq J} (-1)^{|K\backslash S|}\big[\overline{M}_{S^\beta(K)}(-\underline{\lambda}):\lL(-w\cdot\underline{\lambda})\big]\\
			= &\sum_{S \subseteq K\subseteq J} (-1)^{|K\backslash S|}\sum_{w'\in \sW_{S^{\beta}(K),L}}(-1)^{l(w')}m(w',w)\\
			= &\sum_{w'\in \sW_{n,L}}(-1)^{l(w')}m(w',w)\sum_{\mathrm{ supp}(w')\cup S \subseteq S^{\beta}(K)\subseteq S^{\beta}(J)} (-1)^{|K\backslash S|}\\
			= &\sum_{\substack{w'\in \sW_{n,L}\\ \mathrm{ supp}(w')\backslash \Delta^{\beta}=J\backslash S}}(-1)^{l(w')+|J|}m(w',w),	\end{aligned}
	\end{equation}
	where we use the identity
	\[\sum_{\mathrm{ supp}(w')\cup S \subseteq \Delta^{\beta}\cup{K}\subseteq \Delta^{\beta}\cup{J}} (-1)^{|K\backslash S|}=(-1)^{|((\mathrm{ supp}(w')\backslash \Delta^{\beta})\cup S)\backslash S|}(1-1)^{|{J}\backslash((\mathrm{ supp}(w')\backslash \Delta^{\beta})\cup S)|},\]
	which is nonzero if and only if $(\mathrm{ supp}(w')\backslash \Delta^{\beta})\cup S={J}$.\;This completes the proof.\;
\end{proof}

\begin{exmp}Suppose $n=4$ and $(r,k)=(2,2)$ and $L=\bQ_p$.\;By \cite[Theorem 8.4 (iii)]{1987Catehighetweight}, the irreducible constituents of the generalized Verma module $\overline{M}_{\{1,3\}}(0)$ are $\{\overline{L}(0),\overline{L}(-s_2\cdot0), \overline{L}(-s_2s_3s_1s_2\cdot0)\}$.\;Then Proposition \ref{JHanastein1} therefore gives $[v_{\emptyset}^{\ana}(\bm\beta,\underline{0}):C_{w,\emptyset}]=1$ for $w\in \{1,s_2,s_2s_3s_1s_2\}$.\;
\end{exmp}
\begin{rmk}Determining whether the multiplicity in (\ref{alternatesum}) is nonzero is difficult because of the more intricate structure of generalized parabolic Verma modules.\;Many questions about parabolic Verma modules remain open outside special cases; see \cite{humphreysBGG} for a survey.\;Consequently, the composition factors of locally $\bQ_p$-analytic generalized Steinberg representations are considerably more subtle.\;
\end{rmk}

\section{Extensions between locally analytic representations}\label{sectionextana}
This section contains the main result of the paper.\;We study extension groups involving certain locally $\bQ_p$-analytic generalized Steinberg representations.\;In particular, for every $1\leq i\leq k-1$ and $\underline{\lambda}\in X^+_{\Delta}$, we show that the extensions of $v_{ir}^{\lalg}(\bm\beta,\ul{\lambda})$ by $\st_n^{\ana}(\bm\beta,\ul{\lambda})$ are parametrized by $\homo(L^\times,E)$; see Theorem \ref{analyticExt3.3}.\;

Section \ref{comphinftyhana} gives self smooth and locally analytic extension groups of $\{\pi_{\beta,I}\}_{I\subseteq \Delta_{\beta}}$.\;As a preliminary step, Appendix \ref{presmoothext} computes smooth extensions between smooth generalized Steinberg representations.\;In Section \ref{anaext1}, we prove Theorem \ref{analyticExt3.3}, the main theorem of the paper.\;In Section \ref{anaext2}, we study the subrepresentations $\Sigma_{ir}(\bm\beta,\ul{\lambda})$ and $\Sigma(\bm\beta,\ul{\lambda})$ of $\st_n^{\ana}(\bm\beta,\ul{\lambda})$ defined in (\ref{substanrep}).\;Finally, we show that extensions of $v_{ir}^{\lalg}(\bm\beta,\ul{\lambda})$ by $\st_n^{\ana}(\bm\beta,\ul{\lambda})$ are obtained by pushout from extensions by $\Sigma_{ir}(\bm\beta,\ul{\lambda})$ and $\Sigma(\bm\beta,\ul{\lambda})$.\;


\subsection{Preliminaries}\label{comphinftyhana}

Let $\varepsilon_{L}: L^{\times} \rightarrow \co_{L}^{\times}$ be the character that restricts to the identity on $\co_{L}^{\times}$ and sends $\varpi_L$ to $1$.\;Let $\log:L^{\times}\rightarrow L$ be the $p$-adic logarithm restricted to $\cO_L^{\times}$.\;For $\sigma\in \Sigma_{L}$, set $\psi_{\sigma}:=\sigma\circ \log \circ \varepsilon_{L}: L^{\times} \rightarrow E$, and let $\psi_{\ur}: L^{\times} \rightarrow E$ be the unramified character sending $p$ to $1$ (and hence $\varpi$ to $e^{-1}$).\;By \cite[Lemma 1.14]{2015Ding}, the characters $\{\psi_{\sigma}\}_{\sigma\in \Sigma_{L}}$ together with $\psi_{\ur}$ form a basis of $\homo(L^{\times},E)$.\;Thus $\dim_E\homo(L^{\times},E)=d_L+1$, and $\homo_{\sm}(L^{\times},E)$ is the one-dimensional subspace generated by $\psi_{\ur}$.\;

\begin{lem}\label{addcha}
	For $S\subseteq \Delta$, we have
	\begin{equation*}
		\begin{aligned}
			\hH^{\bullet}_{\ana}(\bL_{S}(L),E)&\cong
	\wedge^{\bullet}\homo_{\sm}(\bZ_{S}(L),E)\otimes_E \hH^{\bullet}({\fl}_{S,L},E),\\
			\hH^{\bullet}_{\ana}(\overline{\bL}_{S}(L),E)&\cong \wedge^{\bullet}\homo_{\sm}(\overline{\bZ}_{S}(L),E)\otimes_E \hH^{\bullet}(\overline{\fl}_{S,L},E).
		\end{aligned}
	\end{equation*}
	In particular,\;for $i=1,2$,\;$\hH^i_{\ana}(\bL_{S}(L),E)=\bigwedge\nolimits^{\!i}\!\homo(\bL_{S}(L),E)$ and $\hH^i_{\ana}(\overline{\bL}_{S}(L),E)=\bigwedge\nolimits^{\!i}\!\homo(\overline{\bL}_{S}(L),E)$.\;
\end{lem}
\begin{proof}
	By \cite[Corollary 3.11,\;(3.28)]{schraen2011GL3},\;we have 
	\[\hH^{\bullet}_{\ana}(\bL_{S}(L),E)=\hH^{\bullet}_{\ana}(\bZ_{S}(L),E)\otimes_E\hH^{\bullet}_{\ana}(\bD_{S}(L),E)\cong \hH^{\bullet}_{\ana}(\bZ_{S}(L),E)\otimes_E \hH^{\bullet}\big(\fd_{S,L}, E\big).\]
	Since $\bZ_{I}$ is a torus, \cite[Corollary 8.10]{kohlhaase2011cohomology}, applied to $\res_{L/\bQ_p}\bZ_{I}$, gives
	\begin{equation}\label{decomp2}
		\begin{aligned}
			&\hH^{\bullet}_{\ana}(\bZ_{S}(L),E)\cong \wedge^{\bullet}\homo_{\sm}(\bZ_{S}(L),E)\otimes_E \hH^\bullet({\fz}_{S,L},E)
		\end{aligned}
	\end{equation}
	The K\"{u}nneth formula for Lie algebras (see \cite[(3.24)]{schraen2011GL3}) gives $\hH^{\bullet}({\fl}_{S,L},E)=\hH^\bullet({\fz}_{S,L},E)\otimes_E\hH^{\bullet}\big(\fd_{S,L}, E\big)$.\;This proves the first assertion, and the same argument applies to $\hH^i_{\ana}(\overline{\bL}_{S}(L),E)$.\;
\end{proof}
\begin{rmk}As a direct consequence,\;the following two spectral sequences
	\begin{equation}
		\begin{aligned}
		\hH^{\bullet}_{\sm}(\bZ_{S}(L),E)\otimes_E\hH_s(\fl_{S,L},1)\Rightarrow 	\hH^{\bullet}_{\ana}(\bL_{S}(L),E), \quad	\hH^{\bullet}_{\sm}(\overline{\bL}_{S}(L),E)\otimes_E\hH_s(\overline{\fl}_{S,L},1)\Rightarrow\hH^{\bullet}_{\ana}(\overline{\bL}_{S}(L),E).
		\end{aligned}
	\end{equation}
	from Lemma \ref{smoothtoanalytic} degenerate at the second page.\;
\end{rmk}

For any $I\subseteq \Delta_{\beta}$ and $i\in \BZ_{\geq0}$, set
\[\cE_I^i:=\ext^{i}_{\bL_{\beta,I}(L)}(\pi_{\beta,I},\pi_{\beta,I}),\quad \cEo_ I^i:=\ext^{i}_{\bL_{\beta,I}(L),Z}(\pi_{\beta,I},\pi_{\beta,I}).\]
By Lemma \ref{smoothtoanalytic},\;we have the following spectral sequences:
\begin{equation}
	\begin{aligned}
			\ssE_{I}^r\otimes_E\hH_s(\fl_{\beta,I,L},1)\Rightarrow \cE_I^{r+s}, \quad\sEo_{I}^r\otimes_E\hH_s(\overline{\fl}_{\beta,I,L},1)\Rightarrow\cEo_I^{r+s}.\;
	\end{aligned}
	\end{equation}
where $\ssE_I^r=\ext^{r,\infty}_{\bL_{\beta,I}(L)}(\pi_{\beta,I},\pi_{\beta,I})$ and $ \sEo_I^r:=\ext^{r,\infty}_{\bL_{\beta,I}(L),Z}(\pi_{\beta,I},\pi_{\beta,I})$ (see (\ref{dfnforlevismoothext}) in Section \ref{presmoothext}).\;Moreover, Lemma \ref{SmoothEXTlemmaPre} shows that the natural morphisms induce isomorphisms
\begin{equation}\label{analyticcharactertoanaext}
	\begin{aligned}
		&\kappa_I^i:\ext^{i,\infty}_{\bL_{\beta,I}(L)}(1,1) \rightarrow \ssE_{I}^i,\quad \text{resp.\;}\overline{\kappa}_I^i: \ext^{i,\infty}_{\bL_{\beta,I}(L),Z}(1,1) \rightarrow \sEo_{I}^i
	\end{aligned}
\end{equation}
Tensoring extensions of the trivial representation with $\pi_{\beta,I}$ defines morphisms from the two cohomological spectral sequences in the preceding remark to the two self-extension spectral sequences above.\;On the $E_2$-pages these morphisms are $\kappa_I^r\otimes\mathrm{id}$ and $\overline{\kappa}_I^r\otimes\mathrm{id}$, respectively, and hence are isomorphisms by Lemma \ref{SmoothEXTlemmaPre}.\;The source spectral sequences degenerate at $E_2$ by the preceding remark.\;It follows, page by page, that the target spectral sequences also degenerate at $E_2$.\;The induced maps on the associated graded spaces of the finite abutment filtrations are therefore isomorphisms, and consequently so are the maps on the abutments.\;We obtain the following result.
\begin{pro}\label{hHinj}
	For any $I\subseteq \Delta_{\beta}$ and $i\geq 0$, the maps
	\begin{equation}\label{charactertoanaextana}
		\iota_I^i:\mathrm{H}^i_{\ana}(\bL_{\beta,I}(L),E)\rightarrow \cE_{I}^i,\quad \overline{\iota}_I^i: \mathrm{H}^i_{\ana}(\overline{\bL}_{\beta,I}(L),E)\rightarrow \cEo_{I}^i
	\end{equation}
	are isomorphisms.\;
\end{pro}

Let $J\subseteq I\subseteq \Delta_{\beta}$.\;The natural injection $\bL_{\beta,J}(L)\hookrightarrow \bL_{\beta,I}(L)$ therefore induces a restriction map
\[\mathrm{Res}^I_J:\hH^1_{\ana}(\bL_{\beta,I}(L),E)\cong \homo(\bZ_{\beta,I}(L),E) \rightarrow\hH^1_{\ana}(\bL_{\beta,J}(L),E)\cong \homo(\bZ_{\beta,J}(L),E).\]
This morphism $ \homo(\bZ_{\beta,I}(L),E) \rightarrow\homo(\bZ_{\beta,J}(L),E)$ is induced by the composition
$$\bZ_{\beta,J}(L)\hookrightarrow \bL_{\beta,J}(L)\hookrightarrow \bL_{\beta,I}(L)\xrightarrow{\det}\bZ_{\beta,I}(L).\;$$
We henceforth regard $\homo(\bZ_{\beta,I}(L),E)$ as a subspace of $\homo(\bZ_{\beta,J}(L),E)$.\;


\subsection{Extensions between locally analytic representations-I}\label{anaext1}
We now study extension groups involving certain locally $\bQ_p$-analytic generalized Steinberg representations.\;Our argument adapts the approach of \cite[Section 2.2]{2019DINGSimple} to our case.

\begin{lem}\label{analyticExt1analyticExtC1}Let $I,J\subseteq \Delta_\beta$,\;we have 
\begin{equation}\label{analyticExt1}
\ext^i_{G}(i_I\big(\bm\beta,\underline{\lambda}),\bI_J(\bm\beta,\underline{\lambda})\big)=\left\{
\begin{array}{ll}
	\cE_J^i, & \hbox{$J\subseteq I$;} \\
	0, & \hbox{otherwise.}
\end{array}
\right.
\end{equation}
and
\begin{equation}\label{analyticExtC1}
\ext^i_{G,Z}\big(i_I(\bm\beta,\underline{\lambda}),\bI_J(\bm\beta,\underline{\lambda})\big)=\left\{
\begin{array}{ll}
	\cEo_J^i, & \hbox{$J\subseteq I$;} \\
	0, & \hbox{otherwise.}
\end{array}
\right.
\end{equation}
\end{lem}
\begin{proof}We adapt the proof of \cite[Proposition 2.11]{2019DINGSimple}.\;It suffices to prove (\ref{analyticExtC1}), since the same argument proves (\ref{analyticExt1}).\;To apply \cite[Corollary 4.9]{schraen2011GL3},\;one must verify the separatedness hypothesis for the term $\mathrm{H}_s({\on}_{\beta,J}(L),i_I(\bm\beta,\underline{\lambda}))$.\;This holds because $i_I(\bm\beta,\underline{\lambda})$ is locally $\bQ_p$-algebraic.\;We therefore obtain a spectral sequence
\begin{equation}\label{spectalse1}
\begin{aligned}
	E_2^{r,s}=\ext^r_{\bL_{\beta,J}(L),Z}\big(\mathrm{H}_s\big({\on}_{\beta,J}(L),i_I(\bm\beta,\underline{\lambda})\big),&\pi_{\beta,J}(\underline{\lambda})\big)\Rightarrow \ext^{r+s}_{G,Z}\big(i_I(\bm\beta,\underline{\lambda}),\bI_J(\bm\beta,\underline{\lambda})\big).
\end{aligned}
\end{equation}
By \cite[Theorem 4.10]{schraen2011GL3}, there are isomorphisms of locally $\bQ_p$-algebraic representations of $\bL_{\beta,J}(L)$:
\begin{equation}\label{spectalse2unipotent}
\begin{aligned}
	\hH_s\big(\on_{\beta,J},i_I(\bm\beta,\underline{\lambda})\big)&=J^{\beta}_{\Delta_\beta,J}(i_I(\bm\beta))\otimes_E\mathrm{H}_s(\overline{\mathfrak{n}}_{\beta,J,L},L(\underline{\lambda}))\\
	&=J^{\beta}_{\Delta_\beta,J}(i_I(\bm\beta))\otimes_E\Big(\bigoplus_{\substack{w\in \sW_{n,L},\;\leg(w)=s\\w\cdot \underline{\lambda}\in X^+_{{S^{\beta}(J)}}}}L(w\cdot \underline{\lambda})_{S^{\beta}(J)}\Big),
\end{aligned}
\end{equation}
By \cite[Proposition 4.7(1)]{schraen2011GL3}, the terms $E_2^{r,s}$ vanish unless $s=0$.\;Consequently, for every $r\geq 0$, we have
\begin{equation*}
\begin{aligned}
	 \ext^{r}_{G,Z}\Big(i_I(\bm\beta,\underline{\lambda}),\bI_J(\bm\beta,\underline{\lambda})\Big)=&\; \ext^r_{\bL_{\beta,J}(L),Z}\Big(J^{\beta}_{\Delta_\beta,J}(i_I(\bm\beta))\otimes_EL(\underline{\lambda})_{S^{\beta}(J)},\pi_{\beta,J}(\underline{\lambda})\Big)\\
	=&\; \ext^r_{\bL_{\beta,J}(L),Z}\Big(J^{\beta}_{\Delta_\beta,J}(i_I(\bm\beta)),\pi_{\beta,J}\otimes_E\mathrm{End}_E(L(\underline{\lambda})_{S^{\beta}(J)})\Big),\\
	=&\; \ext^r_{\bL_{\beta,J}(L),Z}\Big(J^{\beta}_{\Delta_\beta,J}(i_I(\bm\beta)),\pi_{\beta,J}\Big), 
\end{aligned}
\end{equation*}
where the last isomorphism follows from \cite[Proposition 4.7(1)]{schraen2011GL3} (note that the algebraic representation $\mathrm{End}_E(L(\underline{\lambda})_{J})$ is semisimple and contains the trivial representation with multiplicity one).\;Recall the Bruhat filtration $\cF^{\bullet}$ from (\ref{filext}) on $i_I(\bm\beta)$.\;It satisfies
\[J^{\beta}_{\Delta_{\beta},J}\Big(\gr^k_{\cF^\bullet}\big(i_I(\bm\beta)\big)\Big)\cong
\bigoplus_{\substack{w\in \sW_n^{\beta,I,J}\\\mathrm{lg}(w)=k}}i^{\beta}_{w^{-1}(I)\cap J,J}\big(J^{\beta,w}_{I,I\cap w(J)}(\pi_{\beta,I})\big)\]
for every $k\geq 0$.\;If either $1\neq w\in \sW_n^{\beta,I,J}$ or $J\nsubseteq I$, we claim that
\[\ext^r_{\bL_{\beta,J}(L),Z}\Big(i^{\beta}_{w^{-1}(I)\cap J,J}\big(J^{\beta,w}_{I,I\cap w(J)}(\pi_{\beta,I})\big),\pi_{\beta,J}\Big)=0.\]
The spectral sequence from Lemma \ref{smoothtoanalytic},
\begin{equation*}
\begin{aligned}
	\ext^{r,\infty}_{\bL_{\beta,J}(L),Z}\Big(\mathrm{H}_s(\overline{\mathfrak{l}}_{\beta,J,L},E)&\otimes_Ei^{\beta}_{w^{-1}(I)\cap J,J}\big(J^{\beta,w}_{I,I\cap w(J)}(\pi_{\beta,I})\big),\pi_{\beta,J}\Big)\\
	\Rightarrow&\; \ext^{r+s}_{\bL_{\beta,J}(L),Z}\Big(i^{\beta}_{w^{-1}(I)\cap J,J}\big(J^{\beta,w}_{I,I\cap w(J)}(\pi_{\beta,I})\big),\pi_{\beta,J}\Big),\;
\end{aligned}
\end{equation*}
shows that it suffices to prove
\begin{equation}\label{smoothvanishdevissage}
		\ext^{r,\infty}_{\bL_{\beta,J}(L),Z}\Big(i^{\beta}_{w^{-1}(I)\cap J,J}\big(J^{\beta,w}_{I,I\cap w(J)}(\pi_{\beta,I})\big),\pi_{\beta,J}\Big)=0
\end{equation}
when $J\nsubseteq I$ or $w\neq 1$.\;In this case,
\begin{equation}\label{extIJfil1}
	\begin{aligned}
		& 	\ext^{r,\infty}_{\bL_{\beta,J}(L),Z}\Big(i^{\beta}_{w^{-1}(I)\cap J,J}\big(J^{\beta,w}_{I,I\cap w(J)}(\pi_{\beta,I})\big),\pi_{\beta,J}\Big)\\
		\cong&\; \ext^{r,\infty}_{\bL_{\beta,J}(L),Z}\Big(\pi_{\beta,J}^{\sim},i^{\infty}_{J\cap w^{-1}(I),J}\big(\big(J^{\beta,w}_{I,I\cap w(J)}(\pi_{\beta,I})\big)^{\sim}\delta_{\bL_{\beta,J}(L)\cap \op_{\beta,{I}}(L)^w}\big)\Big)\\
		\cong&\;\ext^{r,\infty}_{\bL_{\beta,J\cap w^{-1}I}(L),Z}\Big(J^{\beta}_{J,J\cup w^{-1}(I)}(\pi_{\beta,J}^{\sim}),\big(\delta_{I,J,w}\otimes_E\big(J^{\beta}_{I,I\cap w(J)}\pi_{\beta,I}\big)^w \big)^{\sim}\delta_{\bL_{\beta,J}(L)\cap \op_{\beta,{I}}(L)^w}\Big)\\
		=&: \ext^{r,\infty}_{\bL_{\beta,J\cap w^{-1}I}(L),Z}(A,B),
	\end{aligned}
\end{equation}
where the first isomorphism follows from \cite[Lemma 14]{2012Orlsmoothextensions}, with $(-)^{\sim}$ denoting the smooth dual,\;and the second follows from \cite[Lemma 8]{2012Orlsmoothextensions}.\;If $J\nsubseteq I$ or $w\neq 1$, the proof of \cite[Proposition 15, especially (2) and (3), and Lemma 16]{2012Orlsmoothextensions} provides an element $z\in \bZ_{\beta,J\cap w^{-1}I}(L)$ such that $(\delta^{\beta}_{I,J,w})^{\vee}(z)\delta_{\bL_{\beta,J}(L)\cap \op_{\beta,{I}}(L)^w}(z)\neq 1$.\;It follows from (\ref{LEIVEIJrk}) that $z$ acts on $B$ (resp., $A$) by the scalar
$$z_B:=(\pi_{\beta,I\cap w(J)}^w\otimes_E\delta^{\beta}_{I,J,w})^{\sim}(z)\delta_{\bL_{\beta,J}(L)\cap \op_{\beta,{I}}(L)^w}(z),\quad(\mathrm{resp.}\;z_A:=(\pi_{\beta,J\cap w^{-1}(I)})^{\sim}(z)).\;$$
Since $z_A\neq z_B$, Lemma \ref{centervan} gives
(\ref{smoothvanishdevissage}) for $J\nsubseteq I$ or $w\neq 1$.\;If $J\nsubseteq I$,\;a d\'{e}vissage argument now gives $\ext^r_{G,Z}\big(i_I(\bm\beta,\underline{\lambda}),\bI_J(\bm\beta,\underline{\lambda})\big)=0$.\;If $J \subseteq I$,\;by Proposition \ref{axioms} \textbf{[A1]},\;we have
\begin{equation*}
\begin{aligned}
	& \ext^r_{\bL_{\beta,J}(L),Z}\Big(J^{\beta}_{\Delta_{\beta},J}\big(i_{I}(\bm\beta)),\pi_{\beta,J}\Big)= \ext^r_{\bL_{\beta,J}(L),Z}\Big(J^{\beta}_{\Delta_{\beta},J}\big(\gr^0_{\cF^\bullet}(i_{I}(\bm\beta))),\pi_{\beta,J}\Big)= \ext^{r}_{\bL_{\beta,J}(L),Z}\big(\pi_{\beta,J},\pi_{\beta,J}\big).\;
\end{aligned}
\end{equation*}
This completes the proof of (\ref{analyticExtC1}).\;
\end{proof}

As a corollary of Lemma \ref{analyticExt1analyticExtC1}, we have
\begin{cor}\label{coranalyticExt2}
Let $I,J\subseteq \Delta_\beta$,\;we have 
\begin{equation}\label{analyticExt2}
\ext^i_{G}\big(v^{\lalg}_I(\bm\beta,\underline{\lambda}),\bI_J(\bm\beta,\underline{\lambda})\big)=\left\{
\begin{array}{ll}
	\cE_J^{i-|\Delta_{\beta}\backslash I|}, & \hbox{$I\cup J=\Delta_{\beta}$;} \\
	0, & \hbox{otherwise.}
\end{array}
\right.
\end{equation}
\begin{equation}\label{analyticExtC2}
\ext^i_{G,Z}\big(v^{\lalg}_I(\bm\beta,\underline{\lambda}),\bI_J(\bm\beta,\underline{\lambda})\big)=\left\{
\begin{array}{ll}
	\cEo_J^{i-|\Delta_{\beta}\backslash I|}, & \hbox{$I\cup J=\Delta_{\beta}$;} \\
	0, & \hbox{otherwise.}
\end{array}
\right.
\end{equation}
\end{cor}
\begin{proof}We prove (\ref{analyticExtC2}); the proof of (\ref{analyticExt2}) is identical.\;Proposition \ref{smoothexact} gives
	\[\ext^{\bullet}_{G,Z}\big(v^{\lalg}_I(\bm\beta,\underline{\lambda}),\bI_J(\bm\beta,\underline{\lambda})\big)\xrightarrow{\sim}\ext^{\bullet}_{G,Z}\big(\bC_{I,\bullet}^{\lalg}(\bm\beta,\ul{\lambda}),\bI_J(\bm\beta,\underline{\lambda})\big).\]
The argument following \cite[Proposition 5.4]{schraen2011GL3} then yields a spectral sequence
\begin{equation*}
E_1^{s,r}=\ext_{G,Z}^r\Big(\bigoplus_{\substack{I\subseteq K\subseteq \Delta_{\beta}\\|K\backslash I|=s}}i_K(\bm\beta,\underline{\lambda}),\bI_J(\bm\beta,\underline{{\lambda}})\Big) \Rightarrow \ext^{r+s}_{G,Z}\big(v_I^{\lalg}(\bm\beta,\underline{{\lambda}}), \bI_J(\bm\beta,\underline{{\lambda}})\big).
\end{equation*}
By (\ref{analyticExtC1}), for every $r\in \mathbb{Z}_{\geq 0}$, the $r$-th row of the $E_1$-page is
\begin{equation*}
\begin{matrix}& E_1^{|I\cup J|-|I|,r}& \ra & E_1^{|I\cup J|-|I|+1,r} & \ra & \cdots & \ra & E_1^{ k-1-|I|,r} \\
	&\parallel &\empty&\parallel &\empty &\empty &\empty &\parallel  \\
	&\cEo_J^r&\ra &\bigoplus\limits_{\substack{I\cup J\subseteq K,\\ |K\backslash (I\cup J)|=1}}\cEo_J^r& \ra &\cdots & \ra &\cEo_J^r.
\end{matrix}
\end{equation*}
If $I\cup J\neq \Delta_{\beta}$, this row is the constant coefficient complex of the standard simplex associated with $I\cup J$, and is therefore exact.\;It follows that $\ext^{r}_{G,Z}\big(v_I^\lalg(\bm\beta,\underline{\lambda}), \bI_J(\bm\beta,\underline{\lambda})\big)=0$ for all $r\in \mathbb{Z}_{\geq 0}$.\;If $I\cup J =\Delta_{\beta}$, all entries of the $E_1$-page vanish except those in the $|\Delta_{\beta}\backslash I|$-th column, so
\begin{equation*}
\cEo_J^r\cong \ext^r_{G,Z}\big(i_{\Delta_{\beta}}(\bm\beta,\underline{\lambda}),\bI_J(\bm\beta,\underline{\lambda})\big)\cong \ext^{r+|\Delta_{\beta}\backslash I|}_{G,Z}\big(v^{\lalg}_I(\bm\beta,\underline{\lambda}),\bI_J(\bm\beta,\underline{\lambda})\big).
\end{equation*}
This completes the proof.\;
\end{proof}

We now give an explicit description of the $\ext^1$-groups in Lemma \ref{analyticExt1analyticExtC1}.\;Let $J\subseteq I\subseteq\Delta_{\beta}$.\;Let
\begin{equation*}
	\Psi\in \hH^1_{\ana}\big(\bL_{\beta,J}(L),E\big)\cong \homo\big(\bZ_{\beta,J}(L),E\big),
\end{equation*}
which induces a self-extension $\iota_J(\Psi)$ of the representation $\pi_{\beta,J}$ of $\bL_{\beta,J}(L)$:
\begin{equation*}
	\iota_J(\Psi)(a)= \pi_{\beta,J}(a)\otimes_E\begin{pmatrix}
		1 & \Psi\circ\det(a) \\ 0 & 1
	\end{pmatrix}, \\ \quad\forall ~a\in \bL_{\beta,J}(L).
\end{equation*}
We may identify $\iota_J(\Psi)$  with the $E[\epsilon]/\epsilon^2$-valued representation $\pi_{\beta,J}(1+(\Psi\circ\det)\epsilon)$.\;Thus, the locally $\bQ_p$-analytic parabolic induction $\big(\ind_{\op_{\beta,{J}}(L)}^G \iota_J(\Psi)\otimes_EL(\ul{\lambda})_{S^\beta(J)}\big)^{\bQ_p-\ana}$ lies in an exact sequence
\begin{equation}\label{explicitexactseq}
0 \longrightarrow \bI_J(\bm\beta,\ul{\lambda}) \longrightarrow \big(\ind_{\op_{\beta,{J}}(L)}^G \iota_J(\Psi)\otimes_EL(\ul{\lambda})_{S^\beta(J)}\big)^{\bQ_p-\ana} \xrightarrow{\pr} \bI_J(\bm\beta,\ul{\lambda}) \longrightarrow 0.
\end{equation}
Pulling back (\ref{explicitexactseq}) along the natural injections $i_{I}(\bm\beta,\ul{\lambda}) \hooklongrightarrow \bI_I(\bm\beta,\ul{\lambda}) \hooklongrightarrow \bI_J(\bm\beta,\ul{\lambda})$ gives a locally $\bQ_p$-analytic representation  $\sE_{I}^J(\bm\beta,\ul{\lambda},\Psi)^+$,\;which is an extension of $i_{I}(\bm\beta,\ul{\lambda})$ by $\bI_J(\bm\beta,\ul{\lambda})$.\;It gives a cohomology class  $[\sE_{I}^J(\bm\beta,\ul{\lambda},\Psi)^+]\in \ext^{1}_{G}\big(i_{I}(\bm\beta,\ul{\lambda}), \bI_J(\bm\beta,\ul{\lambda})\big)$.\;Adapting the proof of \cite[Lemma 2.13, Remark 2.18]{2019DINGSimple} to our setting, we can show that
\begin{lem}\label{explictextgroup1}
The extension class $[\sE_{I}^J(\bm\beta,\ul{\lambda},\Psi)^+]$  is mapped (up to nonzero scalars) to $\iota_J(\Psi)\in \cE_J^1$ via the isomorphism in (\ref{analyticExt1}).\;
\end{lem}
\begin{proof}By the discussion in \cite[Section 4.4]{schraen2011GL3},  the isomorphism (\ref{analyticExt1}) factors through the following composition:
\begin{equation}\label{expsends}
\begin{aligned}
	\ext^{1}_{G}\big(i_{I}(\bm\beta,\ul{\lambda}), \bI_J(\bm\beta,\ul{\lambda})\big)
	\xrightarrow{\sim}&\;\ext^1_{\op_{\beta,J}(L)}\Big(i_{I}(\bm\beta,\ul{\lambda})|_{\op_{\beta,J}(L)},\pi_{\beta,J}(\underline{\lambda})\Big)\\
	\xrightarrow{\sim}&\; \ext^1_{\bL_{\beta,J}(L)}\Big(J^{\beta}_{\Delta_{\beta},J}(i_{I}(\bm\beta,\ul{\lambda})),\pi_{\beta,J}\Big)\\
	\xrightarrow{\sim}&\; \ext^1_{\bL_{\beta,J}(L)}\left(\pi_{\beta,J}(\underline{\lambda}),\pi_{\beta,J}(\underline{\lambda})\right),
\end{aligned}
\end{equation}
where the first map sends  an extension $V$ to the push-forward of $V|_{\op_{\beta,J}(L)}$ via the natural evaluation map $\bI_J(\bm\beta,\ul{\lambda})\twoheadrightarrow \pi_{\beta,J}(\underline{\lambda})$, $f\mapsto f(1)$,\;the second map is the inverse of the map induced by the natural projection $i_I(\bm\beta,\ul{\lambda})|_{\op_{\beta,J}(L)}\twoheadrightarrow J^{\beta}_{\Delta_{\beta},J}(i_{I}(\bm\beta,\ul{\lambda}))$,\;and the third map is induced by  the injection $\pi_{\beta,J}(\underline{\lambda})\hookrightarrow J^{\beta}_{\Delta_{\beta},J}(i^{\beta}_{I,\Delta_{\beta}}(\bm\beta,\ul{\lambda}))$.\;Note that the composition (\ref{expsends}) is also the edge morphism given by the spectral sequence  (\ref{spectalse1}).\;We claim that $[\sE_{I}^J(\bm\beta,\ul{\lambda},\Psi)^+]$ is mapped  to $\iota_J(\Psi)$ via (\ref{expsends}).\;Similar to the discussion in \cite[Lemma 2.13]{2019DINGSimple}, we can show that the following composition (note that $i_{I}(\bm\beta,\ul{\lambda}) \hooklongrightarrow  \bI_J(\bm\beta,\ul{\lambda})$)
\begin{equation}
\begin{aligned}
	\ext^1_{\bL_{\beta,J}(L)}\left(\pi_{\beta,J}(\underline{\lambda}),\pi_{\beta,J}(\underline{\lambda})\right)
	\rightarrow&\;
	\ext^1_{\op_{\beta,J}(L)}\Big(\bI_{J}(\bm\beta,\ul{\lambda}))|_{\op_{\beta,J}(L)},\pi_{\beta,J}(\underline{\lambda})\Big) \\
	\rightarrow&\; \ext^1_{\op_{\beta,J}(L)}\Big(i_{I}(\bm\beta,\ul{\lambda})|_{\op_{\beta,J}(L)},\pi_{\beta,J}(\underline{\lambda})\Big)
\end{aligned}
\end{equation}
is the inverse of the composition of the last two maps in (\ref{expsends}).\;By the construction of $[\sE_{I}^J(\bm\beta,\ul{\lambda},\Psi)^+]$, we see that the representation $[\sE_{I}^J(\bm\beta,\ul{\lambda},\Psi)^+]$ is just the image of $\iota_J(\Psi)$ via the above composition and the inverse of the first isomorphism in (\ref{expsends}).\;The lemma follows.\;
\end{proof}
\begin{rmk}For $J\subseteq I$, the injection $\bI_I(\bm\beta,\ul{\lambda}) \hooklongrightarrow \bI_J(\bm\beta,\ul{\lambda})$ induces a map
$$\ext^{1}_{G}\big(i_{I}(\bm\beta,\ul{\lambda}), \bI_I(\bm\beta,\ul{\lambda})\big)\rightarrow \ext^{1}_{G}\big(i_{I}(\bm\beta,\ul{\lambda}), \bI_J(\bm\beta,\ul{\lambda})\big).$$
Moreover, the following diagram commutes; an analogous statement holds for $\ext^{1}_{G,Z}(-,-)$:
\begin{equation}\label{commutativeforIJ}
	\xymatrix{
		\ext^{1}_{G}\big(i_{I}(\bm\beta,\ul{\lambda}), \bI_I(\bm\beta,\ul{\lambda})\big)
		\ar[d]^{\sim}_{(\ref{analyticExt1})} \ar[r] & \ext^{1}_{G}\big(i_{I}(\bm\beta,\ul{\lambda}), \bI_J(\bm\beta,\ul{\lambda})\big) \ar[d]^{\sim}_{(\ref{analyticExt1})} \\
		\cE_I^1 \ar[r]^{\mathrm{Res}^I_J} & \cE_J^1 .}
\end{equation}
The preceding descriptions also apply in the smooth setting,\;using the isomorphism (see Lemma \ref{SmoothExt1}) $\ext^{1,\infty}_{G}\big(i_{I}(\bm\beta,\ul{\lambda}), i_{J}(\bm\beta,\ul{\lambda})\big)\xrightarrow{\sim}\ssE_J^1$.\;Since $\homo_{G}\big(i_{I}(\bm\beta,\ul{\lambda}),\bI_J(\bm\beta,\ul{\lambda})/i_{J}(\bm\beta,\ul{\lambda})\big)=0$,\;a d\'{e}vissage argument shows that the injection $i_{J}(\bm\beta,\ul{\lambda})\hookrightarrow \bI_J(\bm\beta,\ul{\lambda})$ induces injections
\[	\ext^{1,\infty}_{G}\big(i_{I}(\bm\beta,\ul{\lambda}), i_{J}(\bm\beta,\ul{\lambda})\big)\hookrightarrow\ext^{1}_{G}\big(i_{I}(\bm\beta,\ul{\lambda}), i_{J}(\bm\beta,\ul{\lambda})\big)\hookrightarrow \ext^{1}_{G}\big(i_{I}(\bm\beta,\ul{\lambda}), \bI_J(\bm\beta,\ul{\lambda})\big).\]
Using (\ref{smoothanalyticedge}) in Lemma \ref{smoothtoanalytic}, Lemma \ref{SmoothExt1}, and the explicit description of $\ext^{1}$, we obtain the commutative diagram
\[\xymatrix@C=2ex{
\ext^{1,\infty}_{G}\big(i_{I}(\bm\beta,\ul{\lambda}), i_{J}(\bm\beta,\ul{\lambda})\big)  
\ar[d]^{\simeq} \ar@{^(->}[r] &  \ext^{1}_{G}\big(i_{I}(\bm\beta,\ul{\lambda}), \bI_J(\bm\beta,\ul{\lambda})\big) \ar[d]^{\simeq} \\
\homo_{\sm}(\bL_{\beta,J}(L),E)   \ar@{^(->}[r] & \homo(\bL_{\beta,J}(L),E) .}\]
In particular, $\sE_{I}^J(\bm\beta,\ul{\lambda},\Psi)^+$ comes from a locally $\bQ_p$-algebraic extension of $i_{I}(\bm\beta,\ul{\lambda})$ by $i_{J}(\bm\beta,\ul{\lambda})$ if and only if $\Psi\in \homo_{\sm}(\bL_{\beta,J}(L),E)$.\;
\end{rmk}

Push forward (\ref{explicitexactseq}) along the surjection $\bI_J(\bm\beta,\ul{\lambda})\twoheadrightarrow v^{\ana}_J(\bm\beta,\ul{\lambda})$, and then pull back along the injections $i_{I}(\bm\beta,\ul{\lambda}) \hooklongrightarrow \bI_I(\bm\beta,\ul{\lambda}) \hooklongrightarrow \bI_J(\bm\beta,\ul{\lambda})$,\;the resulting locally $\bQ_p$-analytic representation $\sE_{I}^J(\bm\beta,\ul{\lambda},\Psi)^0$ is an extension of $i_{I}(\bm\beta,\ul{\lambda})$ by $v^{\ana}_J(\bm\beta,\ul{\lambda})$.\;It defines a class $[\sE_{I}^J(\bm\beta,\ul{\lambda},\Psi)^0]\in \ext^{1}_{G}\big(i_{I}(\bm\beta,\ul{\lambda}),v^{\ana}_J(\bm\beta,\ul{\lambda})\big)$.\;

Let $\Psi'\in \hH^1_{\ana}\big(\bL_{\beta,I}(L),E\big)\cong \homo\big(\bZ_{\beta,I}(L),E\big)$.\;Then we have the following short exact sequence:
\begin{equation}\label{explicitexactseq1}
	0 \longrightarrow \bI_I(\bm\beta,\ul{\lambda}) \longrightarrow \big(\ind_{\op_{\beta,{I}}(L)}^G \iota_I(\Psi')\otimes_EL(\ul{\lambda})_{S^{\beta}(I)}\big)^{\bQ_p-\ana} \xrightarrow{\pr} \bI_I(\bm\beta,\ul{\lambda}) \longrightarrow 0.
\end{equation}
Pushing forward (\ref{explicitexactseq1}) along the natural injection $\bI_I(\bm\beta,\ul{\lambda})\hooklongrightarrow\bI_J(\bm\beta,\ul{\lambda})$ and then pulling back along $i_{I}(\bm\beta,\ul{\lambda}) \hooklongrightarrow \bI_I(\bm\beta,\ul{\lambda})$ recovers $\sE_{I}^J(\bm\beta,\ul{\lambda},\mathrm{Res}^I_J(\Psi'))^+$ by Lemma \ref{explictextgroup1} and (\ref{commutativeforIJ}).\;Pushing forward along the composition $\bI_I(\bm\beta,\ul{\lambda})\hooklongrightarrow\bI_J(\bm\beta,\ul{\lambda})\twoheadrightarrow v^{\ana}_J(\bm\beta,\ul{\lambda})$  recovers $\sE_{I}^J(\bm\beta,\ul{\lambda},\mathrm{Res}^I_J(\Psi'))^0$.\;By definition, the resulting extension class is zero.\;We therefore obtain a homomorphism
\begin{equation}\label{IJmap0}
\varrho_{I,J}^0: \cE_J^1/\mathrm{Res}^I_J(\cE_I^1)\cong \homo\big(\bZ_{\beta,J}(L),E\big)/\homo\big(\bZ_{\beta,I}(L),E\big)\longrightarrow \ext^{1}_{G}\big(i_{I}(\bm\beta,\ul{\lambda}),v^{\ana}_J(\bm\beta,\ul{\lambda})\big).\;
\end{equation}

\begin{pro}\label{analyticExt3.2}Let $J\subseteq I\subseteq\Delta_{\beta}$ such that $|I|=|J|+1$.\;Then
\begin{itemize}
	\item[(a)]$\ext^1_{G,Z}\big(i_{I}(\bm\beta,\ul{\lambda}),v^{\ana}_J(\bm\beta,\ul{\lambda})\big)\xrightarrow{\sim}\ext^1_G\big(i_{I}(\bm\beta,\ul{\lambda}),v^{\ana}_J(\bm\beta,\ul{\lambda})\big).$
	\item[(b)] The map $\varrho_{I,J}^0$ induces an isomorphism of $E$-vector spaces:
	\[\homo\big(\bZ_{\beta,J}(L),E\big)/\homo\big(\bZ_{\beta,I}(L),E\big)\xrightarrow{\sim} \ext^1_{G}(i_{I}\big(\bm\beta,\ul{\lambda}),v^{\ana}_J(\bm\beta,\ul{\lambda})\big).\;\]
\end{itemize}
\end{pro}
\begin{proof}Since $\homo_G\big(i_I(\bm\beta,\ul{\lambda}), v^{\ana}_J(\bm\beta,\ul{\lambda})\big)=0$,\;(a) follows from \cite[Lemma 3.1]{HigherLinvariantsGL3Qp} or \cite[Lemma 3.2]{wholeLINV}.\;By Proposition \ref{analyticexact},
	\[\ext^{\bullet}_{G}\big(i_{I}(\bm\beta,\ul{\lambda}),\bC_{J,\bullet}^{\ana}(\bm\beta,\ul{\lambda})\big)\xrightarrow{\sim}\ext^{\bullet}_{G}\big(i_{I}(\bm\beta,\ul{\lambda}),v^{\ana}_J(\bm\beta,\ul{\lambda})\big).\]
As in the proof of Corollary \ref{coranalyticExt2}, we get a spectral sequence
\begin{equation}\label{ss1}
E_1^{-s,r}=\bigoplus_{\substack{K\subseteq \Delta_{\beta}\\|K\backslash J|=s}} \ext^{r}_{G}\big(i_{I}(\bm\beta,\ul{\lambda}), \bI_K(\bm\beta,\ul{\lambda})\big) \Rightarrow \ext^{r-s}_{G}\big(i_{I}(\bm\beta,\ul{\lambda}), v^{\ana}_J(\bm\beta,\ul{\lambda})\big).
\end{equation}
By (\ref{analyticExt1}), only the $0$-th and $(-1)$-th columns of the $E_1$-page of (\ref{ss1}) can have non-zero terms.\;The $1$-th row is
\begin{equation*}
\begin{matrix}
	& \cE_I^1 &\lra & \cE_J^1 \\
	&\parallel &\empty &\parallel \\
	& E_1^{-1, 1}=\ext^{1}_{G}\big(i_{I}(\bm\beta,\ul{\lambda}), \bI_I(\bm\beta,\ul{\lambda})\big) &\stackrel{d_1^{-1,1}}{\longrightarrow}  & E_1^{0, 1}=\ext^{1}_{G}\big(i_{I}(\bm\beta,\ul{\lambda}), \bI_J(\bm\beta,\ul{\lambda})\big).
\end{matrix}
\end{equation*}
The second row is given by
\begin{equation*}
\begin{matrix}
	& \cE_I^2=\bigwedge\nolimits^{\!2}\!\;\homo(\bL_{\beta,I}(L),E)&\lra & \cE_J^2=\bigwedge\nolimits^{\!2}\!\;\homo(\bL_{\beta,J}(L),E) \\
	&\parallel &\empty &\parallel \\
	&  E_1^{-1, 2}=\ext^{2}_{G}\big(i_{I}(\bm\beta,\ul{\lambda}), \bI_I(\bm\beta,\ul{\lambda})\big) &\stackrel{d_1^{-1,2}}{\longrightarrow}& E_1^{0, 2}=\ext^{2}_{G}\big(i_{I}(\bm\beta,\ul{\lambda}), \bI_J(\bm\beta,\ul{\lambda})\big).
\end{matrix}
\end{equation*}
Thus $d_1^{-1,2}$ is injective.\;Lemma \ref{degspectral} now proves  (b).\;
\end{proof}

The exact sequence
\begin{equation*}
0 \longrightarrow u_I(\bm\beta,\ul{\lambda}) \longrightarrow i_{I}(\bm\beta,\ul{\lambda}) \longrightarrow v^{\lalg}_I(\bm\beta,\ul{\lambda}) \longrightarrow 0, 
\end{equation*}
 induces a long exact sequence
\begin{equation}\label{compareimages}
\begin{aligned}
0 \longrightarrow \homo_G(v^{\lalg}_I&(\bm\beta,\ul{\lambda}),v^{\ana}_J\big(\bm\beta,\ul{\lambda})\big) \longrightarrow \homo_G(i_{I}\big(\bm\beta,\ul{\lambda}),v^{\ana}_J(\bm\beta,\ul{\lambda})\big) \\
\longrightarrow \homo_G&\big( u_I(\bm\beta,\ul{\lambda}), v^{\ana}_J(\bm\beta,\ul{\lambda})\big)
\longrightarrow \ext^1_G(v_I^{\lalg}\big(\bm\beta,\ul{\lambda}),v^{\ana}_J(\bm\beta,\ul{\lambda})\big) \\
&        \ra\;\ext^1_G(i_{I}\big(\bm\beta,\ul{\lambda}),v^{\ana}_J(\bm\beta,\ul{\lambda})\big) \longrightarrow \ext_G^1\big( u_I(\bm\beta,\ul{\lambda}), v^{\ana}_J(\bm\beta,\ul{\lambda})\big).
\end{aligned}
\end{equation}
Since $\homo_G\big(u_I(\bm\beta,\ul{\lambda}), v^{\ana}_J(\bm\beta,\ul{\lambda})\big)=0$ (a direct consequence of the calculation of the Jordan--H\"{o}lder factors of locally $\bQ_p$-analytic generalized Steinberg representations in Section \ref{JHofanaparastrep}),\;we obtain an injection
\begin{equation}\label{injectionwv}
\ext^1_G(v_I^{\lalg}\big(\bm\beta,\ul{\lambda}),v^{\ana}_J(\bm\beta,\ul{\lambda})\big) \hooklongrightarrow \ext^1_G(i_{I}\big(\bm\beta,\ul{\lambda}),v^{\ana}_J(\bm\beta,\ul{\lambda})\big).
\end{equation}

\begin{thm}\label{analyticExt3.3}Let $J\subseteq I\subseteq\Delta_{\beta}$ such that $|I|=|J|+1$.\;The isomorphism $$\varrho_{I,J}^0: \homo\big(\bZ_{\beta,J}(L),E\big)/\homo\big(\bZ_{\beta,I}(L),E\big) \xrightarrow{\sim} \ext^{1}_{G}\big(i_{I}(\bm\beta,\ul{\lambda}),v^{\ana}_J(\bm\beta,\ul{\lambda})\big)$$ factors through the injection (\ref{injectionwv}).\;We put
\begin{equation}\label{rpoIJ1}
\varrho_{I,J}: \homo\big(\bZ_{\beta,J}(L),E\big)/\homo\big(\bZ_{\beta,I}(L),E\big) \xrightarrow{\sim} \ext^1_G(v^{\lalg}_I(\bm\beta,\ul{\lambda}),v^{\ana}_J(\bm\beta,\ul{\lambda}))
\end{equation}
for the resulting isomorphism.\;If $I=J\cup \{vr\}$,\;we have an isomorphism
\begin{equation}\label{homLE}
	\begin{aligned}
		\iota_{I,J}: \homo(L^{\times},E)& \xrightarrow{\sim}   \homo\big(\bZ_{\beta,J}(L),E\big)/\homo\big(\bZ_{\beta,I}(L),E\big),\\
		\psi& \mapsto [(a_1I_r,\cdots, a_{k-|J|}I_r)\mapsto \psi(a_v/a_{v+1})],
	\end{aligned}
\end{equation}
Thus the composition $\varrho_{I,J}\circ\iota_{I,J}$ gives an isomorphism of $E$-vector spaces:
\begin{equation}\label{homLEext}
	\homo(L^{\times},E) \longrightarrow \ext^{1}_{G}\big(v^{\lalg}_{I}(\bm\beta,\ul{\lambda}),v^{\ana}_J(\bm\beta,\ul{\lambda})\big).
\end{equation}
\end{thm}
\begin{proof}Let $\widetilde{w}_{I}^{\infty}(\bm\beta,\ul{\lambda})$ be the kernel of the surjection $\bigoplus_{qr\in \Delta_{\beta}\backslash I}i_{I\cup\{qr\}}(\bm\beta,\ul{\lambda}) \twoheadrightarrow u_I(\bm\beta,\ul{\lambda})$.\;Moreover,\;since $\homo_G(\widetilde{w}_{I}^{\infty}(\bm\beta,\ul{\lambda}), v^{\ana}_J(\bm\beta,\ul{\lambda}))=0$, we deduce from the long exact sequence of extension groups that the natural map
\begin{equation}\label{inside5556}
\ext^1_G\big(u_I(\bm\beta,\ul{\lambda}),v^{\ana}_J(\bm\beta,\ul{\lambda})\big) \hooklongrightarrow \bigoplus_{qr\in \Delta_{\beta}\backslash I} \ext^1_G\big(i_{I\cup\{qr\}}(\bm\beta,\ul{\lambda}),v^{\ana}_J(\bm\beta,\ul{\lambda})\big)
\end{equation}
is injective.\;For $\Psi\in \homo\big(\bZ_{\beta,J}(L),E\big)$, the long exact sequence (\ref{compareimages}) implies that $[\sE_{I}^J(\bm\beta,\ul{\lambda},\Psi)^0]$ lies in the image of $\ext^1_G(v^{\lalg}_I(\bm\beta,\ul{\lambda}),v^{\ana}_J(\bm\beta,\ul{\lambda}))$ if and only if its image in $\ext^1_G(u_I(\bm\beta,\ul{\lambda}),v^{\ana}_J(\bm\beta,\ul{\lambda}))$ is trivial.\;In view of (\ref{inside5556}), it is enough to show that, for every $qr\in\Delta_\beta\backslash I$, the image of $[\sE_{I}^J(\bm\beta,\ul{\lambda},\Psi)^0]$ in $\ext^1_G(i_{I\cup\{qr\}}(\bm\beta,\ul{\lambda}),v^{\ana}_J(\bm\beta,\ul{\lambda}))$ under the inclusion
$i_{I\cup\{qr\}}(\bm\beta,\ul{\lambda})\hookrightarrow i_{I}(\bm\beta,\ul{\lambda})$ is trivial.\;The natural inclusion induces an isomorphism
$$\homo(\bZ_{\beta,J\cup\{qr\}}(L),E)/\homo(\bZ_{\beta,I\cup\{qr\}}(L),E)\xrightarrow{\sim}\homo(\bZ_{\beta,J}(L),E)/\homo(\bZ_{\beta,I}(L),E).$$
We may thus choose $\Psi_{qr}\in \homo(\bZ_{\beta,J\cup\{qr\}}(L),E)$ whose image $\overline{\Psi}_{qr}$ in $\homo(\bZ_{\beta,J}(L),E)/\homo(\bZ_{\beta,I}(L),E)$ equals $\Psi$.\;By (\ref{IJmap0}), the extension class $[\sE_{I}^J(\bm\beta,\ul{\lambda},\Psi_{qr})^0]$ agrees with $[\sE_{I}^J(\bm\beta,\ul{\lambda},\Psi)^0]$.\;Its pullback along $i_{I\cup\{qr\}}(\bm\beta,\ul{\lambda})\hookrightarrow i_{I}(\bm\beta,\ul{\lambda})$ equals the pullback of $[\sE_{J\cup\{qr\}}^J(\bm\beta,\ul{\lambda},\Psi_{qr})^0]$ along $i_{I\cup\{qr\}}(\bm\beta,\ul{\lambda})\hookrightarrow i_{J\cup\{qr\}}(\bm\beta,\ul{\lambda})$.\;The latter extension is split by the discussion preceding (\ref{IJmap0}).\;This completes the proof.\;
\end{proof}

\begin{cor}
For $1\leq i\leq k-1$, there is an isomorphism of $E$-vector spaces
	\begin{equation}\label{thmintro}
		\homo(L^{\times},E)\xrightarrow{\sim}\homo\big(\bZ_{\beta}(L),E\big)/\homo\big(\bZ_{\beta,ir}(L),E\big)\xrightarrow{\sim }\ext^1_{G}\big(v_{ir}^{\lalg}(\bm\beta,\ul{\lambda}), \st_n^{\ana}(\bm\beta,\ul{\lambda})\big).
	\end{equation}
\end{cor}

\begin{rmk}\label{rem: lgln-ext3}For $\psi\in \homo(L^\times,E)$,\;we describe explicitly the extension $\varrho_{\{ir\},\emptyset}\circ \iota_{\{ir\},\emptyset}(\psi)$ in (\ref{thmintro}).\;Let $\sE_{\{ir\}}^\emptyset(\bm\beta,\ul{\lambda},\iota_{\{ir\},\emptyset}(\psi))^0$ be the extension of $i_{ir}(\bm\beta,\ul{\lambda})$ by $\st_n^{\ana}(\bm\beta,\ul{\lambda})$ associated with $\psi$ as in (\ref{homLEext}).\;By Theorem \ref{analyticExt3.3}, the injection 
$\ext^1_{G}\big(v_{ir}^{\lalg}(\bm\beta,\ul{\lambda}), \st_n^{\ana}(\bm\beta,\ul{\lambda})\big)\hookrightarrow\ext^1_{G}\big(i_{ir}(\bm\beta,\ul{\lambda}),\st_n^{\ana}(\bm\beta,\ul{\lambda})\big)$ in (\ref{injectionwv}) is an isomorphism.\;Equivalently,\;the pullback of $\sE_{\{ir\}}^\emptyset(\bm\beta,\ul{\lambda},\iota_{\{ir\},\emptyset}(\psi))^0$ along $u_{ir}(\bm\beta,\ul{\lambda})\rightarrow i_{ir}(\bm\beta,\ul{\lambda})$ splits.\;Taking the quotient by $u_{ir}(\bm\beta,\ul{\lambda})$, we obtain a representation $\widetilde{{\Sigma}}_i(\bm\beta,\ul{\lambda}, \psi)=[\st_n^{\ana}(\bm\beta,\ul{\lambda})-v_{ir}^{\lalg}(\bm\beta,\ul{\lambda})]$, whose extension class corresponds to $\iota_{\{ir\},\emptyset}(\psi)$ under $\varrho_{\{ir\},\emptyset}$.\;Moreover, $\widetilde{{\Sigma}}_i(\bm\beta,\ul{\lambda}, \psi)$ arises from a locally $\bQ_p$-algebraic extension of $v_{ir}^{\lalg}(\bm\beta,\ul{\lambda})$ by $\st_n^{\lalg}(\bm\beta,\ul{\lambda})$ if and only if $\psi\in \homo_\sm(L^{\times},E)$.\;
\end{rmk}

\subsection{Extensions between locally analytic representations-II}\label{anaext2}
This section follows the approach of \cite[Section 2.3]{2019DINGSimple}.\;We show that some subrepresentations of $\st_n^{\ana}(\bm\beta,\underline{\lambda})$ already contributes to the non-split extensions in Theorem \ref{analyticExt3.3}.\;

Let $\underline{\lambda}:=(\lambda_{1,\sigma}, \cdots, \lambda_{n,\sigma})_{\sigma\in \Sigma_L}\in X_{\Delta}^+$, and let $I\subseteq \Delta_{\beta}$.\;Recall the locally $\bQ_p$-analytic (resp., locally $\bQ_p$-algebraic, resp., smooth) generalized Steinberg representations $\st_{I}^{\ana}(\bm\beta,\underline{\lambda})$ (resp., $\st_{I}^{\lalg}(\bm\beta,\underline{\lambda})$, resp., $\st_{I}^{\infty}(\bm\beta)$)  of $\bL_{\beta,I}(L)$ defined in Definition \ref{dfnlagpstrep} (resp., Definition \ref{dfnlagpstreplalg}).\;

Fix $1\leq i\leq k-1$.\;Since $S^{\beta}(\Delta_{\beta}\backslash\{ir\})=\Delta\backslash\{ir\}$, we have $\op_{\beta,\Delta_{\beta}\backslash\{ir\}}=\op_{\Delta\backslash\{ir\}}$, a maximal parabolic subgroup of $\GLN_n$, and $\bL_{\beta,\Delta_{\beta}\backslash\{ir\}}=\bL_{\Delta\backslash\{ir\}}\cong \GLN_{ir}\times \GLN_{(k-i)r}$.\;Consider the composition
\begin{equation}\label{tautauI}
	\tau_{ir}: \Big(\mathrm{Ind}_{\op_{\Delta\backslash\{ir\}}(L)}^{G}\st_{\Delta_{\beta}\backslash\{ir\}}^\lalg(\bm\beta,\underline{\lambda})\Big)^{\bQ_p-\mathrm{an}} \hooklongrightarrow \Big(\mathrm{Ind}_{\op_{\Delta\backslash\{ir\}}(L)}^{G}\st_{\Delta_{\beta}\backslash\{ir\}}^{\ana}(\bm\beta,\underline{\lambda})\Big)^{\bQ_p-\mathrm{an}} \twoheadlongrightarrow \mathrm{St}_n^{\ana}(\bm\beta,\underline{\lambda}).\;
\end{equation}
Let $\widetilde{\Sigma}_{ir}(\bm\beta,\ul{\lambda})$ be its image.\;

\begin{lem}\label{dfnsumwidetilde}
Let $1\leq i\leq k-1$.\;We have the following commutative diagram with exact rows: 
\[\xymatrix{0 \ar[r] & v_{ir}^{\lalg}(\bm\beta,\underline{\lambda}) \ar@{=}[d]  \ar[r] & \Big(\mathrm{Ind}_{\op_{\Delta\backslash\{ir\}}(L)}^{G}\st_{\Delta_{\beta}\backslash\{ir\}}^{\infty}(\bm\beta)\Big)^\infty\otimes_EL(\underline{\lambda}) \ar@{^(->}[d] \ar[r] & \st_n^{\lalg}(\bm\beta,\underline{\lambda}) \ar@{^(->}[d] \ar[r] & 0  \\
0  \ar[r] & v_{ir}^{\lalg}(\bm\beta,\underline{\lambda}) \ar@{^(->}[d] \ar[r] & \Big(\mathrm{Ind}_{\op_{\Delta\backslash\{ir\}}(L)}^{G}\st_{\Delta_{\beta}\backslash\{ir\}}^\lalg(\bm\beta,\underline{\lambda})\Big)^{\bQ_p-\mathrm{an}} \ar@{^(->}[d] \ar[r]^{\hspace{60pt}\tau_{ir}}_{\hspace{60pt}(\ref{tautauI})} & \widetilde{\Sigma}_{ir}(\bm\beta,\underline{\lambda})\ar@{^(->}[d] \ar[r] & 0\\
0 \ar[r] & v_{ir}^{\ana}(\bm\beta,\underline{\lambda}) \ar[r] & \Big(\mathrm{Ind}_{\op_{\Delta\backslash\{ir\}}(L)}^{G}\st_{\Delta_{\beta}\backslash\{ir\}}^{\ana}(\bm\beta,\underline{\lambda})\Big)^{\bQ_p-\mathrm{an}} \ar[r] & \mathrm{St}_n^{\ana}(\bm\beta,\underline{\lambda}) \ar[r] & 0.}\]
\end{lem}
\begin{proof}By \cite[Theorem]{orlik2015jordan}, the locally $\bQ_p$-algebraic vectors in $\big(\mathrm{Ind}_{\op_{\Delta\backslash\{ir\}}(L)}^{G}\st_{\Delta_{\beta}\backslash\{ir\}}^{\lalg}(\bm\beta,\underline{\lambda})\big)^{\bQ_p-\mathrm{an}}$ form the middle term of the first row, namely $\big(\mathrm{Ind}_{\op_{\Delta\backslash\{ir\}}(L)}^{G}\st_{\Delta_{\beta}\backslash\{ir\}}^{\infty}(\bm\beta)\big)^\infty\otimes_EL(\underline{\lambda})$.\;By transitivity of parabolic induction, we have isomorphisms
\begin{equation*}
	\begin{aligned}
		\Big(\mathrm{Ind}_{\op_{\Delta\backslash\{ir\}}(L)}^{G}\st_{\Delta_{\beta}\backslash\{ir\}}^{\infty}(\bm\beta)\Big)^\infty\otimes_EL(\underline{\lambda})&\;=i_{\emptyset}(\bm\beta,\underline{\lambda})\Big/\sum_{\emptyset \neq J\subseteq \Delta_{\beta}\backslash\{ir\}}i_J(\bm\beta,\underline{\lambda}),\\
		\Big(\mathrm{Ind}_{\op_{\Delta\backslash\{ir\}}(L)}^{G}\st_{\Delta_{\beta}\backslash\{ir\}}^{\ana}(\bm\beta,\underline{\lambda})\Big)^{\bQ_p-\mathrm{an}}&\;=\bI_{\emptyset}(\bm\beta,\underline{\lambda})\Big/\sum_{\emptyset \neq J\subseteq \Delta_{\beta}\backslash\{ir\}}\bI_J(\bm\beta,\underline{\lambda}).\;
	\end{aligned}
\end{equation*}
The exactness of the first and third rows follows from Proposition \ref{axioms} and (\ref{analytic[A3]}), respectively.\;Note that
\begin{equation}\label{scolesumwidetildeidentity}
	\begin{aligned}
		\soc_G \big(\ind_{\op_{\Delta\backslash\{ir\}}(L)}^G \st^{\lalg}_{\Delta_{\beta}\backslash\{ir\}}(\bm\beta,\ul{\lambda})\big)^{\bQ_p-\ana}   
		\cong&\;\soc_G \cF_{\op_{\Delta\backslash\{ir\}}}^G\big(\overline{L}(-\ul{\lambda}),\st_{\Delta_{\beta}\backslash\{ir\}}^{\lalg}(\bm\beta)\big)\\\cong&\; \soc_G \big(i^{\beta}_{\Delta_{\beta}\backslash\{ir\},\Delta_{\beta}}\st_{\Delta_{\beta}\backslash\{ir\}}^{\infty}(\bm\beta)\big)\otimes_EL(\ul{\lambda})=v_{ir}^{\lalg}(\bm\beta,\ul{\lambda}),
	\end{aligned}
\end{equation}
where the first two isomorphisms follow from \cite[Corollary 2.5]{breuil2016socle}, and the last follows from Lemma \ref{SmoothExt2}.\;It remains to prove the exactness of the second row.\;The exactness of the first row shows that $\tau_{ir}$ factors through
\begin{equation}\label{secondfirst}
\tau_{ir}': \Big(\mathrm{Ind}_{\op_{\Delta\backslash\{ir\}}(L)}^{G}\st_{\Delta_{\beta}\backslash\{ir\}}^{\lalg}(\bm\beta,\underline{\lambda})\Big)^{\bQ_p-\mathrm{an}}\Big/v_{ir}^{\lalg}(\bm\beta,\underline{\lambda})\longrightarrow
\mathrm{St}_n^{\ana}(\bm\beta,\underline{\lambda}).
\end{equation}
It remains to show that $\tau_{ir}'$ is injective.\;Let $V$ be an irreducible constituent of $\ker(\tau_{ir}')$.\;By (\ref{scolesumwidetildeidentity}) and the fact that the restriction of $\tau_{ir}'$ to the locally $\bQ_p$-algebraic vectors is injective, $V$ is not locally $\bQ_p$-algebraic.\;Thus
\begin{equation}\label{secondfirstform}
V=\mathcal{F}^G_{\op_{\Delta\backslash\{ir\}}}\Big(\overline{L}(-s\cdot\underline{\lambda}),\st_{\Delta_{\beta}\backslash\{ir\}}^\infty(\bm\beta)\Big),
\end{equation}
with $s\cdot\underline{\lambda}\in X^+_{\Delta\backslash\{ir\}}$.\;Since  $\ker(\tau_{ir}')$ is a subquotient of $\sum\limits_{\emptyset \neq I\subseteq \Delta_{\beta}}\bI_I(\bm\beta,\underline{\lambda})$, $V$ is an irreducible constituent of
$\bI_I(\bm\beta,\underline{\lambda})\cong
\mathcal{F}^G_{\op_{\beta,{I}}}\big(\lM_{S^\beta(I)}(-\underline{\lambda}),\pi_{\beta,I}\big)$ for some $\emptyset \neq I\subseteq \Delta_{\beta}$.\;Thus $I\subseteq \Delta_{\beta}\backslash\{ir\}$ and $\st_{\Delta_{\beta}\backslash\{ir\}}^\infty(\bm\beta)$ is  an irreducible subquotient of $i^{\beta}_{I,\Delta_{\beta}\backslash\{ir\}}(\pi_{\beta,I})$ by \cite[Theorem]{orlik2015jordan}.\;This is impossible by Lemma \ref{JHsmooth}.\;
\end{proof}
Since no irreducible constituent of $\widetilde{\Sigma}_{ir}(\bm\beta,\ul{\lambda})/\st_n^{\lalg}(\bm\beta,\underline{\lambda})$ occurs in $\widetilde{\Sigma}_j(\bm\beta,\ul{\lambda})$ for $j\neq i$, the amalgamated sum
\begin{equation*}
	\widetilde{\Sigma}(\bm\beta,\ul{\lambda}):=\bigoplus_{\st_n^{\lalg}(\bm\beta,\underline{\lambda})}^{1\leq i\leq k-1}\widetilde{\Sigma}_{ir}(\bm\beta,\ul{\lambda})
\end{equation*}
is a subrepresentation of $\st_n^{\ana}(\bm\beta,\underline{\lambda})$.\;


For $\sigma\in \Sigma_L$, set $\ul{\lambda}_{\sigma}:=(\lambda_{1,\sigma}, \cdots, \lambda_{n,\sigma})$ and $\ul{\lambda}^{\sigma}:=(\lambda_{1,\sigma'}, \cdots, \lambda_{n,\sigma'})_{\sigma'\in \Sigma_L\backslash\{\sigma\}}$.\;Define $\st_{\Delta_{\beta}\backslash\{ir\}}^{\sigma-\lalg}(\bm\beta,\ul{\lambda}_{\sigma}):= \st_{\Delta_{\beta}\backslash\{ir\}}^{\infty}(\bm\beta)\otimes_E L(\ul{\lambda}_{\sigma})_{\Delta\backslash\{ir\}}$, a locally $\sigma$-algebraic representation of $\bL_{\Delta\backslash\{ir\}}(L)$.\;Locally $\sigma$-analytic parabolic induction, together with \cite[Theorem]{orlik2015jordan} and \cite[Lemmas 2.7 and 2.10]{2019DINGSimple}, gives the following injections of $G$-representations:
\begin{equation}
	\begin{aligned}
		&\big(\ind_{\op_{\Delta\backslash\{ir\}}(L)}^G \st_{\Delta_{\beta}\backslash\{ir\}}^{\infty}(\bm\beta)\big)^{\infty}\otimes_E L(\ul{\lambda})
		\hooklongrightarrow  \big(\ind_{\op_{\Delta\backslash\{ir\}}(L)}^G \st_{\Delta_{\beta}\backslash\{ir\}}^{\sigma-\lalg}(\bm\beta,\ul{\lambda}_{\sigma}) \big)^{\sigma-\ana} \otimes_E L(\ul{\lambda}^{\sigma}) \\
		\hooklongrightarrow&\;\big(\ind_{\op_{\Delta\backslash\{ir\}}(L)}^G \st_{\Delta_{\beta}\backslash\{ir\}}^{\lalg}(\bm\beta,\ul{\lambda}) \big)^{\bQ_p-\ana}.
	\end{aligned}
\end{equation}
Define
\begin{equation}
	\begin{aligned}
		\widetilde{\Sigma}_{ir,\sigma}(\bm\beta,\ul{\lambda}):=\big(\ind_{\op_{\Delta\backslash\{ir\}}(L)}^G \st_{\Delta_{\beta}\backslash\{ir\}}^{\lalg}(\bm\beta,\ul{\lambda}_{\sigma}) \big)^{\sigma-\ana}\otimes_E L(\ul{\lambda}^{\sigma})\big/v_{ir}^{\lalg}(\bm\beta,\ul{\lambda})
	\end{aligned}
\end{equation}
to be the $\sigma$-component of $\widetilde{\Sigma}_{ir}(\bm\beta,\ul{\lambda})$.\;Then
\begin{equation}
	\begin{aligned}
		C_{ir,\sigma}(\bm\beta):=&\;\soc_G\big(\widetilde{\Sigma}_{ir,\sigma}(\bm\beta,\ul{\lambda})/\st_n^{\lalg}(\bm\beta,\underline{\lambda})\big)\\
		\cong&\;\soc_G	 \big(\ind_{\op_{\Delta\backslash\{ir\}}(L)}^G \st_{\Delta_{\beta}\backslash\{ir\}}^{\sigma-\lalg}(\bm\beta,s_{ir}\cdot\ul{\lambda}_{\sigma}) \big)^{\sigma-\ana} \otimes_E L(\ul{\lambda}^{\sigma}),\\
		\cong&\;\cF_{\op_{\Delta\backslash\{ir\}}}^G\big(\overline{L}(-\ul{\lambda}^{\sigma})\otimes_E \overline{L}(-s_{ir,\sigma}\cdot \ul{\lambda}_{\sigma}), \st_{\Delta_{\beta}\backslash\{ir\}}^{\infty}(\bm\beta)\big),
	\end{aligned}
\end{equation}
which appears as an irreducible constituent in $\st_n^{\ana}(\bm\beta,\ul{\lambda})$ with multiplicity $1$ (alternatively, applying Proposition \ref{JHanastein1} to $J=S=\emptyset$,\;the multiplicity equals $\sum_{{w'\in \sW_n,\; \mathrm{supp}(w')=\emptyset}}(-1)^{l(w')}[\lM(-w'\cdot\underline{\lambda}):\overline{L}(-\ul{\lambda}^{\sigma})\otimes_E \overline{L}(-s_{ir,\sigma}\cdot \ul{\lambda}_{\sigma})]=1$).\;In particular,\;$\widetilde{\Sigma}_{ir,\sigma}(\bm\beta,\ul{\lambda})$ admits a subrepresentation of the form
\begin{equation}\label{dfnsigmaisigma}
	\Sigma_{ir,\sigma}(\bm\beta,\ul{\lambda}):=\Big[\st_n^{\lalg}(\bm\beta,\ul{\lambda})\hdashrule[0.5ex]{0.5cm}{1pt}{1pt}C_{ir,\sigma}(\bm\beta)\Big],
\end{equation}
where the dash line in $\bullet\hdashrule[0.5ex]{0.5cm}{1pt}{1pt}\bullet$ means a possible split extension of representations.\;The  amalgamated sum 
 \begin{equation}\label{substanrep}
		\begin{aligned}
			&\Sigma_{ir}(\bm\beta,\ul{\lambda}):=\bigoplus^{\sigma\in \Sigma_L}_{\st_n^{\lalg}(\bm\beta,\ul{\lambda})}\Sigma_{ir,\sigma}(\bm\beta,\ul{\lambda})=\Big[\st_n^{\lalg}(\bm\beta,\ul{\lambda})\hdashrule[0.5ex]{0.5cm}{1pt}{1pt}\bigoplus_{\sigma\in \Sigma_L}C_{ir,\sigma}(\bm\beta)\Big]\hookrightarrow \widetilde{\Sigma}_{ir}(\bm\beta,\ul{\lambda}).\;
					\end{aligned}
	\end{equation}
Then $\st_n^{\ana}(\bm\beta,\underline{\lambda})$ admits a subrepresentation of the  form
\begin{equation}\label{substanreptwolayer}
	 \Sigma(\bm\beta,\ul{\lambda})=\bigoplus^{1\leq i\leq k-1}_{\st_n^{\lalg}(\bm\beta,\ul{\lambda})} \Sigma_{ir}(\bm\beta,\ul{\lambda})=\Big[\st_n^{\lalg}(\bm\beta,\underline{\lambda})\hdashrule[0.5ex]{0.5cm}{1pt}{1pt}\bigoplus_{\substack{\sigma\in \Sigma_L\\1\leq i\leq k-1}}C_{ir,\sigma}\Big]
	\hooklongrightarrow\widetilde{\Sigma}(\bm\beta,\ul{\lambda})\subseteq\st_n^{\ana}(\bm\beta,\underline{\lambda}).
\end{equation}
\begin{rmk} For $V\in \{\widetilde{\Sigma}_{ir}(\bm\beta,\ul{\lambda}),\widetilde{\Sigma}_{ir,\sigma}(\bm\beta,\ul{\lambda}),\Sigma_{ir}(\bm\beta,\ul{\lambda}),\Sigma(\bm\beta,\ul{\lambda}),\st_n^{\ana}(\bm\beta,\ul{\lambda})\}$,\;we usually expect $\soc_{G}V=\st_n^{\lalg}(\bm\beta,\ul{\lambda})$,\;then the dash lines in  (\ref{dfnsigmaisigma}),\;(\ref{substanrep}) and (\ref{substanreptwolayer}) are non-split.\;In particular,\;the subrepresentation $\Sigma_{ir,\sigma}(\bm\beta,\ul{\lambda})$ (resp.,\;$	\Sigma_{ir}(\bm\beta,\ul{\lambda})$,\;resp.,\;$	 \Sigma(\bm\beta,\ul{\lambda})$) is the first two layers of the socle filtration of $\widetilde{\Sigma}_{ir,\sigma}(\bm\beta,\ul{\lambda})$ (resp.,\;$\widetilde{\Sigma}_{ir}(\bm\beta,\ul{\lambda})$, resp.,\;$\st_n^{\ana}(\bm\beta,\underline{\lambda})$).\;
\end{rmk}

\begin{pro}\label{threeisoextensions}
Let $1\leq i,j\leq k-1$ and $\sigma\in \Sigma_L$.\;For $\ast\in\{\emptyset,\sigma\}$,\;the following natural morphisms are isomorphisms
\begin{eqnarray}
\ext^1_G\big(v_{ir}^{\lalg}(\bm\beta,\ul{\lambda}), \Sigma(\bm\beta,\ul{\lambda})\big) &\xrightarrow{\sim}& \ext^1_G\big(v_{ir}^{\lalg}(\bm\beta,\ul{\lambda}), \st_n^{\ana}(\bm\beta,\ul{\lambda})\big),\nonumber\\
\ext^1_G\big(v_{ir}^{\lalg}(\bm\beta,\ul{\lambda}), \Sigma_{jr,\ast}(\bm\beta,\ul{\lambda})\big) &\xrightarrow{\sim}& \ext^1_G\big(v_{ir}^{\lalg}(\bm\beta,\ul{\lambda}), \widetilde{\Sigma}_{jr,\ast}(\bm\beta,\ul{\lambda})\big), \nonumber
\end{eqnarray}
\end{pro}
\begin{proof}By Proposition \ref{JHanastein1} (applied to $w=s_{jr,\sigma}$), every irreducible constituent of any of the quotients $\st_n^{\ana}(\bm\beta,\ul{\lambda})/\Sigma(\bm\beta,\ul{\lambda})$, $\widetilde{\Sigma}_{jr}(\bm\beta,\ul{\lambda})/\Sigma_{jr}(\bm\beta,\ul{\lambda})$, and $\widetilde{\Sigma}_{jr,\sigma}(\bm\beta,\ul{\lambda})/\Sigma_{jr,\sigma}(\bm\beta,\ul{\lambda})$ is of the form
$\cF_{\op_{\beta,{J}}}^G(\overline{L}(-s\cdot \ul{\lambda}),W)$ with  $\leg(s)>1$ and some $J\subseteq\Delta_{\beta}$.\;The asserted isomorphisms therefore follow from \cite[Lemma 2.26]{2019DINGSimple} and d\'evissage.\;
\end{proof}

The following result is proved as in \cite[Lemma 2.28]{2019DINGSimple}.

\begin{lem}\label{twoisoextensions}
Let $1\leq i,j\leq k-1$.\;For every $\sigma\in \Sigma_L$, we have\\
\noindent
(1) $\dim_E\ext^1_G\Big(v_{ir}^{\lalg}(\bm\beta,\ul{\lambda}), \big(\ind_{\op_{\Delta\backslash\{jr\}}(L)}^{G} L(s_{jr, \sigma}\cdot \ul{\lambda})_{\Delta\backslash\{jr\}}\otimes_E \st_{\Delta_{\beta}\backslash\{jr\}}^{\infty}(\bm\beta)\big)^{\bQ_p-\ana}\Big)=\left\{
\begin{array}{ll}
1, & \hbox{$i=j$;} \\
0, & \hbox{$i\neq j$}.
\end{array}
\right.$\\
(2) $\dim_E\ext^1_G\big(v_{ir}^{\lalg}(\bm\beta,\ul{\lambda}), C_{jr,\sigma}(\bm{\beta})\big)=\left\{
\begin{array}{ll}
1, & \hbox{$i=j$;} \\
0, & \hbox{$i\neq j$}.
\end{array}
\right.$

\end{lem}

\begin{rmk}For each $1\leq i\leq k-1$, let $\Pi^{ir,\sigma}_1\in \ext^1_G\big(v_{ir}^{\lalg}(\bm\beta,\ul{\lambda}), C_{ir,\sigma}(\bm\beta)\big)$ be a nonzero extension class, unique up to scalar.\;We expect the representation corresponding to $\Pi^{ir,\sigma}_1$ to occur as a subrepresentation of the conjectural  representation $\Pi^{ir}(\underline{k},\mathrm{WD})$ in \cite[Conjecture 1.1]{breuil2019ext1}.\;
\end{rmk}

\begin{pro}\label{lastextensions}
Let $1\leq i,j\leq k-1$.
\begin{itemize}
	\item[(1)]If $i\neq j$, then the following natural morphism is an isomorphism:
	\begin{equation*}
		\ext^1_G\big(v_{ir}^{\lalg}(\bm\beta,\ul{\lambda}), \st_n^{\lalg}(\bm\beta,\ul{\lambda})\big)\xrightarrow{\sim} \ext^1_G\big(v_{ir}^{\lalg}(\bm\beta,\ul{\lambda}),\Sigma_{jr}(\bm\beta,\ul{\lambda})\big) .
	\end{equation*}
	\item[(2)]The following natural morphism is an isomorphism:
	\begin{equation*}
		\ext^1_G\big(v_{ir}^{\lalg}(\bm\beta,\ul{\lambda}), \Sigma_{ir}(\bm\beta,\ul{\lambda})\big)\xrightarrow{\sim} \ext^1_G\big(v_{ir}^{\lalg}(\bm\beta,\ul{\lambda}),\Sigma(\bm\beta,\ul{\lambda})\big).
	\end{equation*}
	\item[(3)] For $\sigma\in \Sigma_L$ and $\ast\in\{\emptyset,\sigma\}$, we have isomorphisms 
	\begin{equation}\label{equ: Lsigma}
		\begin{aligned}
			\homo_{\ast}(\bZ_{\beta}(L),E)/\homo_{\ast}(\bZ_{\beta,ir}(L),E)\xrightarrow{\sim} \ext^1_G\big(v_{ir}^{\lalg}(\bm\beta,\ul{\lambda}), \Sigma_{ir,\ast}(\bm\beta,\ul{\lambda})\big).
		\end{aligned}
	\end{equation}
	In particular,\;$\dim_E \ext^1_G\big(v_{ir}^{\lalg}(\bm\beta,\ul{\lambda}),  \Sigma_{ir,\sigma}(\bm\beta,\ul{\lambda}) \big)= 2$.
\end{itemize}
\end{pro}
\begin{proof}The proof is the same as that of \cite[Lemma 2.29]{2019DINGSimple}, with ``[31, Theorem 1]" and ``[Lemma 2.28(2)]" replaced by Proposition \ref{smoothExt4} and Lemma \ref{twoisoextensions}(2), respectively.\;$(3)$ follows from Theorem \ref{analyticExt3.3}, Proposition \ref{threeisoextensions}, Lemma \ref{lastextensions} (2), and the argument of \cite[Proposition 2.30]{2019DINGSimple}.\;
\end{proof}

\begin{rmk}\label{rmkforvsa} We suppose that the central character of $\pi_{\beta}(\ul{\lambda})$ is unitary (that is,\;take values in $\cO_E^{\times}$).\;
\begin{itemize}
	\item[(1)]  For any $I\subseteq \Delta_{\beta}$,\;by \cite[Theorem B]{sorensen2013proof} applied to $\res_{L/\bQ_p}\bL_{\beta,I}$,\;$\st_{I}^{\lalg}(\bm\beta,\ul{\lambda})$ admits a $\bL_{\beta,I}(L)$-invariant norm whose completion gives a Banach representation.\;If this Banach representation is admissible,\;then 
	$\bL_{\beta,I}(L)$ is very strongly admissible;\;that is,\;it admits a continuous $\bL_{\beta,I}(L)$-equivariant $E$-linear embedding into an admissible continuous representation of $\bL_{\beta,I}(L)$ on an $E$-Banach space;\;see \cite[Proposition-Definition 6.2.3]{emerton2017locally}.\;
	\item[(2)] By \cite[Proposition 1.18]{EmpLfun}, $\pi_{\beta}(\ul{\lambda})$ admits a universal  unitary completion $\widehat{V}$,\;which is \underline{non-admissible} if $r>1$.\;By \cite[Proposition 3.1]{Finiteslope}, the universal unitary completion of $i_{\emptyset,I}(\bm\beta,\ul{\lambda})$ is  $\big(\ind_{\op_{\beta}(L)\cap \bL_{\beta,I}(L)}^{\bL_{\beta,I}(L)}\widehat{V}\big)^{\cC^0}$. The same argument as in \cite[Proposition 2.23]{2019DINGSimple} therefore gives
	\begin{equation}\label{unitarycomplstein}
		\st_{I}^{\lalg}(\bm\beta,\ul{\lambda})\hookrightarrow 	\st_{I}^{\lalg}(\bm\beta,\ul{\lambda})^{\univ}:=\big(\ind_{\op_{\beta}(L)\cap \bL_{\beta,I}(L)}^{\bL_{\beta,I}(L)}\widehat{V}\big)^{\cC^0}\big/\sum_{\emptyset\neq J\subseteq I}\big(\ind_{\op_{\beta,J}(L)\cap \bL_{\beta,I}(L)}^{\bL_{\beta,I}(L)}\pi_{\beta,J}(\ul{\lambda})\big)^{\cC^0}.
	\end{equation}
	The $p$-adic local Langlands philosophy suggests that $\st_{n}^{\lalg}(\bm\beta,\ul{\lambda})^{\univ}:=\st_{\Delta_{\beta}}^{\lalg}(\bm\beta,\ul{\lambda})^{\univ}$ should admit a rich family of admissible quotients.\;When $r=1$ and $\ul{\lambda}=\ul{0}$,\;the continuous (Borel) Steinberg representation  is an admissible unitary completion of $\st_n^{\lalg}(\ul{0})$.\;
\item[(3)] There exists a possible way to prove  $\soc_{G}\st_n^{\ana}(\bm\beta,\ul{\lambda})=\st_n^{\lalg}(\bm\beta,\ul{\lambda})$.\;Suppose that an irreducible consistent $C_{w,J}(\bm\beta)$ (see Section \ref{JHofanaparastrep}) with $w\neq 1$ lies in the socle,\;then $\st_{n}^{\lalg}(\bm\beta,\ul{\lambda})^{\univ}$ induces a $G$-invariant lattice on $C_{w,J}(\bm\beta)$.\;If the condition in  \cite[Corollary 3.5]{breuil2016socle}  does not hold,\;we get a contradiction.\;This argument works for the Borel case $r=1$ and $\pi=|\cdot|_L^{\frac{1-n}{2}}$.\;
\end{itemize}
\end{rmk}

\section{Further expectations}\label{HigherDrTower}

This section outlines a possible parabolic analogue of \cite{wholeLINV}.\;Its purpose is primarily to present a tentative picture and a collection of guiding principles for future developments, rather than a complete theory.\;Although we record several precise statements, much of the discussion is conjectural.\;In particular, we propose what we expect to be the appropriate higher extension groups for the questions considered here.\;At present, however, the available computational techniques do not allow us to analyze these groups in the desired generality.\;We plan to pursue these questions in future work.\;We assume in this section that $L=\bQ_p$ for simplicity.\;For general $L$,\;we use the locally $\bQ_p$-analytic version in \cite[Remark 1.3.11]{wholeLINV}.\;

\subsection{On $\ext^{\bullet}$ between locally $\bQ_p$-analytic generalized Steinberg representations}

When $r>1$,\;the higher extension groups between locally $\bQ_p$-analytic generalized Steinberg representations behave quite differently from those in the Borel case.\;For any $I_2\subseteq I_1\subseteq  \Delta_{\beta}$,\;by Proposition \ref{analyticexact},\;
\begin{equation}
	M^{\bullet}_{I_1,I_2}:=\ext^\bullet_{G,Z}(\bC^{\ana}_{I_1,\bullet}(\bm\beta,\ul{\lambda}),\bC^{\ana}_{I_2,\bullet}(\bm\beta,\ul{\lambda}))=\ext^{\bullet}_{G,Z}(v^{\ana}_{I_1}(\bm\beta,\ul{\lambda}),v^{\ana}_{I_2}(\bm\beta,\ul{\lambda})).\;
\end{equation}
By the same argument as in \cite[Lemma 3.3.2 $\&$ Proposition 3.3.3]{wholeLINV} and Lemma \ref{analyticExt1analyticExtC1},\;$M^{\bullet}_{I_1,I_2}$ can be identified with (see \cite[(387),\;(391),\;(593)]{wholeLINV} using its notation)
\begin{equation}
	M^{\bullet}_{I_1,I_2}\cong \ext^{\bullet-|I_1\backslash I_2|}_{\PGLN_n(\bQ_p)}(\bC_{S^{\beta}(I_1),\Delta},\bC_{S^{\beta}(I_2),\Delta}).\;
\end{equation}
Applying \cite[Theorem 4.5.2]{wholeLINV} to the case $(I_0,I_1)=(S^{\beta}(I_1),\Delta)$ and $(I_2,I_3)=(S^{\beta}(I_2),\Delta)$,\;we deduce (using the notation of  \cite[Theorem 4.5.2]{wholeLINV})
\begin{thm}\label{parabolichigherext}
	Set $J'=S^{\beta}(I_2)\cup (\Delta\backslash S^{\beta}(I_1))$.\;We have the following results.\;
	\begin{itemize}
		\item[(i)] If $\ext^{h}_{G,Z}\big(v^{\ana}_{I_1}(\bm\beta,\ul{\lambda}),v^{\ana}_{I_2}(\bm\beta,\ul{\lambda})\big)\neq 0$,\;then $|I_1\backslash I_2|\leq h\leq n^2-1-| I_1\backslash I_2|$.\;
		\item[(ii)] Define $\BE_{I_1,I_2}(\bm\beta):=\ext_{G,Z}
		^{|I_1\backslash I_2|}\big(v^{\ana}_{I_1}(\bm\beta,\ul{\lambda}),v^{\ana}_{I_2}(\bm\beta,\ul{\lambda})\big)$.\;It admits a canonical decreasing filtration
		\begin{equation}
			\begin{aligned}
				0=\fil^{-|J'|+1}(\BE_{I_1,I_2}(\bm\beta))\subseteq \fil^{-|J'|}(\BE_{I_1,I_2}(\bm\beta))\subseteq
				\cdots\subseteq \fil^{-(n-1)}(\BE_{I_1,I_2}(\bm\beta))=\BE_{I_1,I_2}(\bm\beta)
			\end{aligned}
		\end{equation}
		with graded piece $\fil^{-l}(\BE_{I_1,I_2}(\bm\beta))/\fil^{-l+1}(\BE_{I_1,I_2}(\bm\beta))\cong E_{2,J',\Delta}^{-l,l+(n-1)-2|J'|}$ admitting a basis induced from $\Psi_{J',J,l}$ for $|J'|\leq l\leq n-1$.\;
	\end{itemize}
\end{thm}

The case $r>1$ differs fundamentally from the Borel case $r=1$.\;Applying \cite[Theorem 5.3.6]{wholeLINV}, we obtain the following result.
\begin{thm}\label{decompforextgps} Suppose that $r>1$.\;For $I_3\subseteq I_2\subseteq I_1\subseteq \Delta_{\beta}$,\;the natural cup-product map
	\begin{equation}\label{kappacup}
		\kappa_{I_1,I_3}^{I_2}:	\BE_{I_1,I_2}(\bm\beta)\otimes_E \BE_{I_2,I_3}(\bm\beta)\xrightarrow{\cup} \BE_{I_1,I_3}(\bm\beta)
	\end{equation}
is an isomorphism.\;In particular,\;$\BE_{I_1,I_3}(\bm\beta)\cong \homo(L^{\times},E)^{\otimes |I_1\backslash I_3|}$.\;
	\end{thm}
\begin{proof}The case for $I_1=I_3$ is obvious.\;If $|I_1\backslash I_3|=1$,\;this is Theorem  \ref{analyticExt3.3}.\;If $|I_1\backslash I_3|>1$,\;(\ref{kappacup}) follow from the fact 
that $S^{\beta}(I_1)\backslash S^{\beta}(I_3)$ is not  an interval.\;The last assertion follows by choosing a sequence $I_3=T_0\subseteq T_1\subseteq J_2\subseteq \cdots \subseteq T_{|I_1\backslash I_3|}=I_1$ such that $|T_s\backslash T_{s-1}|=1$.\;
\end{proof}

\begin{rmk}
	Pur $r=1$ and $\pi=|\cdot|_L^{\frac{1-n}{2}}$.\;By \cite[Corollary 1.3.7]{wholeLINV},\;for each $I_2\subseteq I_1\subseteq \Delta$ satisfying $|I_1\backslash I_2|>1$,\;we set
	\[\BE^{<}_{I_1,I_3}(\bm{\beta}):=\sum_{I_3\subsetneq I_2\subsetneq I_1}\im(\kappa_{I_1,I_3}^{I_2})\subseteq \BE_{I_1,I_3}.\]
	Then $\dim_E\BE_{I_1,I_3}(\bm{\beta})/\BE^{<}_{I_1,I_3}(\bm{\beta})=1$ if $I_1\backslash I_3$ is an interval,\;i.e.,\;$I_1\backslash I_3=\{a,a+1,\cdots,b\}$ for some $1\leq a<a+2\leq b\leq n-1$.\;In particular,\;put
	\[\bE_{n-1}:=\bE_{\Delta,\emptyset}(\bm{\beta}).\]
	Then $\dim_E\bE_{n-1}=2\dim_E\bE_{n-2}+\sum_{m=0}^{n-3}\dim_E\bE_{m}$,\;where $\dim_E\bE_{0}=1$ and $\dim_E\bE_{1}=2$.\;On the other hand, on the Galois side, we consider the following universal extensions between special $(\varphi,\Gamma)$-modules.\;Let $\cR_E$ be the Robba ring over $\bQ_p$ with coefficients in $E$.\;Let ${\ccyc}:\gal_{\bQ_p}\rightarrow \bZ_p^\times$ be the $p$-adic cyclotomic character.\;We define $\cE_0=E$ and define $\cD_0=\cR_E$.\;For each $t\geq 1$,\;we inductively define the universal extension group  $\cE_t=\ext^1_{g}(\cD_{t-1},\cR_E(\ccyc^t))$ and then $\cD_{t}$ as the following universal de Rham extension (see that discussion before \cite[Theorem 1.3.12]{wholeLINV})
	\[0\rightarrow \cE_t\otimes_E\cR_E(\ccyc^t) \rightarrow \cD_{t}\rightarrow \cD_{t-1}\rightarrow 0.\]
	Then $\{\dim_E\cE_t\}_{t\geq 0}$ satisfies the same recursion formula with $\{\dim_E\bE_t\}_{t\geq 0}$.\;Thus,\;we see that $\dim_E\cE_{n-1}=\dim_E\bE_{n-1}$.\;Indeed, under certain hypotheses,\;Qian further constructs a collection of  isomorphisms $\{r_t:\bE_t\xrightarrow{\sim}\cE_t\}$ in \cite[Theorem 1.3.18,\;Section 8.2]{wholeLINV}.\;Such results indicate that the higher extension groups between locally analytic generalized (Borel) Steinberg representations play a crucial role in identifying the automorphic counterparts of Fontaine-Mazur  $\sL$-invariants in certain $p$-adic semistable Galois representation.\;
\end{rmk}

When $r>1$, the dimension of $\BE_{\Delta_{\beta},\emptyset}(\bm\beta)$, namely $2^{k-1}$, is considerably smaller than the corresponding dimension $b_{k-1}$ of the space of universal de Rham extensions (see Remark \ref{parauniversalfiltered},\;$b_{k-1}=\dim_E\cE_{k-1}(\pi,\lambda)$),\;where $b_{k-1}$ satisfies
$b_{k-1}=(1+r^2)b_{k-2}+r^2\sum_{m=0}^{k-3}b_m$.\;Thus Theorem \ref{parabolichigherext} suggests that the collection of higher extension groups $\{\BE_{I_1,I_2}(\bm\beta)\}_{I_2\subseteq I_1\subseteq\Delta_{\beta}}$ does not capture all the expected parameters in certain $p$-adic \textit{potentially semistable} Galois representation.\;This leads us to the following perspective,\;motivated by a parabolic analogue of Gehrmann's automorphic $\sL$-invariants and by the discussion in \cite[Section 1.2.4]{wholeLINV}.\;

\subsection{Parabolic version of Gehrmann's automorphic $\sL$-invariants}\label{paraGehr}

We begin with a global setting for the conjectural $p$-adic Langlands program and its local--global compatibility.\;Let $F$ be a number field, let $w$ be a place above $p$ such that $F_w=\bQ_p$, and let $\BG$ be a reductive group satisfying $\BG_w(L)=\mathrm{GL}_n(\bQ_p)$.\;Let $\BA_F$ (resp.,\;$\BI_F$) be the ring of adeles (resp.,\;idele group) of $F$.\;For a compact open subgroup $U^w\subseteq \BG(\BA_F^{\infty,w})$, define the $E$-vector space of $p$-adic automorphic forms  by
\[\widehat{S}(U^w,E):=C^{\mathrm{cont}}(\BG(F)\backslash \BG(\BA_F^{\infty,w})/U^w,\cO_E)\otimes_{\cO_E}E.\]
Let $\bT(U^w)$ be the universal Hecke algebra acting on $\widehat{S}(U^w,E)$.\;Let $\rho_p:\gal_{\bQ_p}\rightarrow \GLN_n(E)$ be a potentially semistable $p$-adic Galois representation occurring in $\widehat{S}(U^w,E)$.\;More precisely, suppose that there is a continuous representation $\rho:\gal(\overline{F}/F)\rightarrow \BG^{\vee}(E)$ with an associated maximal ideal $\fm\subseteq \bT(U^w)$ such that $\rho_p=\rho|_{\gal(\overline{F_w}/F_w)}$.\;Assume that $\bh:=(\hpi_{1}>\hpi_{2}>\cdots>\hpi_{n})$ are Hodge--Tate weights of $\rho_p$.\;Put $\lambda=(\lambda_i)_{1\leq i\leq n}=(\bh_i+(i-1))_{1\leq i\leq n}$.\;Assume that (roughly speaking,\;$\st^{\lalg}_{n}(\bm\beta,\lambda)$ comes from a $p$-adic automorphic form)
\[\iota:\st^{\lalg}_{n}(\bm\beta,\lambda)\hookrightarrow \widehat{S}(U^w,E)[\fm],\]
and the Weil--Deligne representation associated with $\rho_p$ (by Fontaine's theory) is absolutely indecomposable,\;i.e.,\;$N^{k-1}\neq0$ (so its associated smooth representation of $G$ is $\st^{\infty}_{n}(\bm\beta)$ by classical local Langlands correspondence).\;Let $\st^{\lalg}_{n}(\bm\beta,\lambda)^{\;\widehat{}}_{\rho_p}$ be the closure of $\st^{\lalg}_{n}(\bm\beta,\lambda)$ in $\widehat{S}(U^w,E)[\fm]$.\;Passing to locally $\bQ_p$-analytic vectors gives an injection of $G$-representations
\[\iota^{\ana}:\st^{\ana}_{n}(\rho_p)^+:=\big(\st^{\lalg}_{n}(\bm\beta,\lambda)^{\;\widehat{}}_{\rho_p}\big)^{\bQ_p-\ana}\hookrightarrow \widehat{S}(U^w,E)^{\bQ_p-\ana}[\fm].\]

Let $\Dpik=D_{\rig}(\rho_p)$ be the rank-$n$ $(\varphi,\Gamma)$-module over $\cR_E$ associated with $\rho_p$.\;Assume that  $\Dpik$ admits a $\bP_{\beta}$-parabolic filtration by saturated $(\varphi,\Gamma)$-submodules, called a \textit{non-critical special $\Omega_{\beta}$-filtration} in \cite{2022parabolivinv} (an parabolic analogue of triangulation):
\[\cF=\fil_{\bullet}^{\cF} \Dpik: \ 0 =\fil_0^{\cF} \Dpik \subsetneq \fil_1^{\cF}\Dpik \subsetneq \cdots \subsetneq \fil_{k}^{\cF} \Dpik=\Dpik,\]
where $\Dpik_i:=\gr_{i}^{\cF}\Dpik$ is a rank-$r$ $(\varphi,\Gamma)$-module over $\cR_E$ for $1\leq i\leq k$.\;We assume that $\rho_p$ is \underline{non-critical}, meaning that the Hodge--Tate weights of $\fil_i^{\cF}\Dpik$ are $\hpi_{1}>\hpi_{2}>\cdots>\hpi_{ir}$.\;The filtration $\cF$ realizes $\Dpik$ as the  successive extension of the $\{\Dpik_i\}_{1\leq i\leq k}$ (i.e.,\;$	\Dpik=[\Dpik_{1}-\Dpik_{2}-\cdots-\Dpik_{k}]$).\;

Assume $r>1$.\;For each $1\leq i\leq k$, suppose that Breuil's locally analytic $\ext^1$-conjecture holds for $\Dpik_i$ (see \cite{breuil2019ext1}), yielding a locally analytic representation $\pi_{\ana}(\Dpik_i)$ that determines $\Dpik_i$.\;More precisely, for $1\leq i\leq k$, the locally analytic $\ext^1$-conjecture predicts locally $\bQ_p$-analytic representations $\{\pi^{(i)}_{c,j}\}_{1\leq j\leq r-1}$, depending only on $\bm\beta$ and $\lambda$,\;together with an isomorphism
\[\bigwedge\nolimits^{\!j}\!D_{\dr}(D_i)\xrightarrow{\sim }\ext^{1}_{\GLN_r(L)}\big(\pi^{(i)}_{c,j},\pi^{(i)}_{\alg}(\bm\beta,\lambda)\big),\]
where $D_i$ is the underlying absolutely irreducible Deligne--Fontaine module associated with $\Dpik_i$, obtained by forgetting the Hodge filtration on $D_{\pcr}(\Dpik_i)$, and $\pi^{(i)}_{\alg}(\bm\beta,\lambda)=\pi\nu^{k-i}\otimes L(\lambda^i)$.\;Here $L(\lambda^i)$ is the unique irreducible $\bQ_p$-algebraic representation of $\GLN_r(L)$ with highest weight $\lambda^i:=(\lambda_{(i-1)r+1},\ldots,\lambda_{ir})$.\;Define the universal locally $\bQ_p$-analytic extension
\[\pi_{\ana}^{(i)}(\bm\beta,\lambda)^{\univ}:=\Big[\pi^{(i)}_{\alg}(\bm\beta,\lambda)-\bigoplus_{1\leq j\leq r-1}\pi^{(i)}_{c,j}\otimes_E\ext^{1}_{\GLN_r(L)}\big(\pi^{(i)}_{c,j},\pi^{(i)}_{\alg}(\bm\beta,\lambda)\big)^{\vee}\Big].\]
Then the Hodge filtration $\fil_H^{\bullet}(D_{\dr}(\Dpik_i))$ corresponds uniquely to a quotient map $\pi_{\ana}^{(i)}(\bm\beta,\lambda)^{\univ}\twoheadrightarrow \pi_{\ana}(\Dpik_i)$.\;

Consider the following morphism of locally $\bQ_p$-analytic parabolic inductions:
\begin{equation*}
	\begin{aligned}
			\mathrm{PS}^{\univ}(\bm\beta,\lambda):=&\Big(\ind^{G}_{\op_{\beta}(L) }\big(\boxtimes_{i=1}^k\pi_{\ana}^{(i)}(\bm\beta,\lambda)^{\univ}\big)\otimes_E\delta^{1/2}_{\op_\beta}\Big)^{\bQ_p-\ana}\\
		&\twoheadrightarrow\mathrm{PS}^{+}(\rho_p):=\Big(\ind^{G}_{\op_{\beta}(L) }\big(\boxtimes_{i=1}^k\pi_{\ana}(\Dpik_i)\big)\otimes_E\delta^{1/2}_{\op_\beta}\Big)^{\bQ_p-\ana}.\;
	\end{aligned}
\end{equation*}

For $J\subseteq \Delta_{\beta}$, suppose that $\mathrm{PS}^{\univ}(\bm\beta,\lambda)$ (resp., $\mathrm{PS}^{+}(\rho_p)$) admits a unique quotient $Q^{\univ}_J(\bm\beta,\lambda)$ (resp., $Q^{\ana}_J(\rho_p)$) with socle $v_{J}^{\lalg}(\bm\beta,\lambda)$.\;Define the following locally $\bQ_p$-analytic representations if possible (the "enhanced" locally $\bQ_p$-analytic generalized Steinberg representations):
\begin{equation}
\begin{aligned}
	v^{\univ}_J(\bm\beta,\lambda):=&\;\ker\Big(Q^{\univ}_J(\bm\beta,\lambda)\twoheadrightarrow\sum_{K\subsetneq J}Q^{\univ}_K(\bm\beta,\lambda)\Big),\\
	v^{\ana}_J(\rho_p):=&\;\ker\Big(Q^{\ana}_J(\rho_p)\twoheadrightarrow\sum_{K\subsetneq J}Q^{\ana}_K(\rho_p)\Big).\;
\end{aligned}
\end{equation}
These definitions yield injections of $G$-representations $v_{J}^{\lalg}(\bm\beta,\lambda)\hookrightarrow v^{\ana}_J(\rho_p)$ and $v_{J}^{\lalg}(\bm\beta,\lambda)\hookrightarrow v^{\univ}_J(\bm\beta,\lambda)$. In particular,\;denote 
\[\st^{\univ}_{n}(\bm\beta,\lambda):=v^{\univ}_{\emptyset}(\bm\beta,\lambda),\quad \st^{\ana}_{n}(\rho_p):=v^{\ana}_{\emptyset}(\rho_p) \]
the "enhanced" locally $\bQ_p$-analytic Steinberg representations.\;
\begin{rmk}
We expect $\st^{\univ}_{n}(\bm\beta,\lambda)$,\;$\st^{\ana}_{n}(\rho_p)$ and $\st^{\ana}_{n}(\rho_p)^+$ to be admissible quotients of the non-admissible Banach representation $\st_{n}^{\lalg}(\bm\beta,\lambda)^{\univ}$.\;We expect 
$\st^{\univ}_{n}(\bm\beta,\lambda)\twoheadrightarrow\st^{\ana}_{n}(\rho_p)\hookrightarrow \st^{\ana}_{n}(\rho_p)^+$.\;We also expect  $\st^{\ana}_{n}(\bm\beta,\lambda)$ to be a subrepresentation of $\st^{\ana}_{n}(\rho_p)$.\;
\end{rmk}

As in \cite[(3.3)]{gehrmann2019automorphic},\;let 
$Z_{\BG}$ be the center of $\BG$ and $P\BG(F)=\BG(F)/Z_{\BG}(F)$.\;Let $\omega:\BI_F/F^{\times}\rightarrow E$ be a character  and $\epsilon_0:\pi_0(\BG(F\otimes_{\BQ}\BR))\rightarrow \{\pm\}$ be a sign.\;Let $\cA_{\omega}^{\cC_0}(U^w,V;E(\epsilon_0))$ be the $P\BG(F)$-module of functions $ f:\BG(\BA_F^{\infty,w})/U^w\rightarrow V^{\vee}(\epsilon_0)$ such that $f|_{Z_{\BG}(\BA_F^{\infty,w})}=\omega$.\;Consider the $p$-arithmetic cohomology
\[\hH^{\bullet}(P\BG(F),V^{\vee})^{\epsilon_0}:=\hH^{\bullet}(P\BG(F),\cA_{\omega}^{\cC_0}(U^w,V;E(\epsilon_0))).\]
For $I\subseteq \Delta_\beta$,\;put $\hH^{\bullet}(P\BG(F),I)^{\epsilon_0}:=\hH^{\bullet}(P\BG(F),v^{\lalg}_{I}(\bm\beta,\lambda)^{\vee})^{\epsilon_0}$ for simplicity.\;Assume that, for every $I_1\subseteq I_0\subseteq \Delta_{\beta}$ (at least for $(I_1,I_0)=(\emptyset,I_0)$) and some degree $k$,\;the cup-product map (a parabolic analogue of \cite[Condition 1.2.3]{wholeLINV} or \cite[Proposition 3.9]{gehrmann2019automorphic})
\begin{equation*}
	\ext^{|I_0\backslash I_1|}_{G,Z}\big(v^{\lalg}_{I_0}(\bm\beta,\lambda),v^{\lalg}_{I_1}(\bm\beta,\lambda)\big)\otimes_E\hH^{k+|I_1|}\big(P\BG(F),I_1\big)^{\epsilon_0}[\fm]\xrightarrow{\cup}\hH^{k+|I_0|}\big(P\BG(F),I_0\big)^{\epsilon_0}[\fm]
\end{equation*}
is an isomorphism.\;As in \cite[Proposition 3.11]{gehrmann2019automorphic}, we expect the natural embedding $\st^{\lalg}_{n}(\bm\beta,\lambda)\hookrightarrow \st^{\ana}_{n}(\rho_p)^+$ to induce an isomorphism
\[\hH^{k}\big(P\BG(F)^+,\emptyset\big)^{\epsilon_0}[\fm]\xrightarrow{\sim}\hH^{k}\big(P\BG(F),\big(\st^{\ana}_{n}(\rho_p)^+\big)^{\vee}\big)^{\epsilon_0}[\fm].\]
This gives the following commutative diagram:
\[\xymatrix@C=2.5ex{\ext^{|I_0|}_{G,Z}\big(v^{\lalg}_{I_0}(\bm\beta,\lambda),\st^{\lalg}_{n}(\bm\beta,\lambda)\big) \ar[d]  & \otimes_E      & \hH^{k}\big(P\BG(F),\emptyset\big)^{\epsilon_0}[\fm]\ar[r]^{\cup\;\;}\ar[d]^{\sim} & \hH^{k+|I_0|}\big(P\BG(F),I_0\big)^{\epsilon_0}[\fm] \ar[d]^{\sim}\\
\ext^{|I_0|}_{G,Z}\big(v^{\lalg}_{I_0}(\bm\beta,\lambda),\st^{\ana}_{n}(\rho_p)^+\big)&  \otimes_E   & \hH^{k}\big(P\BG(F),\big(\st^{\ana}_{n}(\rho_p)^+\big)^{\vee}\big)^{\epsilon_0}[\fm]\ar[r]^{\;\;\;\;\cup}& \hH^{k+|I_0|}\big(P\BG(F)^+,I_0\big)^{\epsilon_0}[\fm] .}\]
We define a parabolic analogue of Gehrmann's automorphic $\sL$-invariants to be the largest subspace of $\ext^{|I_0\backslash I_1|}_G\big(v^{\lalg}_{I_0}(\bm\beta,\lambda),\st^{\ana}_{n}(\rho_p)^+\big)$ that annihilates the middle term $\hH^{k}\big(P\BG(F),\big(\st^{\ana}_{n}(\rho_p)^+\big)^{\vee}\big)^{\epsilon_0}[\fm]$.\;Moreover, we also consider the following higher extension groups:
\begin{equation*}
	\begin{aligned}
	\BE^{\univ}_{I_1,I_2}(\bm\beta):=\ext_{G,Z}
		^{|I_1\backslash I_2|}\big(v^{\lalg}_{I_1}(\bm\beta,\lambda),v^{\univ}_{I_2}(\bm\beta,\lambda)\big)\rightarrow \BE^{+}_{I_1,I_2}(\rho_p):=\ext_{G,Z}
		^{|I_1\backslash I_2|}\big(v^{\lalg}_{I_1}(\bm\beta,\lambda),v^{\ana}_{I_2}(\rho_p)\big).
	\end{aligned}
\end{equation*}
Such new higher extension groups  may provide additional parameters encoding Galois-theoretic information,\;we do not know how to compute the  in general,\;since the structures of $\st^{\ana}_{n}(\rho_p)^+$ and $\widehat{S}(U^w,E)^{\bQ_p-\ana}[\fm]$ remain difficult to determine.\;

\subsection{Geometric realization from Drinfeld tower}

We now seek an analogue, for the Drinfeld tower, of the argument in \cite[Section 1.2.5]{wholeLINV}.\;Dat \cite{Dat} develops a nonabelian analogue of Lubin--Tate theory over $L$, using the geometry of the Lubin--Tate and Drinfeld towers to realize the local Langlands and Jacquet--Langlands correspondences.\;Let \(D_n\) be the central division algebra over $L$ of invariant \(1/n\).\;The Drinfeld's upper half-space over $L$ is defined by 
\[\BH^{n-1}=\mathbf P_L^{n-1}\setminus\bigcup_{H}H,\]
where \(H\) ranges over the $L$-rational hyperplanes in \(\mathbf P_L^{n-1}\).\;Let $\breve{L}$ be the completion of the maximal unramified extension of $L$, and put $\BH^{n-1}_{\breve{L}}:=\BH^{n-1}\times_L\breve{L}$.\;The Drinfeld tower is a tower of finite étale coverings over $\breve{L}$:
\[\cdots \longrightarrow \breve{\cM}_{\mathrm{Dr},m+1}
\longrightarrow \breve{\cM}_{\mathrm{Dr},m}
\longrightarrow \cdots
\longrightarrow
\breve{\cM}_{\mathrm{Dr},0}=\BH^{n-1}_{\breve{L}}\times \BZ,\]
whose inverse limit is denoted by $\breve{\cM}_{\mathrm{Dr},\infty}$.\;The tower carries commuting actions of $G$, $D_n^\times$, and $W_L$.\;Dat's key insight is to retain the full derived compactly supported cohomology
\[R\Gamma_c:=R\Gamma_c(\breve{\cM}_{\mathrm{Dr},\infty},\overline{\bQ}_{\ell}),\]
as an object with an action of $G\times D_n^\times\times W_L$.\;For every irreducible smooth representation $\pi_n$ of $G$, his main result gives an isomorphism (see \cite[Theorem A]{Dat})
\begin{equation}\label{DatJLLLC}
	H^\bullet\!\left(
	R\homo_{G}\bigl(R\Gamma_c,\pi_n\bigr)\right)\xrightarrow{\sim}\mathrm{JL}_n(\pi_n)\otimes_{\overline{\bQ}_{\ell}}\rec(\pi_n)	\left(\frac{1-n}{2}\right)
\end{equation}
of \(D_n^\times\times W_L\)-representations, where $\mathrm{JL}(\pi_n)$ is the Jacquet--Langlands transfer of $\pi_n$ to $D_n^\times$, and $\rec(\pi_n)$ is the $\overline{\bQ}_{\ell}$-representation of $W_{\bQ_p}$ attached to $\pi_n$ by the classical local Langlands correspondence for $G$.\;

Fix an irreducible representation $\rho$ of $D_n^{\times}$, with central character $\omega_{\rho}:\bQ_p^{\times}\rightarrow \overline{\bQ}_{\ell}^{\times}$, and let $D^b_{\omega_{\rho}}(G)$ be the derived category of smooth $G$-representations with central character $\omega_{\rho}$.\;The complex $R\Gamma_c[\rho]:=R\Gamma_c\otimes^{\bL}_{D_n^{\times}}\rho$ is split (\cite[Proposition 1.3]{Dat}):
\[R\Gamma_c[\rho]\xrightarrow{\sim}\bigoplus_{i}\hH_c^i[\rho],\;\quad \hH_c^i[\rho]:=\hH_c^i\otimes_{D_n^{\times}}\rho.\]
This splitting induces an isomorphism of $\overline{\bQ}_{\ell}$-algebras (\cite[Proposition 1.4]{Dat})
\[\EndO_{D^b_{\omega_{\rho}}(G)}(R\Gamma_c[\rho])\xrightarrow{\sim }\bigoplus_{i\geq j}\ext^{i-j}_{G,\omega_{\rho}}(\hH_c^i[\rho],\hH_c^j[\rho]).\]
Fix an isomorphism $\varsigma:\overline{\bQ}_p\xrightarrow{\sim} \overline{\bQ}_l$.\;In particular,\;if $\rho=\mathrm{JL}_n(\st_n^{\infty}(\bm\beta)\otimes_{E,\varsigma}\overline{\bQ}_l)$, then we have (\cite[Theorem 4.1.2,\;Fact 4.2.5]{Dat})
\begin{itemize}
	\item[(a)] $\hH_c^{j}[\rho]\neq 0$ if and only if $n-1\leq j\leq n-1+(k-1)$.\;
	\item[(b)] For  $0\leq i\leq k-1$,\;$\hH_c^{n-1+i}[\rho]\xrightarrow{\sim}(v^{\infty}_{\{r,\cdots,ir\}}(\bm{\beta})\otimes_{E,\varsigma}\overline{\bQ}_l)\otimes_{\overline{\bQ}_{\ell}} \rec(\pi)|\cdot|_L^{\frac{k-n}{2}-i}$.\;
	\item[(c)] $\EndO_{D^b_{\omega_{\rho}}(G)}(R\Gamma_c[\rho])\xrightarrow{\sim }\EndO_{\overline{\bQ}_{\ell}}(\rec(\pi))\otimes_{\overline{\bQ}_{\ell}} U_k$,\;where $U_k$ is the upper-triangular matrix algebra over $\overline{\bQ}_{\ell}$ of order $k$.\;
\end{itemize}
Put $C:=\widehat{\overline{L}}$, and let $\breve{\cM}_{\mathrm{Dr},m,C}:=\breve{\cM}_{\mathrm{Dr},m}\times_{\breve{L}}C$ and $\breve{\cM}_{\mathrm{Dr},\infty,C}:=\varprojlim_m\breve{\cM}_{\mathrm{Dr},m,C}$.\;To formulate a $p$-adic analogue of (\ref{DatJLLLC}), consider the derived global sections of the de Rham complex of $\breve{\cM}_{\mathrm{Dr},\infty,C}$:
\[R\Gamma_{\dr}(\breve{\cM}_{\mathrm{Dr},\infty,C}):=\varinjlim_mR\Gamma_{\dr}\big(\breve{\cM}_{\mathrm{Dr},m,C},\Omega^{\bullet}_{\breve{\cM}_{\mathrm{Dr},m,C}}\big),\]
where $\Omega^{\bullet}_{\breve{\cM}_{\mathrm{Dr},m,C}}$ is the de Rham complex of $\breve{\cM}_{\mathrm{Dr},m,C}$.\;For any $G$-representation $V$ over $E$,\;put $V_{C}:=V\otimes_EC$.\;
\begin{ques} Can one construct a locally $\bQ_p$-analytic representation $\Pi(\rho_p)$ such that
	\begin{equation}
		\homo_{\cM_{Z,\omega_{\rho}\chi_{\lambda}}(G)}\big(\Pi(\rho_p)_C^{\vee}[n-1],R\Gamma_{\dr}(\breve{\cM}_{\mathrm{Dr},\infty,C})[\rho]\otimes_{C}L(\lambda)^{\vee}_C\big)\cong D_{\dr}(\rho_p)\otimes_EC
	\end{equation}
as filtered modules?\;Here the Hodge filtration $\fil^p_H(-)$ on the RHS is induced by the images of the maps
\begin{equation*}
	\begin{aligned}
			\homo_{\cM_{Z,\omega_{\rho}\chi_{\lambda}}(G)}\big(\Pi(\rho_p)_C^{\vee}&[n-1],F^pR\Gamma_{\dr}(\breve{\cM}_{\mathrm{Dr},\infty,C})[\rho]\otimes_{C}L(\lambda)^{\vee}_C\big)\\&\rightarrow 		\homo_{\cM_{Z,\omega_{\rho}\chi_{\lambda}}(G)}\big(\Pi(\rho_p)_C^{\vee}[n-1],R\Gamma_{\dr}(\breve{\cM}_{\mathrm{Dr},\infty,C})[\rho]\otimes_{C}L(\lambda)^{\vee}_C\big)
	\end{aligned}
\end{equation*}
where $F^pR\Gamma_{\dr}(\breve{\cM}_{\mathrm{Dr},\infty,C})=\varinjlim_mR\Gamma_{\dr}\big(\breve{\cM}_{\mathrm{Dr},m,C},\Omega^{\geq p}_{\breve{\cM}_{\mathrm{Dr},m,C}}\big)$.\;
\end{ques}
\begin{rmk}
Suppose that the (geometric) Hyodo-Kato cohomology is available for $\{\breve{\cM}_{\mathrm{Dr},m,C}\}_{m\geq1}$.\;Let $L_0$ be the maximal unramified subextension of $L/\bQ_p$.\;Then
\[R\Gamma_{\mathrm{HK}}(\breve{\cM}_{\mathrm{Dr},\infty,C}):=\varinjlim_mR\Gamma_{\mathrm{HK}}(\breve{\cM}_{\mathrm{Dr},m,C}),\]
is a finitely generated $\breve{L}_0$-module endowed with an action of $(\varphi,N)$.\;It is related to $R\Gamma_{\dr}(\breve{\cM}_{\mathrm{Dr},\infty,C})$ by the Hyodo--Kato isomorphism
\[R\Gamma_{\mathrm{HK}}(\breve{\cM}_{\mathrm{Dr},\infty,C})\otimes_{\breve{L}_0}C\xrightarrow{\sim }R\Gamma_{\dr}(\breve{\cM}_{\mathrm{Dr},\infty,C}).\]
For example,\;see \cite{CN20a},\;\cite{CDN20b} and \cite{GKHKcoh} for discussions of Hyodo--Kato cohomology and the Hyodo--Kato isomorphism with de Rham cohomology.\;Thus $R\Gamma_{\mathrm{HK}}(\breve{\cM}_{\mathrm{Dr},\infty,C})$ is a filtered $(\varphi,N)$-module.\;One may then ask whether
	\begin{equation}\label{HKcoh}
	\homo_{\cM_{Z,\omega_{\rho}\chi_{\lambda}}(G)}(\Pi(\rho_p)_C^{\vee}[n-1],R\Gamma_{\mathrm{HK}}(\breve{\cM}_{\mathrm{Dr},\infty,C})[\rho]\otimes_{\breve{L}_0}L(\lambda)^{\vee}_C)=D_{\pst}(\rho_p)\otimes_EC
\end{equation}
as a filtered $(\varphi,N)$-module (up to a twist)?\;We refer to \cite{HKfirstcovering} for the discussion on the Hyodo-Kato Cohomology of the first covering of Drinfeld Spaces.\;
\end{rmk}

\begin{rmk}
When $r=1$ and $\lambda=\ul{0}$, there is an isomorphism $R\Gamma_{\mathrm{HK}}(\BH^{n-1})\cong R\Gamma_{\dr}(\BH^{n-1})$.\;Qian \cite[Theorem 1.3.12]{wholeLINV} proves that (see the notation in Example \ref{exampleBorel})
\[\homo_{\cM_{Z,1}(G)}(\st_n^{\ana}(\ul{0})^{\vee}[n-1],R\Gamma_{\dr}(\BH^{n-1}))\xrightarrow{\sim} D_{\dr}(\cD_{n-1}\otimes_{\cR_E}\cR_E(\epsilon^{1-n}))\]
as filtered $(\varphi,N)$-modules,\;where $\cR_{E}(\delta)$ denotes the  rank-$1$ $(\varphi,\Gamma)$-module over $\cR_{E}$ associated to a continuous character $\delta:\bQ_p^\times\rightarrow E^\times$.\;Moreover, there is a commutative diagram of specializations of $(\varphi,N)$-modules:
\[\xymatrix{
	\homo_{\cM_{Z,1}(G)}(\st_n^{\ana}(\ul{0})^{\vee}[n-1],R\Gamma_{\dr}(\BH^{n-1}))
	\ar@{->>}[d]\ar[r]^{\hspace{20pt}\sim} & D_{\dr}(\cD_{n-1}\otimes_{\cR_E}\cR_E(\epsilon^{1-n}))  \ar@{->>}[d]\\
	\homo_{\cM_{Z,1}(G)}(\Pi(\rho_p)^{\vee}[n-1],R\Gamma_{\dr}(\BH^{n-1})) \ar[r]^{\hspace{60pt}\sim} & D_{\dr}(\rho_p\otimes_E\epsilon^{1-n}) ,}\]
where the first vertical map is induced by the injection $\st_n^{\ana}(\ul{0})\hookrightarrow \Pi(\rho_p)$.\;
\end{rmk}

When $r>1$,\;Theorem \ref{decompforextgps} shows that $\homo_{\cM(G)}(\st_{n}^{\ana}(\bm\beta,\lambda)^{\vee}[n-1],R\Gamma_{\dr}(\breve{\cM}_{\mathrm{Dr},\infty,C}))$ is too small: it records only universal extensions of $(\varphi,N)$-modules associated with simple roots in $\Delta_{\beta}$.\;

\begin{ques} Can one describe
\begin{equation}
	\begin{aligned}
		&\homo_{\cM_{Z,\omega_{\rho}\chi_{\lambda}}(G)}\big(\st_{n}^{\univ}(\bm\beta,\lambda)_C^{\vee}[n-1],R\Gamma_{\dr}(\breve{\cM}_{\mathrm{Dr},\infty,C})[\rho]\otimes_{C}L(\lambda)_C^{\vee}\big),\\
		\text{resp.,\;}&\homo_{\cM_{Z,\omega_{\rho}\chi_{\lambda}}(G)}\big(\st_{n}^{\univ}(\bm\beta,\lambda)_{C}^{\vee}[n-1],R\Gamma_{\mathrm{HK}}(\breve{\cM}_{\mathrm{Dr},\infty,C})[\rho]\otimes_{\breve{L}_0}L(\lambda)_C^{\vee}\big)
	\end{aligned}
\end{equation}
	 as a universal filtered module (resp., a filtered $(\varphi,N)$-module)?
\end{ques}
\begin{rmk}\label{parauniversalfiltered}
By the normalized classical local Langlands correspondence \cite[Page 163]{breuil2007first}, the representation $\pi$ determines an irreducible Weil--Deligne representation $\mathbf{r}_{\pi}=\rec(\pi)$ of $W_L$.\;Berger's theory \cite{berger2008equations} then provides an absolutely irreducible rank-$r$ $(\varphi,\Gamma)$-module $\Delta_\pi$ over $\cR_{E}$, the $p$-adic differential equation associated with $\mathbf{r}_{\pi}$.\;It is de Rham with constant Hodge--Tate weight $0$, and $D_{\pst}(\Delta_\pi)$,\;after forgetting the Hodge filtration,\;is isomorphic to the absolutely irreducible Deligne--Fontaine module attached to $\mathbf{r}_{\pi}$ by Fontaine \cite[Proposition 4.1]{breuil2007first}.\;We define $\cE_0(\pi,\lambda)=E$ and define $\cD_0(\pi,\lambda)=\Delta_\pi\otimes_{\cR_E}\cR_E(z^{\bh_{kr}})$.\;For each $1\leq t\leq k-1$,\;we inductively define the extension group  $\cE_t(\pi,\lambda)=\ext^1_{g}(\cD_{t-1}(\pi,\lambda),\Delta_\pi\otimes_{\cR_E}\cR_E(\ccyc^tz^{\bh_{(k-t)r}-t}))$ and then $\cD_{t}(\pi,\lambda)$ as the following universal de Rham extension (discussion before \cite[Theorem 1.3.12]{wholeLINV})
\[0\rightarrow \cE_t(\pi,\lambda)\otimes_E\Delta_\pi\otimes_{\cR_E}\cR_E(\ccyc^tz^{\bh_{(k-t)r}-t}) \rightarrow \cD_{t}(\pi,\lambda)\rightarrow \cD_{t-1}(\pi,\lambda)\rightarrow 0.\]
The universal filtered $(\varphi,N)$-module  should be closely related to $D_{\pst}(\cD_{k-1}(\pi,\lambda)\otimes_{\cR_E}\cR_E(\epsilon^{1-n}))\otimes_EC$.\;
\end{rmk}

\section{Appendix}

\subsection{Review:\;Bernstein-Zelevinsky theory}\label{introBZ}

We now recall briefly the Bernstein-Zelevinsky theory, following  \cite{bernstein1977induced1} and \cite{av1980induced2}.\;Although several of our references work over $\BC$, the various results over $\BC$ are transferred to our setting over $\overline{\bQ}_p$ if we choose (and fix) an isomorphism $\iota:\overline{\bQ}_p\xrightarrow{\sim}\BC$.\;

Let $\underline{\alpha}=(n_1,n_2,\cdots,n_s)$ be an ordered  partition of $n$.\;We put
\[\bP_{\underline{\alpha}}:=\left(\begin{array}{cccc}
\GLN_{n_1} & \ast & \cdots & \ast \\
0 & \GLN_{n_2} & \cdots & \ast \\
\vdots & \vdots & \ddots & \ast \\
0 & 0 & \cdots & \GLN_{n_s} \\
\end{array}\right)\twoheadrightarrow \GLN_{\underline{\alpha}}:=\left(\begin{array}{cccc}
\GLN_{n_1} & 0 & \cdots & 0 \\
0 & \GLN_{n_2} & \cdots & 0 \\
\vdots & \vdots & \ddots & 0 \\
0 & 0 & 0 & \GLN_{n_s} \\
\end{array}\right).\;\]
For each $1\leq i\leq s$, let $\pi_i$ be a smooth representation of $\GLN_{n_i}(L)$ over $\overline{\bQ}_p$.\;Then $\pi_1\otimes \cdots \otimes \pi_s$ is naturally a smooth representation of $\GLN_{\underline{\alpha}}$ over $\overline{\bQ}_p$ and thus of $\bP_{\underline{\alpha}}$ (where the tensor products are taken over $\overline{\bQ}_p$).\;Put
\begin{equation}\label{ind}
\pi_1\times \cdots \times \pi_s:=\big(\ind_{\bP_{\underline{\alpha}}(L)}^{G}\delta_{\bP_{\underline{\alpha}}}^{1/2}(\pi_1\otimes \cdots \otimes \pi_s)\big)^{\infty},
\end{equation}
where $\delta_{\bP_{\underline{\alpha}}}$ is the modulus character of $\bP_{\underline{\alpha}}$.\;If $\pi_1, \ldots, \pi_s$ are irreducible cuspidal representations, \cite[Corollaries 2.2 and 2.3]{av1980induced2} shows that the Jordan--H\"{o}lder factors $\mathbf{JH}(\pi_1\times \cdots \times \pi_s)$ are classified by orientations of a graph.\;The Jordan--H\"{o}lder length of $\pi_1\times \cdots \times \pi_s$ is $2^{f}$, where $f$ is the number of pairs of the form $\{\pi,\pi\nu\}\subseteq \{\pi_1, \cdots, \pi_s\}$.\;

Bernstein-Zelevinsky theory classifies the irreducible smooth representations of $G$ on $\overline{\bQ}_p$-vector spaces in terms of \textit{Zelevinsky segments}.\;A \textit{Zelevinsky segment} of length $s$ is a non-empty set of irreducible cuspidal representations of $\GLN_{m}(L)$ of the form $\{\pi,\pi\nu,\cdots,\pi\nu^{s-1}\}$,\;where $m,s\geq 1$ are  integers,\;and $\pi$ is an irreducible cuspidal representation of $\GLN_m(L)$ over $\overline{\bQ}_p$.\;Since the parabolic subgroups that we use are opposite to those used in \cite{av1980induced2}, the Zelevinsky segment is arranged in descending order in our setting.
\begin{dfn}\label{dfnparastein}Let $\pi$ be an irreducible cuspidal representation of $\GLN_m(L)$ over $\overline{\bQ}_p$.\;We call
 \[\bm\beta_{\pi,s}:=\{\pi\nu^{s-1},\pi\nu^{s-2},\cdots,\pi\}.\]
a Zelevinsky segment of length $s$.\;
\end{dfn}
By \cite[Proposition 2.10]{av1980induced2},\;$\pi\nu^{s-1}\times \cdots \times \pi$ admits a  unique irreducible subrepresentation $\langle\bm\beta_{\pi,s}\rangle$ and a unique irreducible  quotient $\langle\bm\beta_{\pi,s}\rangle^t$,\;in the terminology of \cite[Sections 3.1 and 9.1]{av1980induced2}.\;$\langle\bm\beta_{\pi,s}\rangle^t$ is the  \textit{smooth Steinberg representation} of $\GLN_{ms}(L)$ over $\overline{\bQ}_p$ associated with $\bm\beta_{\pi,s}$.\;


\subsection{Proof of Proposition \ref{axioms}}\label{proofsaximos}

Write $\bL_{\beta,I}=\GLN_{k_1r}\times \cdots\times\GLN_{k_{l_I}r}$ for the  ordered partition $\underline{k}_{I}=(k_1,\cdots,k_{l_I})$ of $k$.\;
\begin{proof}By transitivity of parabolic induction, $\pi_{\beta,I}$ is a subrepresentation of $i^{\beta}_{J,I}(\pi_{\beta,J})$, since it is the unique irreducible subrepresentation of $i^{\beta}_{\emptyset,I}(\pi_{\beta})$.\;This proves the second assertion of $\mathbf{[A1]}$.\;It remains to prove that $\pi_{\beta,J}=J^{\beta}_{I,J}(\pi_{\beta,I})$ (indeed,\;this is essentially contained in \cite[3.4.\;Proposition]{av1980induced2}).\;By definition and \cite[Proposition 1.5(b)]{bernstein1977induced1}, $\pi_{\beta,I}$ is the unique irreducible subrepresentation of
\begin{equation*}
\begin{aligned}
&i^{\beta}_{\emptyset,I}(\pi_{\beta})=\boxtimes_{i=1}^{l_I} \big(\mathrm{Ind}_{\op_{\beta}(L)\cap \GLN_{k_ir}(L)}^{\GLN_{k_ir}(L)}\pi_{\beta}|_{\GLN_{k_ir}(L)}\big)^{\infty}, 
\end{aligned}
\end{equation*}
where we identify $\GLN_{k_ir}(L)$  with the subgroup $1\times \cdots \times \GLN_{k_ir}(L) \times \cdots \times 1$ of $\bL_{\beta,I}(L)$.\;Then 
\begin{equation}\label{f11}
\begin{aligned}
\pi_{\beta,I}&=\boxtimes_{i=1}^{l_I}\langle\bm\beta_{i}\rangle,\quad
\bm\beta_{i}:=\bm\beta_{\pi\nu^{a_i},k_i},\;a_i:=-r(k-2s_{i-1}-k_i)/2+k-s_i.\;
\end{aligned}
\end{equation}
Therefore, by \cite[Proposition 1.5 (a)]{bernstein1977induced1}, we deduce
\begin{equation}\label{f2}
\begin{aligned}
J^{\beta}_{I,J}\pi_{\beta,I}=\boxtimes_{i=1}^{l_I}r^{\GLN_{k_ir}(L)}_{\GLN_{k_ir}(L)\cap \op_{\beta,{J}}(L)}\langle\bm\beta_{i}\rangle,
\end{aligned}
\end{equation}
Let $\underline{k}_J$ be the ordered partition of $k$ associated with $\bL_{\beta,J}(L)$.\;Since $J\subseteq I$, write $\underline{k}_J=(\underline{k_1},\cdots,\underline{k_{l}})$, where $\underline{k_i}=(k_{i,1},\cdots,k_{i,l_i})$ is an ordered partition of $k_i$.\;For $0\leq t\leq l_i$, set $s_{i,0}'=0$ and $s_{i,t}'=\sum_{j=1}^t k_{i,j}$.\;For $1\leq t\leq l_i$,\;put $a_{i,j}:=-r(k-2(s_{i-1}+ s_{i,j-1}')-k_{i,j})/2+k-s_{i-1}-s_{i,j}'$.\;Analogously to (\ref{f11}), we have
\begin{equation}\label{f1}
\begin{aligned}
\pi_{\beta,J}=\boxtimes_{i=1}^{l_I}\boxtimes_{j=1}^{l_i}\langle \bm\beta_{i,j}\rangle,\quad\bm\beta_{i,j}:=\bm\beta_{\pi\nu^{a_{i,j}},k_{i,j}}.\;
\end{aligned}
\end{equation}
By applying \cite[3.4.\;Proposition]{av1980induced2} step by step, we have
\begin{equation}\label{f3}
\begin{aligned}
r^{\GLN_{k_ir}(L)}_{\GLN_{k_ir}(L)\cap \op_{\beta,{J}}(L)}\langle\bm\beta_i\rangle=\boxtimes_{j=1}^{l_i}\langle\bm\beta_i|_{\GLN_{k_{i,j}r}(L)}\rangle \otimes_E \nu^{-\frac{r}{2}(k_i-2s_{i,j-1}'-k_{i,j})}=\boxtimes_{j=1}^{l_i}\langle \bm\beta_{i,j}\rangle.\;
\end{aligned}
\end{equation}
The result follows by comparing
(\ref{f2}), (\ref{f1}) and (\ref{f3}).\;\\
\textbf{Proof of the first assertion of $\mathbf{[A2]}$}.\;We clearly have $i^{\beta}_{I,\Delta_{\beta}}(\pi_{\beta,I})\cap i^{\beta}_{J,\Delta_{\beta}}(\pi_{\beta,J})\supseteq i^{\beta}_{I\cup J,\Delta_{\beta}}(\pi_{\beta,I\cup J})$.\;Let $\mathcal{V}$ be the cokernel of this inclusion.\;We show that $\mathcal{V}=0$.\;By Lemma \ref{resvanshing}, it is enough to prove that $J^{\beta}_{\Delta_\beta,\emptyset}\mathcal{V}=0$.\;Note that
\[i^{\beta}_{I\backslash I\cap J,\Delta_{\beta}}\pi_{I\backslash I\cap J}\cap i^{\beta}_{J,\Delta_{\beta}}(\pi_{\beta,J})\supseteq i^{\beta}_{I,\Delta_{\beta}}(\pi_{\beta,I})\cap i^{\beta}_{J,\Delta_{\beta}}(\pi_{\beta,J})\supseteq i^{\beta}_{I\cup J,\Delta_{\beta}}(\pi_{\beta,I\cup J}),\]
by $\mathbf{[A1]}$.\;Replacing $(I,J)$ by $(I\backslash(I\cap J),J)$, we may assume that $I\cap J=\emptyset$.\;By transitivity of parabolic induction, it suffices to prove that
\begin{equation}
i^{\beta}_{I,I\cup J}(\pi_{\beta,I})\cap i^{\beta}_{J,I\cup J}(\pi_{\beta,J})=\pi_{\beta,I\cup J}.
\end{equation}
By \cite[Proposition 1.5(b)]{bernstein1977induced1}, we may further assume that $I\cup J=\Delta_{\beta}$.\;Since $i^{\beta}_{I,\Delta_{\beta}}(\pi_{\beta,I})$, $i^{\beta}_{J,\Delta_{\beta}}(\pi_{\beta,J})$,\;and $\pi_{\Delta_{\beta}}$ are subrepresentations of $i^{\beta}_{\emptyset,\Delta_\beta}(\pi_{\beta})$, for $\ast\in\{I,J\}$ we have
\begin{equation}\label{semisimple1}
\begin{aligned}
J^{\beta}_{\Delta_\beta,\emptyset}i^{\beta}_{\ast,\Delta_{\beta}}(\pi_{\beta,\ast})&=\big(J^{\beta}_{\Delta_\beta,\emptyset}i^{\beta}_{\ast,\Delta_{\beta}}(\pi_{\beta,\ast})\big)^{\mathrm{ss}}&= \bigoplus_{w\in \sW_n^{\beta,\ast,\emptyset}}w(\pi_{\beta}),
\end{aligned}
\end{equation}
where the first equality follows from the semisimplicity of $J^{\beta}_{\Delta_\beta,\emptyset}i^{\beta}_{\emptyset,\Delta_\beta}(\pi_{\beta})$.\;Under the assumption $I\cup J=\Delta_{\beta}$, we have $\sW_n^{\beta,I,\emptyset}\cap \sW_n^{\beta,J,\emptyset}=\{1\}$.\;Hence
\[\big(J^{\beta}_{\Delta_\beta,\emptyset}i^{\beta}_{I,\Delta_{\beta}}(\pi_{\beta,I})\big)\cap \big(J^{\beta}_{\Delta_\beta,\emptyset}i^{\beta}_{J,\Delta_{\beta}}(\pi_{\beta,J})\big)=\pi_{\beta}=J^{\beta}_{\Delta_\beta,\emptyset}(\pi_{\Delta_{\beta}}).\]
The claim follows.\;\\
\textbf{Proof of the second assertion of $\mathbf{[A2]}$}.\;We clearly have \[i^{\beta}_{I,\Delta_{\beta}}(\pi_{\beta,I})\cap \Big(\sum_{j=1}^m i^{\beta}_{I_j,\Delta_{\beta}}(\pi_{\beta,I_j})\Big)
\supseteq\sum_{j=1}^m i^{\beta}_{I,\Delta_{\beta}}(\pi_{\beta,I})\cap i^{\beta}_{I_j,\Delta_{\beta}}(\pi_{\beta,I_j}).\]Both $J^{\beta}_{\Delta_\beta,\emptyset}\big(i^{\beta}_{I,\Delta_{\beta}}(\pi_{\beta,I})\cap \big(\sum_{j=1}^m i^{\beta}_{I_j,\Delta_{\beta}}(\pi_{\beta,I_j})\big)\big)$ and $J^{\beta}_{\Delta_\beta,\emptyset}\big(\sum_{j=1}^m i^{\beta}_{I,\Delta_{\beta}}(\pi_{\beta,I})\cap i^{\beta}_{I_j,\Delta_{\beta}}(\pi_{\beta,I_j})\big)$ are subrepresentations of the semisimple representation $J^{\beta}_{\Delta_\beta,\emptyset}i^{\beta}_{\emptyset,\Delta_\beta}(\pi_{\beta})$.\;Let $\mathcal{V}'$ be its cokernel.\;All the relevant sums are therefore direct, and $J^{\beta}_{\Delta_\beta,\emptyset}\mathcal{V}'=0$.\;Lemma \ref{resvanshing} now gives $\mathcal{V}'=0$.
\end{proof}

\begin{lem}\label{resvanshing}Let $\mathcal{V}$ be an irreducible subquotient of $i^{\beta}_{\emptyset,\Delta_\beta}(\pi_{\beta})$.\;Then $J^{\beta}_{\Delta_\beta,\emptyset}\mathcal{V}\neq 0$.\;
\end{lem}
\begin{proof}Suppose that $J^{\beta}_{\Delta_\beta,\emptyset}\mathcal{V}=0$.\;Choose $\emptyset\neq I\subset \Delta_{\beta}$ such that $J^{\beta}_{\Delta_\beta,I}\mathcal{V}$ is nonzero and quasi-cuspidal; that is, $J_{\Delta_{\beta},J}^{\beta}\mathcal{V}=0$ for every $J\subsetneq I$ (see also \cite[Section 1.10]{av1980induced2}).\;Let $\cW$ be an irreducible subquotient of $J^{\beta}_{\Delta_\beta,I}\mathcal{V}$.\;Then $\cW$ is cuspidal.\;Moreover,\;by the argument in Section \ref{Brufil},\;$J^{\beta}_{\Delta_\beta,I}i^{\beta}_{\emptyset,\Delta_\beta}(\pi_{\beta})$ admits a canonical increasing filtrartion with graded pieces
\[
\Big\{i^{\beta}_{\emptyset,I}\big(\pi_{\beta}^w\otimes_E\delta^{\beta}_{\emptyset,I,w}\big)\Big\}_{w\in \sW_n^{\beta,\emptyset,I}}.\;\]
Therefore, there exists an element $w\in \sW_n^{\beta,\emptyset,I}$ such that 
\begin{equation*}
\begin{aligned}
0   \neq  \homo_{\bL_{\beta,I}(L)}\big(&\cW,i^{\beta}_{\emptyset,I}\big(\pi_{\beta}^w\otimes_E\delta^{\beta}_{\emptyset,I,w}\big)\big)=\homo_{\bL_{\beta}(L)^w}\big(J^{\beta}_{I,\emptyset}\cW,\pi_{\beta}^w\otimes_E\delta^{\beta}_{\emptyset,I,w}\big).
\end{aligned}
\end{equation*}
This contradicts the cuspidality of $\cW$.\;
\end{proof}


\subsection{Jordan-H\"{o}lder factors of smooth principal series}\label{JHfactorPARA}

Let $\mathbf{JH}(i_{\emptyset}(\bm\beta))$ denote the set of Jordan--H\"{o}lder factors of $i_{\emptyset}(\bm\beta)$.\;Recall that \cite[Section 2]{av1980induced2} gives a bijection between $\mathbf{JH}(i_{\emptyset}(\bm\beta))$ and  orientations of a graph $\Gamma_\pi$.\;The vertices of $\Gamma_\pi$ are $\{v_i:=\pi\nu^{k-i}\}_{1\leq i\leq k}$, and its edges are $\{(v_i,v_{i+1})\}_{1\leq i\leq k-1}$.\;An orientation of $\Gamma_\pi$ is a choice of direction for each edge.\;

Let $\cO(\Gamma_\pi)$ be the set of orientations of $\Gamma_\pi$.\;Define $\Theta:\sW_n^{\beta}\rightarrow \cO(\Gamma_\pi)$ as follows.\;For $w\in \sW_n^{\beta}$, let $\overrightarrow{\Gamma_\pi}(w)$ be the orientation for which $v_i\rightarrow v_{i+1}$ if and only if $w(ir)<w((i+1)r)$.\;For $I\subseteq\Delta_\beta$, let $\overrightarrow{\Gamma_\pi}(I)$ be the orientation defined by $v_i\rightarrow v_{i+1}$ if and only if $ir\in I$.\;Then $\Theta^{-1}(\overrightarrow{\Gamma_\pi}(I))=\{w\in \sW_n^{\beta}:w^{-1}(\Phi^+)\cap \Delta_{\beta}=I\}$;\;note that $w^{-1}(\Phi^+)\cap \Delta=S^{\beta}(I)$ if and only if $w^{-1}(\Phi^+)\cap \Delta_{\beta}=I$ (since $w^{-1}(\Delta^{\beta})=\Delta^{\beta}$).\;Write $I=\{i_1r,\cdots,i_{|I|}r\}$,\;and let $w_{\Delta_{\beta}\backslash I}=f_{\beta}(w_{\Delta_k\backslash\{i_1,\cdots,i_{|I|},0})$,\;where $w_{\Delta_k\backslash\{i_1,\cdots,i_{|I|},0}$ is  the longest element in $\sW_{\Delta_k\backslash\{i_1,\cdots,i_{|I|}\}}$;\;see (\ref{fbetasWbeta}).\;By definition, $(w_{\Delta_{\beta}\backslash I})^2=1$ and $w_{\Delta_{\beta}\backslash I}\in \Theta^{-1}(\overrightarrow{\Gamma_\pi}(I))$.\;

For each  $w\in\sW_n^{\beta}\cong  S_k$, we let $y_{w}^{\infty}(\bm\beta)$ (resp., $z_{w}^{\infty}(\bm\beta)$) be the unique irreducible quotient  (resp., subrepresentation) of $i^{\beta}_{\emptyset,\Delta_\beta}w(\pi_{\beta})$ (see \cite[2.10.\;Proposition]{av1980induced2}).\;Let $\Sigma_0$ be a system of representatives of the classes in $\sW_n^{\beta}$ for the equivalence relation
\[w\sim w' \text{ if and only if } i^{\beta}_{\emptyset,\Delta_\beta}w(\pi_{\beta})\cong i^{\beta}_{\emptyset,\Delta_\beta}w'(\pi_{\beta}).\]
Then \cite[2.11.\;Corollary]{av1980induced2} shows that the following conditions are equivalent:
\begin{equation}
	w\sim w' \Leftrightarrow w^{-1}(\Phi^+)\cap \Delta_{\beta}=(w')^{-1}(\Phi^+)\cap \Delta_{\beta} \Leftrightarrow z_{w}^{\infty}(\bm\beta)\cong z_{w'}^{\infty}(\bm\beta) \Leftrightarrow y_{w}^{\infty}(\bm\beta)\cong y_{w'}^{\infty}(\bm\beta).\;
\end{equation}Together with the Frobenius reciprocity and the semi-simplicity of $J^{\beta}_{\Delta_\beta,\emptyset}i^{\beta}_{\emptyset,\Delta_\beta}(\pi_{\beta})$, the above equivalent conditions imply that
\begin{equation}\label{zwfactors}
	J^{\beta}_{\Delta_\beta,\emptyset}\big(z_{w}^{\infty}(\bm\beta)\big)=\bigoplus_{w\sim w'}w'(\pi_{\beta}).
\end{equation}
These also show that the system $\Sigma_0$ of representatives can be chosen to be $\{w_{\Delta_{\beta}\backslash I}\}_{I\subseteq \Delta_{\beta}}$.\;We put $z_{I}^{\infty}(\bm\beta):=z_{w_{\Delta_{\beta}\backslash I}}^{\infty}(\bm\beta)$.\;By \cite[Section 2]{av1980induced2} or \cite[Page 628-629]{2014Vers}, we see that 
\[\mathbf{JH}(i_{\emptyset}(\bm\beta))=\{z_{I}^{\infty}(\bm\beta)\}_{I\subseteq \Delta_{\beta}}.\]
The following lemma compares the construction in Definition \ref{dfnlagpstreplalg} with the Zelevinsky classification of $\mathbf{JH}(i_{\emptyset}(\bm\beta))$ in \cite[Section 2]{av1980induced2}.\;Its proof follows \cite[Theorem 8.1.2 and Lemmas 8.1.3--8.1.4]{Drimodularvar}.\;
\begin{lem}\label{JHsmooth}
	\begin{itemize}
		\item[(1)] $i_I(\bm\beta)$ (resp.,\;$u_{I}(\bm\beta)$) has a composition series whose successive quotients are the $v_J^{\infty}(\bm\beta)$ with $J\supseteq I$ (resp.,\;$J\varsupsetneq I$), each occurring with multiplicity one.
		\item[(2)] The composition series constructed in part (1) are Jordan--H\"{o}lder series.\;In particular, for each $I\subset \Delta_{\beta}$, $v_I^{\infty}(\bm\beta)$ is irreducible, and $v_{\emptyset}^{\infty}(\bm\beta)=\mathrm{St}_n^{\infty}(\bm\beta)$.\;
		\item[(3)] For any $I\subseteq \Delta_{\beta}$, $v_I^{\infty}(\bm\beta)$ is isomorphic to $z_{I}^{\infty}(\bm\beta)$.\;In particular, for any $I',I''\subseteq \Delta_{\beta}$, the representations $v_{I'}^{\infty}(\bm\beta)$ and $v_{I''}^{\infty}(\bm\beta)$ are isomorphic if and only if $I'=I''$.\;
	\end{itemize}
\end{lem}
\begin{proof}\textbf{Proof of $(1)$}.\;An elementary induction gives an ordering
	$I=J_0,J_1,\cdots,J_N=\Delta_{\beta}$ of the set $\{J\mid I\subseteq J\subseteq\Delta_\beta\}$ such that $J_n\subseteq J_m$ implies $n\leq m$.\;Define
	\[F^l=\sum_{j=l}^{N}i_{J_j}(\bm\beta).\]
	Then we have a filtration $0=F^{N+1}\subset F^N\subset \cdots \subset F^0=i_I(\bm\beta)$
	such that, for each $l=0,\cdots,N$, we have
	\begin{equation*}
		\begin{aligned}
			F^l/F^{l+1}&=\sum_{j=l}^{N}i_{J_j}(\bm\beta)\Big/\sum_{j=l+1}^{N}i_{J_j}(\bm\beta)=i_{J_l}(\bm\beta)\Big/\sum_{j=l+1}^{N}i_{J_j}(\bm\beta)\cap i_{J_l}(\bm\beta)\cong v_{J_l}^{\infty}(\bm\beta).
		\end{aligned}
	\end{equation*}
	The second equality and the third isomorphism follow from \textbf{[A2]} in Proposition \ref{axioms}.\;The result for $u_{I}(\bm\beta)$ follows from the short exact sequence
	\[0\longrightarrow u_{I}(\bm\beta)\longrightarrow i_I(\bm\beta)\longrightarrow v_I^{\infty}(\bm\beta)\longrightarrow 0.\]
	\textbf{Proof of $(2)$}.\;By \cite[Corollary 2.3]{av1980induced2}, every Jordan--H\"older series of $i_{\emptyset}(\bm\beta)$ has length $2^{k-1}$.\;Thus the composition series constructed in (1) is a Jordan--H\"older series, and each $v_J^{\infty}(\bm\beta)$ is irreducible.\;The uniqueness of the irreducible quotient of $i_{\emptyset}(\bm\beta)$ gives $v_{\emptyset}^{\infty}(\bm\beta)=\mathrm{St}_n^{\infty}(\bm\beta)$.\;\\
	\textbf{Proof of $(3)$}.\;The preceding discussion yields a bijection $\alpha$ of $\{I\subseteq \Delta_{\beta}\}$ such that $v_I^{\infty}(\bm\beta)\cong z_{\alpha(I)}^{\infty}(\bm\beta)$ for every $I\subseteq \Delta_{\beta}$.\;We prove that $\alpha(I)=I$ by induction on $|\Delta_{\beta}\backslash I|$.\;Clearly, $\alpha(\Delta_{\beta})=\Delta_{\beta}$.\;Let $I\subseteq \Delta_{\beta}$ with $|\Delta_{\beta}\backslash I|>0$, and assume that $\alpha(J)=J$ whenever $|\Delta_{\beta}\backslash J|<|\Delta_{\beta}\backslash I|$.\;For such $J$, we have
	\begin{equation}\label{zwfactors1}
		J^{\beta}_{\Delta_\beta,\emptyset}v_J^{\infty}(\bm\beta)=\bigoplus_{w\sim w'}w'(\pi_{\beta}).
	\end{equation}
	by (\ref{zwfactors}).\;The same argument as in the discussion following (\ref{semisimple1}) and (\ref{zwfactors}) shows that $w_{\Delta_{\beta}\backslash I}(\pi_{\beta})$ occurs in the semisimplification of $J^{\beta}_{\Delta_\beta,\emptyset}i^{\beta}_{I,\Delta_{\beta}}(\pi_{\beta,I})$.\;By the induction hypothesis,\;$(1)$,\;and (\ref{zwfactors1}), it cannot be a subquotient of $J^{\beta}_{\Delta_\beta,\emptyset}v_J^{\infty}(\bm\beta)$ for any $J\varsupsetneq I$.\;Hence it is a subquotient of $J^{\beta}_{\Delta_\beta,\emptyset}v_I^{\infty}(\bm\beta)=J^{\beta}_{\Delta_\beta,\emptyset}z_{\alpha(I)}^{\infty}(\bm\beta)$.\;Thus $w_{\Delta_{\beta}\backslash I}\sim w_{\Delta_{\beta}\backslash \alpha(I)}$, and therefore $\alpha(I)=w_{\Delta_{\beta}\backslash \alpha(I)}(\Phi^+)\cap \Delta_{\beta}=w_{\Delta_{\beta}\backslash I}(\Phi^+)\cap \Delta_{\beta}=I$.\;This proves the first assertion of (3).\;By definition of the equivalence relation, $z_{I'}^{\infty}(\bm\beta)$ and $z_{I''}^{\infty}(\bm\beta)$ are isomorphic if and only if $I'=I''$.\;Consequently, $v_{I'}^{\infty}(\bm\beta)$ and $v_{I''}^{\infty}(\bm\beta)$ are isomorphic if and only if $I'=I''$.\;This completes the proof.\;
\end{proof}

\subsection{Smooth extensions of generalized Steinberg representations}\label{presmoothext}

The results of this section are not new.\;In \cite[Corollary 18]{2012Orlsmoothextensions}, Orlik computes the smooth extension groups $\ext^{i,\infty}_{G,Z}\big(v_I^{\infty}(\bm\beta), v_J^{\infty}(\bm\beta)\big)$ for arbitrary $I$ and $J$ using the Zelevinsky classification and the theory of semisimple types of Bushnell and Kutzko \cite{bushnell2016admissible}.\;A full account of semisimple types would require substantial additional notation and would lead us away from the main themes of the paper.\;We therefore give a new proof of \cite[Corollary 18]{2012Orlsmoothextensions}, adapting Orlik's argument in the Borel case to the parabolic setting and using the smooth Tits complexes constructed in Section \ref{Titcomp} (set $\ul{\lambda}=\ul{0}$).\;

Let $I\subset \Delta_{\beta}$.\;Write $\bL_{\beta,I}=\GLN_{k_1r}\times \cdots\times\GLN_{k_{l_I}r}$ for the  ordered partition $\underline{k}_{I}=(k_1,\cdots,k_{l_I})$ of $k$.\;For $i\in \BZ_{\geq0}$, recall  in \cite{2012Orlsmoothextensions} that
\begin{equation}
	\begin{aligned}
		\ext^{i,\infty}_{\bL_{\beta,I}(L)}(1,1)=&\;\hH^{i}_{\sm}(\bL_{\beta,I}(L),1)\cong  \bigwedge\nolimits^{\!i}\!\;\homo_{\sm}(\bL_{\beta,I}(L),E),\\
		\ext^{i,\infty}_{\bL_{\beta,I}(L),Z}(1,1)=&\;\hH^{i}_{\sm}(\overline{\bL}_{\beta,I}(L),1)\cong \bigwedge\nolimits^{\!i}\!\;\homo_{\sm}(\overline{\bL}_{\beta,I}(L),E),
	\end{aligned}
\end{equation}
here $\overline{\bL}_{\beta,I}=\bL_{\beta,I}/\bZ$.\;We write
\begin{equation}\label{dfnforlevismoothext}
	\ssE_I^i:=\ext^{i,\infty}_{\bL_{\beta,I}(L)}(\pi_{\beta,I},\pi_{\beta,I}), \quad\sEo_I^i:=\ext^{i,\infty}_{\bL_{\beta,I}(L),Z}(\pi_{\beta,I},\pi_{\beta,I}),
\end{equation}
for short.\;Interpreting the elements of $\ssE_I^i$, $\sEo_I^i$, $\hH^{i}_{\sm}(\bL_{\beta,I}(L),1)$, and $\hH^{i}_{\sm}(\overline{\bL}_{\beta,I}(L),1)$ as Yoneda $i$-extensions, we obtain natural maps
\begin{equation}\label{charactertoanaext}
\begin{aligned}
&\kappa_I^i:\hH^{i}_{\sm}(\bL_{\beta,I}(L),1) \rightarrow \ssE_{I}^i,\quad \text{resp.\;}\overline{\kappa}_I^i: \hH^{i}_{\sm}(\overline{\bL}_{\beta,I}(L),1) \rightarrow \sEo_{I}^i
\end{aligned}
\end{equation}
by tensoring with $\pi_{\beta,I}$.\;The proof of the following lemma relies crucially on the theory of types of Bushnell and Kutzko \cite{bushnell2016admissible}.

\begin{lem}\label{SmoothEXTlemmaPre}
$\kappa_I^i$ and $\overline{\kappa}_I^i$ in (\ref{charactertoanaext}) induce  isomorphisms.\;In particular,\;$\ssE_I^i\cong\bigwedge^i \homo_{\sm}(\bL_{\beta,I}(L),E)$ and $\sEo_{I}^i\cong \bigwedge^i\homo_{\sm}(\overline{\bL}_{\beta,I}(L),E)$.\;
\end{lem}
\begin{proof}
Write $l=l_I$ and $H:=\bL_{\beta,I}(L)=\prod_{j=1}^{l}\GLN_{k_jr}(L)$ for short.\;By (\ref{f11}), we have $\pi_{\beta,I}\cong\boxtimes_{j=1}^{l}\langle\bm\beta_j\rangle$.\;For each $j$, \cite[Corollary 18 and its proof]{2012Orlsmoothextensions}, applied to the segment $\bm\beta_j$, gives
\[
\ext^{\bullet,\infty}_{\GLN_{k_jr}(L)}
\big(\langle\bm\beta_j\rangle,\langle\bm\beta_j\rangle\big)
\cong 
\bigwedge\nolimits^{\!\bullet}\!
\homo_{\sm}\big(\GLN_{k_jr}(L),E\big),
\]
and this isomorphism is induced by tensoring extensions of the trivial representation with $\langle\bm\beta_j\rangle$.\;The K\"{u}nneth formula for smooth extension groups \cite[Chapter X, Theorem 6.1]{borel2000continuous} therefore yields an isomorphism
\begin{equation}\label{unrestrictedLeviSelfExt}
\kappa_I^\bullet:\hH^\bullet_{\sm}(H,1)
\xrightarrow{\sim}
\ssE_{I}^\bullet.
\end{equation}
This proves the assertion concerning $\kappa_I^\bullet$.\;It remains to impose the central character.\;Put $\overline H:=H/\bZ(L)$.\;We use the central-character version of the Hochschild--Serre argument in \cite[Lemma 5 and the proof of Corollary 2]{2012Orlsmoothextensions}.\;For $i\geq 0$,\;put $\overline{A}^i:=\hH^i_{\sm}(\overline H,E)\cong \bigwedge^i \homo_{\sm}(\overline H,E)$ and $A^i:=\hH^i_{\sm}(H,E)\cong \bigwedge^i \homo_{\sm}(H,E)$ for simplicity.\;More precisely, we compare the resulting spectral sequences for the trivial representation and for $\pi_{\beta,I}$:
\begin{equation}\label{fixedCenterHSspectral}
\begin{aligned}
{}^{1}E_2^{r,s}
&=\overline{A}^r\otimes_E\wedge^{s}\homo_{\sm}(\bZ(L),E)
\Rightarrow A^{r+s},\\
{}^{2}E_2^{r,s}
&=\sEo_{I}^r\otimes_E\wedge^{s}\homo_{\sm}(\bZ(L),E)
\Rightarrow \ssE_{I}^{r+s}.
\end{aligned}
\end{equation}
Tensoring extensions with $\pi_{\beta,I}$ defines a morphism from the first spectral sequence to the second.\;Since $\homo_{\sm}(\bZ(L),E)$ is one-dimensional,\;the two $E_2$-pages in (\ref{fixedCenterHSspectral}) have only the rows $s=0,1$, and the maps between corresponding rows are induced by $\overline{\kappa}_I^i:\overline{A}^i\longrightarrow\sEo_{I}^i$.\;By Lemma \ref{degspectral},\;we have the following commutative diagram of long exact sequences:
\[\xymatrix{
	\cdots\ar[r] & \overline{A}^i \ar[r] \ar[d]^{\overline{\kappa}_I^i}& {A}^i \ar[r]  \ar[d]^{\simeq} &  \overline{A}^{i-1}\otimes_E\homo_{\sm}(\bZ(L),E)\ar[r] \ar[d]^{\overline{\kappa}_I^{i-1}\otimes\mathrm{Id}}& \overline{A}^{i+1}\ar[d]^{\overline{\kappa}_I^{i+1}} \ar[r] &  {A}^{i+1} \ar[r] \ar[d]^{\simeq} &\cdots \\
		\cdots\ar[r] & \sEo_{I}^i \ar[r] & \ssE_{I}^i \ar[r] &  \sEo_{I}^{i-1}\otimes_E\homo_{\sm}(\bZ(L),E) \ar[r] & \sEo_{I}^{i+1} \ar[r] &  \ssE_{I}^{i+1} \ar[r] & \cdots}\]
It is easy to see  that $\overline{A}^i\rightarrow A^i$ is injective (since $\homo_{\sm}(\overline H,E)\rightarrow \homo_{\sm}(H,E)$ is injective) and $A ^i\rightarrow \overline{A}^{i-1}\otimes_E\homo_{\sm}(\bZ(L),E)$ is surjective (using the fact that $\homo_{\sm}(H,E)\rightarrow\homo_{\sm}(\bZ(L),E)$ is surjective and the Leibniz rule).\;In particular,\;we obtain a short exact sequence for every $i\geq 1$:
\begin{equation*}
	\begin{aligned}
		0&\longrightarrow\overline{A}^i 
		\longrightarrow{A}^i 
		\longrightarrow\overline{A}^{i-1}\otimes_E\homo_{\sm}(\bZ(L),E)\longrightarrow0.
	\end{aligned}
\end{equation*}
Since $A^i\xrightarrow{\sim }\ssE_{I}^i $ for all $i$,\;we deduce that $\overline{A}^i\rightarrow \sEo_{I}^i $ are injective for all $i$.\;Starting with degree $0$ and comparing dimensions inductively, a diagram chase shows that $\overline{\kappa}_I^i$ is an isomorphism for every $i$.\;In particular,\;we get the following  short exact sequences for all $i\geq 1$:
\begin{equation*}
\begin{aligned}
0&\longrightarrow\sEo_{I}^i
\longrightarrow\ssE_{I}^i
\longrightarrow\sEo_{I}^{i-1}\otimes_E\homo_{\sm}(\bZ(L),E)\longrightarrow0,
\end{aligned}
\end{equation*}
We complete the proof.\;
\end{proof}


\begin{lem}\label{SmoothExt1}Let $I,J\subseteq \Delta_\beta$,\;we have 
\begin{equation}\label{smoothExt1}
\ext^{i,\infty}_{G}\big(i_I(\bm\beta),i_J(\bm\beta)\big)=\left\{
\begin{array}{ll}
\ssE_{J}^i=\bigwedge^i \homo_{\sm}(\bL_{\beta,J}(L),E), & \hbox{$J\subseteq I$;} \\
0, & \hbox{otherwise,}
\end{array}
\right.
\end{equation}
and
\begin{equation}\label{smoothExtC1}
\ext^{i,\infty}_{G,Z}\big(i_I(\bm\beta),i_J(\bm\beta)\big)=\left\{
\begin{array}{ll}
\sEo_{J}^i=\bigwedge^i \homo_{\sm}(\overline{\bL}_{\beta,J}(L),E), & \hbox{$J\subseteq I$;} \\
0, & \hbox{otherwise.}
\end{array}
\right.
\end{equation}
\end{lem}
\begin{proof}We prove (\ref{smoothExtC1}) first, and (\ref{smoothExt1}) follows by the same strategy.\;By \cite[Lemma 8]{2012Orlsmoothextensions}, we have for all $i\geq 0$ the isomorphisms
\begin{equation*}
\begin{aligned}
\ext^{i,\infty}_{G,Z}\big(i_I(\bm\beta),i_J(\bm\beta)\big)
=\ext^{i,\infty}_{\bL_{\beta,J}(L),Z}\big(J^{\beta}_{\Delta_\beta,J}(i_I(\bm\beta)),\pi_{\beta,J}\big).\\
\end{aligned}\end{equation*}
Recall the Bruhat filtration $\cF^{\bullet}$ on $i_I(\bm\beta)$ (see (\ref{filext})) and the vanishing result of (\ref{smoothvanishdevissage}).\;If $J\nsubseteq I$, we deduce by d\'{e}vissage 
(see the graded pieces (\ref{rgradfilext}) for all $h\in\BZ_{\geq 0}$) that $\ext^{i,\infty}_{G,Z}\big(i_I(\bm\beta),i_J(\bm\beta)\big)=0$ for all $i\geq 0$.\;If $J \subseteq I$, we have for all $i\geq 0$ the isomorphisms
\begin{equation*}
\begin{aligned}
\ext^{i,\infty}_{\bL_{\beta,J}(L),Z}\big(J^{\beta}_{\Delta_{\beta},J}(i_I(\bm\beta)),\pi_{\beta,J}\big)\cong\; &\ext^{i,\infty}_{\bL_{\beta,J}(L),Z}\big(J^{\beta}_{\Delta_{\beta},J}(\cF^0(i_I(\bm\beta)),\pi_{\beta,J}\big)=\ext^{i,\infty}_{\bL_{\beta,J}(L),Z}\big(\pi_{\beta,J},\pi_{\beta,J}\big).
\end{aligned}
\end{equation*}
This completes the proof.\;
\end{proof}
\begin{cor}\label{SmoothExt2}
Let $I,J\subseteq \Delta_\beta$,\;we have 
	\begin{equation}\label{smoothAnalyticExt2}
		\ext^{i,\infty}_{G}\big(v^{\infty}_I(\bm\beta),i_J(\bm\beta)\big)=\left\{
		\begin{array}{ll}
			\ssE_J^{i-|\Delta_{\beta}\backslash I|}, & \hbox{$I\cup J=\Delta_{\beta}$;} \\
			0, & \hbox{otherwise.}
		\end{array}
		\right.
	\end{equation}
	\begin{equation}\label{smoothAnalyticExtC2}
		\ext^{i,\infty}_{G,Z}\big(v^{\infty}_I(\bm\beta),i_J(\bm\beta)\big)=\left\{
		\begin{array}{ll}
			\sEo_J^{i-|\Delta_{\beta}\backslash I|}, & \hbox{$I\cup J=\Delta_{\beta}$;} \\
			0, & \hbox{otherwise.}
		\end{array}
		\right.
	\end{equation}
\end{cor}
The above discussion shows that the extension group $\ext^{i,\infty}_{G}\big(v^{\infty}_I(\bm\beta),v^{\infty}_J(\bm\beta)\big)$ can be identified with the extension group $\ext^{i,\infty}_{G}\big(v^{G}_{S^{\beta}(I)},v^{G}_{S^{\beta}(J)}\big)$ in \cite[Theorem 1]{2012Orlsmoothextensions},\;thus recovering \cite[Corollary 18]{2012Orlsmoothextensions}:
\begin{pro}\label{smoothExt4}Let $I,J\subseteq \Delta_\beta$,\;set $i=|I\cup J|-|I\cap J|$,\;then 
	$\ext^{\bullet,\infty}_{G}\big(v^{\infty}_I(\bm\beta),v^{\infty}_J(\bm\beta)\big)=E[-i]\oplus E[-i-1]$.\;
\end{pro}


\begin{thebibliography}{99}

\bibitem{bergdall2018adjunction}
John Bergdall, Przemys{\l}aw Chojecki,
"An adjunction formula for the Emerton-Jacquet functor."
\textit{Israel J.\;Math.} 223 (2018), no.\;1, 1–52.

\bibitem{berger2008equations}
\textit{L.~Berger},
\'{E}quations diff\'{e}rentielles $p$-adiques et $(\varphi,N)$-modules filtr\'{e}s ,
Ast\'{e}risque No. \textbf{319} (2008), 13-38. 

\bibitem{bernstein1977induced1}
I.\;N.\;Bernstein,  A.\;V.\;Zelevinsky, 
"{Induced representations of reductive $p$-adic groups I}."
\textit{Ann.\;Sci.\;École Norm.\;Sup.} (4) 10 (1977), no.\;4, 441–472.


\bibitem{borel2000continuous}
A.\;Borel,  N.\;Wallach, 
"{ Continuous cohomology, discrete subgroups, and representations of reductive groups.\;(English summary)	Second edition.}."
\textit{Mathematical Surveys and Monographs, 67.\;American Mathematical Society,} Providence, RI, 2000.\;xviii+260 pp.\;

\bibitem{2014Vers}
C.\;Breuil, 
"{Vers le socle localement analytique pour $\GLN_n$ II}."
\textit{Math.\;Ann.} 361 (2015), no.\;3-4, 741–785.


\bibitem{breuil2016socle}
C.\;Breuil, 
"{Socle localement analytique I}."
\textit{\em Ann.\;Inst.\;Fourier (Grenoble)} 66 (2016), no.\;2, 633–685.


\bibitem{breuil2019ext1}
C.\;Breuil,
"{${\ext}^1$ localement analytique et compatibilit{\'e}
local-global}."
\textit{Amer.\;J.\;Math.} 141 (2019), no.\;3, 611–703.


\bibitem{HigherLinvariantsGL3Qp}
C.\;Breuil, Y.\;Ding, 
"{Higher $L$-invariants for $\GLN_3(\bQ_p)$ and local-global
compatibility (English summary).}."
\textit{Camb.\;J.\;Math.} 8 (2020), no.\;4, 775–951.

	\bibitem{Ding2021}
C. Breuil,  Y. Ding, 
\newblock Bernstein eigenvarieties.
\newblock {\em  Peking Mathematical Journal } (2024) 7:471–642.


\bibitem{breuil2015ordinary}
C.\;Breuil, F.\;Herzig, 
"{Ordinary representations of $G(\bQ_p)$ and fundamental algebraic representations}."
\textit{Duke Math.\;J.} 164 (2015), no.\;7, 1271–1352.


\bibitem{Finiteslope}
C.\;Breuil, F.\;Herzig, 
"{Towards the finite slope part for GLn }."
\textit{Int. Math. Res. Not. IMRN} 2020, no. 24, 10495-10552.


	
\bibitem{breuil2007first}
\textit{C.~Breuil}  and \textit{P.~Schneider},
First steps towards $p$-adic langlands functoriality , 
J. Reine Angew. Math. \textbf{610} (2007), 149-180.


\bibitem{bushnell2016admissible}
C.\;J.\;Bushnell, P.\;C.\;Kutzko, 
"{The admissible dual of {${\GLN}(N)$} via compact open
subgroups}."
\textit{In Annals of Mathematics Studies,} 129.\;Princeton University Press, Princeton, NJ, 1993.\;xii+313 pp.\;






\bibitem{RWcarter}
R.\;W.\;Carter, 
"{Finite groups of Lie type.
Conjugacy classes and complex characters.\;Reprint of the 1985 original.\;}."
\textit{Wiley Classics Library.\;A Wiley-Interscience Publication.\;John Wiley \& Sons, Ltd., Chichester,} 1993.\;xii+544 pp.\;

\bibitem{casselman1975introduction}
W.\;A.\;Casselman, 
"{Introduction to the theory of admissible representations of
$p$-adic reductive groups}."
\textit{Sonderforschungsber.} 72 der Univ., 1975.



\bibitem{CDN20b}
P. Colmez, G. Dospinescu, W. Niziol, 
"{Cohomology of p-adic Stein spaces}."
\textit{Invent. Math.} 219, 2020, no. 3, 873-985.

\bibitem{CN20a}
P. Colmez, W. Niziol, 
"{On p-adic comparison theorems for rigid analytic varieties, I}."
\textit{M¨unster J. of Math.} 13 (2020), 445–507 








\bibitem{Dat}
J. F. Dat , 
"{Th\'{e}orie de Lubin-Tate non-ab\'{e}lienne et repr\'{e}sentations elliptiques.}."
\textit{Invent. math.} 169, 75–152 (2007).


\bibitem{2015Ding}
Y.\;Ding,
"{$\mathcal{L}$-invariants, partially de Rham families, and local-global compatibility}"
\textit{Ann.\;Inst.\;Fourier (Grenoble)} 67 (2017), no.\;4, 1457–1519.


\bibitem{2019DINGSimple}
Y.\;Ding, 
"{Simple $\cL$-invariants for $\GLN_{n}$}."
\textit{Trans.\;Amer.\;Math.\;Soc.} 372 (2019), no.\;11, 7993–8042.



\bibitem{EmpLfun}
M.\;Emerton, 
"$p$-adic $L$-functions and unitary completions of representations of p-adic reductive groups."
\textit{Duke Math. J. } 130 (2005), no. 2, 353–392. 

\bibitem{Em1}
M.\;Emerton, 
"On the interpolation of systems of eigenvalues attached to automorphic Hecke eigenforms."
\textit{Invent.\;Math.} 164 (2006), no.\;1, 1–84.


\bibitem{emerton2007jacquet}
M.\;Emerton, 
"Jacquet modules of locally analytic representations of $p$-adic
reductive groups II.\;The relation to parabolic induction."
\textit{J.\;Institut Math.\;Jussieu,} 2007.

\bibitem{Emerton2007summary}
M.\;Emerton, 
"Locally analytic representation theory of p-adic reductive groups: a
summary of some recent developments."
\textit{L-functions and Galois representations} 407–437,
London Math.\;Soc.\;Lecture Note Ser., 320, Cambridge Univ.\;Press, Cambridge, 2007.


\bibitem{emerton2017locally}
M.\;Emerton, 
"{Locally analytic vectors in representations of locally $p
$-adic analytic groups}."
\textit{Mem.\;Amer.\;Math.\;Soc.} 248 (2017), no.\;1175, iv+158 pp.\;

\bibitem{1987Catehighetweight}
T.\;J.\;Enright,  B.\;Shelton, 
"{Categories of highest weight modules: applications to classical Hermitian symmetric pairs}."
\textit{Mem.\;Amer.\;Math.\;Soc.} 67 (1987), no.\;367, iv+94 pp.


\bibitem{gehrmann2019automorphic}
L.\;Gehrmann, 
"{Automorphic $L$-invariants for reductive groups}."
\textit{J.\;Reine Angew.\;Math.} 779 (2021), 57–103.


\bibitem{goldfeld2011automorphic}
D.\;Goldfeld, J.\;Hundley,  
"{Automorphic representations and $L$-functions for the general linear group.\;Volume II, With exercises and a preface by Xander Faber.}."
\textit{Cambridge Studies in Advanced Mathematics, 130.} Cambridge University Press, Cambridge, 2011.\;xx+188 pp.\;

\bibitem{GKHKcoh}
E. Grosse-Kl\"{o}nne, 
"Frobenius and monodromy operators in rigid analysis, and Drinfeld’s symmetric space."
\textit{J. Algebraic Geom.} 14 (2005), 391-437.

\bibitem{2022parabolivinv}
Y.\;He,
"{Parabolic $\sL$-invariants}."
\textit{preprint 2022, {http://arxiv.org/abs/2211.10847}}.

\bibitem{humphreysBGG}
J.\;E.\;Humphreys, 
"{Representations of Semisimple Lie Algebras in the BGG Category
$\cO$}, volume~94."
\textit{Graduate Studies in Mathematics, 94.} American Mathematical Society, Providence, RI, 2008.\;xvi+289 pp.\;

\bibitem{kohlhaase2011cohomology}
J.\;Kohlhaase,
"The cohomology of locally analytic representations".
\textit{J.\;Reine Angew.\;Math.} 651 (2011), 187–240.


\bibitem{Drimodularvar} 
G.\;Laumon,  
"{Cohomology of {D}rinfeld modular varieties.\;{P}art {I}}."
\textit{Cambridge University Press, Cambridge,} 1996.\;xiv+344 pp.\;

\bibitem{GeneralizedVerma}
J.\;Lepowsky, 
"Generalized Verma modules, the Cartan-Helgason theorem, and the
Harish-Chandra homomorphism."
\textit{J.\;Algebra} 49 (1977), no.\;2, 470–495.

\bibitem{HKfirstcovering}
Z. Li, 
"{Hyodo-Kato Cohomology of the First Covering of Drinfeld Spaces}."




\bibitem{2012Orlsmoothextensions}
S.\;Orlik, 
"On extensions of generalized Steinberg representations".
\textit{J.\;Algebra} 293 (2005), no.\;2, 611–630.


\bibitem{orlik2014jordan}
S.\;Orlik,   B.\;Schraen,
"The Jordan-Hölder series of the locally analytic Steinberg representation".
\textit{Doc.\;Math.} 19 (2014), 647–671.


\bibitem{orlik2015jordan}
S.\;Orlik,
"On Jordan-Hölder series of some locally analytic representations."
\textit{J.\;Amer.\;Math.\;Soc.} 28 (2015), no.\;1, 99–157.


\bibitem{PyvovarovBSgeneric}
A.\;Pyvovarov,
"On the Breuil--Schneider conjecture: generic case."
\textit{Algebra Number Theory} 15 (2021), no.\;2, 309--339.


\bibitem{Dilogarithm}
Z.\;Qian, 
"{Dilogarithm and higher {$\mathcal{L}$}-invariants for {${\GLN}_3({\BQ}_p)$}}."
\textit{Represent.\;Theory} 25 (2021), 344–411.


\bibitem{wholeLINV}
Z.\;Qian,
"On $\ext^{\bullet}$ between locally analytic generalized Steinberg with applications.\;"
\textit{arXiv:\;2512.24279v1, preprint.} 




\bibitem{schneider1991cohomology}
P.\;Schneider, U.\;Stuhler,  
"The cohomology of $p$-adic symmetric spaces."
\textit{Invent.\;Math.} 105 (1991), no.\;1, 47–122.
11F85 (14G20)

\bibitem{schneider2002banach}
P.\;Schneider, J.\;Teitelbaum,  
"Banach space representations and Iwasawa theory".
\textit{Israel J.\;Math.} 127 (2002), 359–380.
22E50

\bibitem{schneider2002locally}
P.\;Schneider,  J.\;Teitelbaum,  
"{Locally analytic distributions and $p$-adic representation theory,
with applications to $\GLN_2$}."
\textit{J.\;Amer.\;Math.\;Soc.} 15 (2002), no.\;2, 443–468.
11S80 (22E50)

\bibitem{schneider2003algebras}
P.\;Schneider,  J.\; Teitelbaum,
"Algebras of $p$-adic distributions and admissible representations."
\textit{Invent.\;Math.} 153 (2003), no.\;1, 145–196.


\bibitem{schneider2005duality}
P.\;Schneider, J.\;Teitelbaum,  
"Duality for admissible locally analytic representations".
\textit{Represent.\;Theory} 9 (2005), 297–326.


\bibitem{schneider200finite}
P.\;Schneider, J.\;Teitelbaum, D.\;Prasad, 
"{$U(\fg)$-finite locally analytic representations.\;(English summary), With an appendix by Dipendra Prasad }."
\textit{Represent.\;Theory} 5 (2001), 111–128.


\bibitem{Errschraen2011GL3}
B.\;Schraen,
"Erratum to the paper ``Repr{\'e}sentations localement analytiques de
$\mathrm{GL}_3(\bQ_p)$''."
\textit{"https://www.math.u-psud.fr/~schrean/Erratum\_GL3.pdf"}.

\bibitem{schraen2011GL3}
B.\; Schraen, 
"{Repr{\'e}sentations localement analytiques de
$\mathrm{GL}_3(\bQ_p)$.\;(French.\;English, French summary) [Locally analytic representations of $\mathrm{GL}_3(\bQ_p)$]}."
\textit{Ann.\;Sci.\;{\'E}c.\;Norm.\;Sup{\'e}r.}(4)44 (2011), no.\;1, 43–145.

\bibitem{sorensen2013proof}
C.\;M.\;Sorensen,  
"A proof of the Breuil-Schneider conjecture in the indecomposable case."
\textit{Ann.\;of Math.} (2) 177 (2013), no.\;1, 367–382.

\bibitem{vigneras1996book}
M.\;Vign\'{e}ras,  
"{Repr\'{e}sentations {$l$}-modulaires d'un groupe r\'{e}ductif {$p$}-adique avec {$l\neq p$}}."
\textit{Progress in Mathematics, 137.\;Birkhäuser Boston, Inc., Boston, MA,} 1996.\;xviii and 233 pp.\;

\bibitem{vigneras1997extensions}
Marie-France Vignéras, 
"{Extensions between irreducible representations of p-adic $\GLN(n)$}."
\textit{Pacific J.\;Math.} 1997, Special Issue, 349–357.\;

\bibitem{av1980induced2}
A.\;V.\;Zelevinsky,
"{Induced representations of reductive p-adic groups.\;II.\;On irreducible representations of $GL(n)$}."
\textit{Ann.\;Sci.\;École Norm.\;Sup.} (4) 13 (1980), no.\;2, 165–210.

\end{thebibliography}
\end{document}